\newif\iffinal
\finalfalse	

\documentclass[letterpaper,11pt,reqno]{amsart} 
\RequirePackage[utf8]{inputenc}
\usepackage[portrait,margin=3cm]{geometry}
\usepackage{mathrsfs,soul}
\usepackage{hyperref}
\usepackage[foot]{amsaddr}
\usepackage{amssymb,amsthm,amsfonts,amsbsy,latexsym,dsfont}
\usepackage[textsize=small]{todonotes}       
\usepackage{graphicx}
\usepackage{upref,setspace}

\usepackage{amsmath} 
\usepackage{amsthm}
\usepackage{amssymb}

\usetikzlibrary{arrows, quotes}
\usetikzlibrary{positioning,calc}

\definecolor{seagreen}{RGB}{28,190,160}%
\definecolor{bluegreen}{rgb}{0.00000,0.61700,0.79100}%
\definecolor{greenlight}{rgb}{0.278431, 0.76470,0}
\definecolor{backgroundColor}{rgb}{0,0,0}%

\usepackage{enumerate}

\numberwithin{equation}{section}
\numberwithin{figure}{section}
\numberwithin{table}{section}

\usepackage{hyperref,bbm}
\definecolor{MyDarkBlue}{rgb}{0,0.08,0.50}
\definecolor{BrickRed}{rgb}{0.65,0.08,0}

\hypersetup{
colorlinks=true,       
    linkcolor=MyDarkBlue,          
    citecolor=BrickRed,        
    filecolor=red,      
    urlcolor=cyan           
}

\newtheorem{Lemma}{Lemma}[section]

\newtheorem{Proposition}[Lemma]{Proposition}

\newtheorem{Theorem}[Lemma]{Theorem}

\newtheorem{Remark}[Lemma]{Remark}

\newtheorem{Corollary}[Lemma]{Corollary}
\newtheorem{Definition}[Lemma]{Definition}
\newtheorem{Condition}[Lemma]{Condition}

\newtheorem{remark}[Lemma]{Remark}

\newtheorem{Statement}[Lemma]{Statement}

\usepackage{dsfont}

\newcommand{\dopp}[1]{\mathbb{#1}}
\newcommand{\sub}[1]{\boldsymbol{#1}}
\newcommand{\grosso}{\displaystyle}

\newcommand{\R}{\dopp{R}}

\newcommand{\N}{\dopp{N}}
\newcommand{\I}{\mathds{1}}
\newcommand{\pr}{\dopp{P}}
\newcommand{\E}{\dopp{E}}

\newcommand{\CMnd}{\mathrm{CM}_{n}}
\newcommand{\CMndd}{\mathrm{CM}_{n}(\sub{d})}
\newcommand{\EE}{\mathcal{E}}

\newcommand{\eqn}[1]{\begin{equation} #1 \end{equation}}
\newcommand{\eqan}[1]{\begin{align} #1 \end{align}}
\newcommand{\nn}{\nonumber}
\newcommand{\PA}[1]{\mathrm{PA}_{#1}}
\newcommand{\PAmd}[1]{\mathrm{PA}_{#1}(m,\delta)}
\newcommand{\PAONE}[1]{\mathrm{PA}_{#1}(1,\delta/m)}
\newcommand{\core}{\mathrm{Core}}
\newcommand{\cM}{\mathcal{M}}
\newcommand{\cN}{\mathcal{N}}

\newcommand{\cG}{\mathcal{G}}
\newcommand{\vep}{\varepsilon}
\newcommand{\indic}[1]{\mathds{1}_{\{#1\}}}

\newcommand{\dmin}{d_{\mathrm{min}}}
\newcommand{\dfwd}{d_{\mathrm{fwd}}}
\newcommand{\cdist}{c_{\mathrm{dist}}}
\newcommand{\dist}{\mathrm{dist}}
\newcommand{\Var}{\mathrm{Var}}
\newcommand{\sss}{\scriptscriptstyle}
\newcommand{\e}{{\mathrm e}}

\newcommand{\cC}{\mathcal{C}}
\newcommand{\cW}{\mathcal{W}}

\newcommand{\blue}{} 
\newcommand{\shortversion}[1]{} 
\newcommand{\longversion}[1]{{#1}}
\begin{document}

\title[Diameter in ultra-small scale-free random graphs]{Diameter in ultra-small scale-free random graphs
\longversion{\!\!:\linebreak Extended version}}

\date{}

\author[Caravenna]{Francesco Caravenna$^1$}
\address{$^1$Dipartimento di Matematica e Applicazioni,
 Universit\`a degli Studi di Milano-Bicocca,
 via Cozzi 55, 20125 Milano, Italy}
\author[Garavaglia]{Alessandro Garavaglia$^2$}
\address{$^2$Department of Mathematics and
    Computer Science, Eindhoven University of Technology, 5600 MB Eindhoven, The Netherlands}
\author[van der Hofstad]{Remco van der Hofstad$^2$}
\email{francesco.caravenna@unimib.it, a.garavaglia@tue.nl, rhofstad@win.tue.nl}

\begin{abstract}
It is well known that many random graphs with infinite variance degrees are ultrasmall. More precisely, for configuration models and preferential attachment models where the proportion of vertices of degree at least $k$ is approximately $k^{-(\tau-1)}$ with $\tau\in(2,3)$, typical distances between pairs of vertices in a graph of size $n$ are asymptotic to $\frac{2\log\log n}{|\log(\tau-2)|}$ and $\frac{4\log\log n}{|\log(\tau-2)|}$, respectively. In this paper, we investigate the behavior of the diameter in such models. We show that the diameter is of order $\log\log n$ precisely when the minimal forward degree $\dfwd$ of vertices is at least $2$. We identify the exact constant, which equals that of the typical distances plus $2/\log \dfwd$. Interestingly, the proof for both models follows identical steps, even though the models are quite different in nature.
\end{abstract}

\maketitle

\section{Introduction and results}
\label{sec-int}

In this paper, we study the diameter of two different random graph models: the {\itshape configuration model} 
and the {\itshape preferential attachment model}, when these two models have a power-law degree distribution with exponent $\tau\in(2,3)$, so that
the degrees have finite mean but infinite variance. In this first section, we give a brief introduction to the models, stating the main technical conditions required as well as the two main results proved in the paper.

{\blue Throughout the paper, we write ``with high probability''
to mean ``with probability $1- o(1)$ as $n\to\infty$, or as $t \to \infty$'',
where $n$ and $t$ denote the number of vertices in the configuration model
and in the preferential attachment model, respectively.}

\subsection{Configuration model and main result}
\label{sec:model-form}
The configuration model $\CMnd$ is a random graph with vertex set $[n]:=\{1,2,\ldots,n\}$
and with prescribed degrees.  Let $\sub{d} = (d_1,d_2,\ldots, d_n)$ be a  given {\it degree sequence}, i.e., a sequence of $n$ positive integers with total degree
 	\eqn{
	\label{cmnd-elln}
    	\ell_n = \sum_{i\in [n]} d_i,
    	}
assumed to be even. The configuration model (CM) on $n$ vertices with degree sequence $\sub{d}$ is constructed as follows:
Start with $n$ vertices and $d_i$ half-edges adjacent to vertex $i \in [n]$. Randomly choose pairs of half-edges and match the chosen pairs together to form edges.  Although self-loops may occur, these become rare as $n\to\infty$ {\blue (see e.g.\ \cite[Theorem 2.16]{Boll01}, \cite{Jans06b})}.
We denote the resulting multi-graph on $[n]$ by $\CMnd$, with corresponding edge set $\EE_n$. We often omit the dependence on the degree sequence $\sub{d}$, and write $\CMnd$ for $\CMndd$.
\medskip

\paragraph{\bf Regularity of vertex degrees.}
Let us now describe our regularity assumptions.
For each $n\in\N$ we have a degree sequence
$\sub{d}^{\sss(n)} = (d^{\sss(n)}_1, \ldots, d^{\sss(n)}_n)$. To lighten notation, we omit the superscript $(n)$ and write $\sub{d}$ instead of $\sub{d}^{\sss(n)}$ or $(\sub{d}^{\sss(n)})_{n\in\N}$ and  $d_i$ instead of $d^{\sss(n)}_i$.
Let $(p_k)_{k\in \N}$ be a probability mass function on $\N$. We introduce the {\itshape empirical degree distribution} of the graph as
\begin{equation}
\label{emp-deg-dis}
      p_k^{\sss(n)} = \frac{1}{n}\sum_{i\in[n]}\I_{\{d_i=k\}}.
\end{equation}
We can define now the {\itshape degree regularity conditions}:

\begin{Condition}[Degree regularity conditions]
\label{drc}

Let $\CMnd$ be a configuration model, then we say that $\sub{d}$ satisfies the degrees regularity conditions \eqref{it:a}, \eqref{it:b}, 
  with respect to $(p_k)_{k\in \N}$ if:
  \begin{enumerate}
\renewcommand{\theenumi}{\alph{enumi}}%
\renewcommand{\labelenumi}{{\rm(\theenumi)}}%
   \item\label{it:a} for every $k\in \N$, as $n\rightarrow\infty$
	  \begin{equation}
			\label{drc-a}
	    p_k^{\sss(n)}\longrightarrow p_k.
	  \end{equation}
  \item\label{it:b} $\sum_k kp_k<\infty$, and as $n\rightarrow\infty$
	  \begin{equation}
			\label{drc-b}
	    \sum_{k\in\N}kp_k^{\sss(n)}\longrightarrow \sum_{k\in \N}kp_k.
	  \end{equation}
	\end{enumerate}
	As notation, we write that $\sub{d}$ satisfies the d.r.c. \eqref{it:a}, \eqref{it:b}.
\end{Condition}
\medskip

Let $F_{\sub{d},n}$ be the distribution function of $(p_k^{\sss(n)})_{k\in\N}$, that is, for $k\in\N$,
	\eqn{
	\label{cmnd-discrete-distrib-function}
  	F_{\sub{d},n}(k) = \frac{1}{n}\sum_{i\in[n]}\I_{\{d_i\leq k\}}.
	}
We suppose that $\sub{d}$ satisfies
the d.r.c.\ \eqref{it:a} and \eqref{it:b} with respect to some probability mass function $(p_k)_{k\in\N}$, corresponding to a distribution function $F$.

\begin{Condition}[Polynomial distribution condition]
\label{pol-dis-con}
We say that $\sub{d}$ satisfies the polynomial distribution condition with exponent $\tau\in(2,3)$ if for all $\delta>0$ 
there exist $\alpha=\alpha(\delta)>\frac{1}{2}$, $c_1(\delta)>0$ and $c_2(\delta)>0$ such that,
for every $n\in\N$, the lower bound
	\begin{equation}
	\label{pol-con-left}
	1-F_{\sub{d},n}(x)\geq c_1x^{-(\tau-1+\delta)}
	\end{equation}
holds for all $x\leq n^\alpha$, and the upper bound
\begin{equation}
	\label{pol-con-right}
  	1-F_{\sub{d},n}(x)\leq c_2x^{-(\tau-1-\delta)}
	\end{equation}
holds for all $x\geq 1$.
\end{Condition}
There are two examples that explain Condition \ref{pol-dis-con}. Consider the case of i.i.d.\ degrees with $\pr\left(D_i>x\right)= cx^{-(\tau-1)}$, then the degree sequence satisfies Condition \ref{pol-dis-con} a.s. A second case is when the number of vertices of degree $k$ is $n_k = \lceil nF(k)\rceil - \lceil nF(k-1)\rceil$, and $1-F(x) = cx^{-(\tau-1)}$. Condition \ref{pol-dis-con} allows for more flexible degree sequences than just these examples. 

{\blue 
If we fix $\beta < \min\{\alpha, \frac{1}{\tau-1+\delta}\}$, 
the lower bound \eqref{pol-con-left} ensures that the number of vertices of degree higher than $x = n^{\beta}$ is at least $n^{1-\beta(\tau-1+\delta)}$, which diverges as a positive power of $n$.
If we take $\beta > \frac{1}{2}$,  these vertices with high probability form a complete graph.
This will be essential for proving our main results. The precise value of $\beta$ is irrelevant in the sequel of this paper.}

\medskip
For an asymptotic degree distribution with asymptotic probability mass function $(p_k)_{k\in\N}$, we say that
	\begin{equation}
	\label{eq:dmin0}
	\dmin = \min\left\{k\in\N \colon p_k>0\right\}
	\end{equation}
is the minimal degree of the probability given by $(p_k)_{k\in\N}$.
With these technical requests, we can state the main result for the configuration model:

\medskip
\begin{Theorem}[Diameter of $\CMnd$ for $\tau\in(2,3)$] 
\label{main-cmnd}
Let $\sub{d}$ be a sequence satisfying Condition \ref{drc} with asymptotic degree distribution $(p_k)_k$ with $\dmin\geq3$. Suppose that $\sub{d}$ satisfies Condition \ref{pol-dis-con} with $\tau\in(2,3)$ and $d_i\geq \dmin$ for all $i\in[n]$. Then
	\begin{equation}
	\label{diam-cmnd}
  	\frac{\mathrm{diam}(\CMnd)}{\log\log n}
  	{\xrightarrow[\,n\to\infty\,]{\pr} \ }
  	\frac{2}{\log(\dmin-1)}+\frac{2}{|\log(\tau-2)|},
	\end{equation}
where ${\xrightarrow[\,n\to\infty\,]{\pr} \ }$ denotes convergence in probability as $n\rightarrow \infty$.
\end{Theorem}
\medskip

In fact, the result turns out to be false when $p_1+p_2>0$, as shown by {\blue  Fernholz and Ramachandran}  \cite{FeRama} (see also \cite{vdHHGZD1}), since then there are long strings of vertices with low degrees that are of logarithmic length.

\subsection{Preferential attachment model and main result}
\label{subsect-prefatt-intro}
The configuration model presented in the previous section is a {\itshape static model}, because the size $n\in\N$ of the graph was fixed. 

The preferential attachment model instead is a {\itshape dynamic model}, because, in this model, vertices are added sequentially with a number of edges connected to them. These
edges are attached to a receiving vertex with a probability proportional to the degree of the receiving vertex at that time plus a constant, thus favoring vertices with high degrees.

The idea of the preferential attachment model is simple, and we start by defining it informally. We start with a single vertex with a self loop, which is the graph at time $1$. At every time $t\geq2$, we add a vertex to the graph. 
This new vertex has an edge incident to it, and we attach this edge to a random vertex already present in the graph, with probability proportional to the degree of the receiving 
vertex plus a constant {\blue $\delta$}, which means that vertices with large degrees are favored.
Clearly, at each time $t$ we have a graph of size $t$ with exactly $t$ edges.

We can modify this model by changing the number of edges incident to each new vertex we add. 
If we start at time $1$ with a single vertex with $m\in\N$ self loops, and at every time
$t\geq 2$ we add a single vertex with $m$ edges, then at time $t$ we have 
a graph of size $t$ but with $mt$ edges, that we call $\PAmd{t}$. When no confusion can arise, we omit the arguments $(m,\delta)$ and abbreviate $\PA{t}=\PAmd{t}$.
We now give the explicit expression for the attachment probabilities.

\medskip
\begin{Definition}[Preferential attachment model]
\label{def-attprob}
	Fix $m \in \N$, $\delta \in (-m,\infty)$.
	Denote  by $\{t\stackrel{j}{\rightarrow} v\}$
	the event that the $j$-th edge of vertex $t \in \N$ is attached to vertex $v \in [t]$
	(for $1 \le j \le m$).
	The preferential attachment model with parameters
	$(m,\delta)$ is defined by the attachment probabilities
	\begin{equation}
	\label{prefatt-attprob}
		\pr\left(\left.t\stackrel{j}{\rightarrow} v \,\right| \PA{t,j-1}\right) = 
		\begin{cases}
			\grosso \frac{D_{t,j-1}(v) + 1+j\delta/m}
			{c_{t,j}} & \text{ for }v=t, \\
			\rule{0pt}{2em}\grosso \frac{D_{t,j-1}(v)+\delta}
			{c_{t,j}} & \text{ for }v < t,
		\end{cases}
	\end{equation}
where $\PA{t,j-1}$ is the graph after the first $j-1$ 
edges of vertex $t$ have been attached, and correspondingly
$D_{t,j-1}(v)$ is the degree of vertex $v$. 
The normalizing constant $c_{t,j}$ in \eqref{prefatt-attprob} is
\begin{equation} \label{eq:normali}
	c_{t,j} := \left[m(t-1)+(j-1)\right]\left(2+\delta/m\right)+1+\delta/m \,.
\end{equation}

\end{Definition}

We refer to Section~\ref{sec:altprefatt} for more details and explanations
on the construction of the model (in particular, for the reason
behind the factor $j\delta/m$ in the first line of \eqref{prefatt-attprob}).

\medskip
Consider, as in \eqref{emp-deg-dis}, the empirical degree distribution of the graph, which we denote by $P_k(t)$, where in this case the degrees are random variables. It is known from the literature (\cite{BRST01}, \cite{vdH1}) that, for every $k\geq m$, as $t\rightarrow\infty$,
	\eqn{
	\label{prefatt-empdegdist}
	P_k(t){\xrightarrow[\sss\,t\to\infty\,]{\sss \pr} \ } p_k,
	}
where $p_k\sim ck^{-\tau}$, and $\tau = 3+\delta/m$. We focus on the case $\delta\in(-m,0)$,
so that $\PA{t}$ has a power-law degree sequence with power-law exponent $\tau\in(2,3)$.

For the preferential attachment model, our main result is the following:

\begin{Theorem}[Diameter of the preferential attachment model]
\label{main-prefatt}
Let $(\PA{t})_{t\geq 1}$ be a preferential attachment model with $m\geq 2$ and $\delta\in(-m,0)$. Then
	\begin{equation}
	\label{diam-prefatt}
  	\frac{\mathrm{diam}(\PA{t})}{\log\log t}{\xrightarrow[\,t\to\infty\,]{\pr}\ }\frac{2}{\log m}+\frac{4}{|\log(\tau-2)|},
	\end{equation}
where $\tau = 3+\delta/m\in(2,3)$.
\end{Theorem}

In the proof of Theorem \ref{main-prefatt} we are also able to identify the typical distances in $\PA{t}$:

\begin{Theorem}[Typical distance in the preferential attachment model]
\label{typ-dis-prefatt}
Let $V_1^t$ and $V_2^t$ be two independent uniform random vertices in $[t]$. 
Denote the distance between $V_1^t$ and $V_2^t$ in $\PA{t}$ by $H_t$. Then
	\eqn{
	\label{typ-dist-formula}
	\frac{H_t}{\log\log t}
	{\xrightarrow[\sss\,t\to\infty\,]{\sss \pr} \ }
		\frac{4}{|\log(\tau-2)|}.
	}
\end{Theorem}

\medskip

Theorems \ref{main-prefatt}--\ref{typ-dis-prefatt} prove \cite[Conjecture 1.8]{vdHHGZD1}.

\subsection{Structure of the paper and heuristics}
The proofs of our main results on the diameter in 
Theorems~\ref{main-cmnd} and~\ref{main-prefatt} have a surprisingly similar structure. 
We present a detailed outline in Section~\ref{sec-structure} below, where we split 
the proof into a lower bound (Section~\ref{sub-lower}) and an upper bound 
(Section~\ref{sub-upper}) on the diameter. 
Each of these bounds is then divided into 3 statements, that hold for each model.
In Sections~\ref{lowerproofs-cmnd} and~\ref{lowerproofs-prefatt} we prove the lower bound 
for the configuration model and for the preferential attachment model, respectively, 
while in Sections~\ref{upperproofs-cmnd} and~\ref{upperproofs-prefatt} we prove 
the corresponding upper bounds. 
\longversion{Finally, Appendixes~\ref{section-appendix} and~\ref{app-B}
collects full details of some proofs of technical results that are 
minor modifications of proofs in the literature.} 
\shortversion{In \cite[Appendix]{CarGarHof16ext}, some proofs of technical results 
that are minor modifications of proofs in the literature are presented in detail.}

Even though the configuration and preferential attachment models are quite
different in nature, they are \emph{locally} similar, because
for both models the attachment probabilities are roughly proportional to the degrees.
The core of our proof is a combination of \emph{conditioning arguments} 
(which are particuarly
subtle for the preferential attachment model), that allow to
combine local estimates in order to derive
bounds on \emph{global} quantities, such as the diamter.

Let us give a heuristic explanation of the proof 
(see Figure \ref{fig-proof} for a graphical representation).
For a quantititative outline, we refer to Section~\ref{sec-structure}.
We write $\PA{n}$ instead of $\PA{t}$ to simplify 
the exposition, and denote by $d_{\rm fwd}$ the minimal \emph{forward degree}, 
that is $d_{\rm fwd}=\dmin-1$ for the configuration model and $d_{\rm fwd}=m$ for 
the preferential attachment model.

\begin{itemize}
\item For the {\em lower bound} on the diamter, 
we prove that there are so-called {\em minimally-connected} vertices. 
These vertices are quite special, in that their neighborhoods 
up to distance $k_n^- \approx \log\log{n}/\log{d_{\rm fwd}}$
are \emph{trees
with the minimal possible degree}, given by $\dfwd+1$.
This explains the first term in the right hand sides of \eqref{diam-cmnd}
and \eqref{diam-prefatt}.

Pairs of
minimally-connected vertices are good candidates for achieving the maximal 
possible distance, i.e., the diameter. In fact,
the boundaries of their tree-like neighborhoods turn out to be at distance equal 
to the \emph{typical distance} $2\bar{k}_n$ between vertices in the graph,
that is $2\bar{k}_n 
\approx 2\cdist \log\log{n}/|\log(\tau-2)|$, where
$\cdist = 1$ for the configuration model and 
$\cdist = 2$ for the preferential attachment model.
This leads to the second term in the right hand sides of \eqref{diam-cmnd}
and \eqref{diam-prefatt}.

In the proof, we split the possible paths between the boundaries of two minimally 
connected vertices into bad paths, which are too short, 
and typical paths, which have the right number of edges in them, 
and then show that the contribution due to bad paths vanishes. 
The degrees along the path determine whether a path is bad or typical. 

The strategy for the lower bound is depicted in the bottom part of Figure \ref{fig-proof}.

\item For the {\em upper bound} on the diamter, 
we perform a lazy-exploration from every vertex in the graph 
and realize that the neighborhood up to a distance $k_n^+$, which is roughly the same as 
$k_n^-$, contains at least \emph{as many vertices as the tree-like neighborhood of a 
minimally-connected vertex}. All possible other vertices in this neighborhood are ignored. 

We then show that the vertices at the boundary of these lazy neighborhoods are 
with high probability \emph{quickly} connected to the core, that is
by a path of $h_n=o(\log\log{n})$ steps. 
By {\em core} we mean the set of all vertices with large degrees,
which is known to be highly connected, with a diameter 
close to $2\bar{k}_n$, similar to the typical distances
(see \cite{vdHHGZD1} for the configuration model and \cite{DSvdH} 
for the preferential attachment model). 

The proof strategy for the upper bound is depicted in the top part of Figure \ref{fig-proof}.
\end{itemize}

\begin{figure}[t]
\begin{center}
\scalebox{0.4}{
\begin{tikzpicture}[
    node/.style={circle, draw=black, fill=white, line width=1.5 pt, minimum size=0.5cm,
        text centered, anchor=center, text=black},
    root/.style={circle, draw=black, fill=red!40, line width=1.9 pt, minimum size=1cm,
        text centered, anchor=north, text=black},
    invisible/.style={circle, draw=white, opacity = 0, line width=0pt, minimum size=0.1cm,
        text centered, anchor=north, text=black},
    whitenode/.style={circle, draw=white,  line width=1 pt, minimum size=0.5cm,
        text centered, anchor=north, text=black},
        	edge/.style = {line width = 2pt},
    block1/.style={
      rectangle,
      draw=greenlight,
      align=center,
      rounded corners,
      line width = 3pt,
    },	
    level distance=0.5cm, growth parent anchor=south]

SEPARATION
\draw[edge]  [color = white, sloped, pos = 0.5] (-12cm,-4.5cm) to ["\color{black}\huge Lower bound"] (-12cm,0cm);
\draw[edge]  [color = white,sloped] (-12cm,4.5cm) to ["\color{black}\huge Upper bound"] (-12cm,8.4cm);
\draw[edge]  [color = white,sloped ] (12cm,0cm) to ["\color{black}\huge Lower bound"] (12cm,-4.5cm);
\draw[edge]  [color = white,sloped] (12cm,8.4cm) to ["\color{black}\huge Upper bound"](12cm,4.5cm) ;

\draw[edge]  [line width=3pt, |-|] (-3.5cm,8.4cm) to ["\huge $2\bar{k}_n$"] (3.5cm,8.4cm)  ;

\draw[edge]  [line width=3pt, |-|, above ] (-5.5cm,8.4cm) to ["\huge $h_n$"] (-3.4cm,8.4cm);
\draw[edge]  [line width=3pt, |-| ] (3.4cm,8.4cm) to ["\huge $h_n$"] (5.5cm,8.4cm);

\draw[edge]  [line width=3pt, |-| ] (5.4cm,8.4cm) to ["\huge $k^+_n$"] (12cm,8.4cm);
\draw[edge]  [line width=3pt, |-| ] (-12cm,8.4cm) to ["\huge $k^+_n$"] (-5.4cm,8.4cm);
\draw[edge]  [line width=3pt, |-|] (5.5cm,-8.4cm) to ["\huge $2\bar{k}_n$"] (-5.5cm,-8.4cm)  ;
\draw[edge]  [line width=3pt, |-| ] (12cm,-8.4cm) to ["\huge $k^-_n$"] (5.4cm,-8.4cm);
\draw[edge]  [line width=3pt, |-| ] (-5.4cm,-8.4cm) to ["\huge $k^-_n$"] (-12cm,-8.4cm);

\node(1001) [node, minimum size = 1.5cm, line width = 2pt,anchor = center] at (-12cm,0cm) {\Huge v};
\node (1002)[node, minimum size = 1.5cm, line width = 2pt, anchor = center] at (12cm,0cm) {\Huge w};
	\draw  [line width=3pt, color = black] (1001) to (1002) ;

\node (1) [node,minimum size=7cm, draw = bluegreen, fill = bluegreen!40, line width = 5pt] at (0cm,4cm) {\Huge Core};

\draw[edge]  [line width=3pt, color = magenta, bend right] (-4.8cm,-7cm) to  (4.8cm,-3cm)  ;
\draw[edge]  [line width=3pt, color = magenta, bend left] (-4.8cm,-4.6cm) to  (4.8cm,-4cm)  ;
\draw[edge]  [line width=3pt, color = magenta, bend left] (-4.8cm,-1cm) to  (4.8cm,-7cm)  ;
\draw[edge]  [line width=3pt, color = magenta, bend left] (4.8cm,-1cm) to  (-4.8cm,-5.6cm)  ;
\draw[edge]  [line width=3pt, color = magenta, bend left] (4.8cm,-6.5cm) to  (-4.8cm,-2cm)  ;

\node (999) [node, minimum size = 5.3cm, draw = magenta, fill = magenta!40, line width = 5pt] at (0cm,-4cm) {\huge typical paths};

\node (2) [node,fill = seagreen] at (-5.5cm,7cm) {};
\node (3) [node, fill = seagreen] at (-5.5cm,6cm) {};	
\node (4) [node, fill = seagreen] at (-5.5cm,5cm) {};
\node (5) [node,fill = seagreen] at (-5.5cm,4cm) {};
\node (6) [node,fill = seagreen] at (-5.5cm,3cm) {};
\node (7) [node, fill = seagreen] at (-5.5cm,2cm) {};
\node (8) [node, fill = seagreen] at (-5.5cm,1cm) {};
\node (9) [node] at (-5.5cm,-7cm) {};
\node (10) [node] at (-5.5cm,-6cm) {};	
\node (11) [node] at (-5.5cm,-5cm) {};
\node (12) [node] at (-5.5cm,-4cm) {};
\node (13) [node] at (-5.5cm,-3cm) {};
\node (14) [node] at (-5.5cm,-2cm) {};
\node (15) [node] at (-5.5cm,-1cm) {};

\draw[block1]  ($(2.north west)+(-0.5cm,0.5cm)$)  rectangle ($(9.south east)+(0.5cm,-0.5cm)$);	

\node (16) [node, fill = seagreen] at (5.5cm,7cm) {};
\node (17) [node,fill = seagreen] at (5.5cm,6cm) {};	
\node (18) [node, fill = seagreen] at (5.5cm,5cm) {};
\node (19) [node, fill = seagreen] at (5.5cm,4cm) {};
\node (20) [node, fill = seagreen] at (5.5cm,3cm) {};
\node (21) [node, fill = seagreen] at (5.5cm,2cm) {};
\node (22) [node, fill = seagreen] at (5.5cm,1cm) {};
\node (23) [node] at (5.5cm,-7cm) {};
\node (24) [node] at (5.5cm,-6cm) {};	
\node (25) [node] at (5.5cm,-5cm) {};
\node (26) [node] at (5.5cm,-4cm) {};
\node (27) [node] at (5.5cm,-3cm) {};
\node (28) [node] at (5.5cm,-2cm) {};
\node (29) [node] at (5.5cm,-1cm) {};

\draw[block1]  ($(16.north west)+(-0.5cm,0.5cm)$)  rectangle ($(23.south east)+(0.5cm,-0.5cm)$);

\draw [color = seagreen, line width = 3pt] (2) to (1);
\draw [color = seagreen, line width = 3pt] (5) to (1);
\draw [color = seagreen, line width = 3pt] (6) to (1);
\draw [color = seagreen, line width = 3pt] (17) to (1);
\draw [color = seagreen, line width = 3pt] (21) to (1);

\draw [color = seagreen, line width = 3pt] (3) to (1);
\draw [color = seagreen, line width = 3pt] (4) to (1);
\draw [color = seagreen, line width = 3pt] (7) to (1);
\draw [color = seagreen, line width = 3pt] (8) to (1);

\draw [color = seagreen, line width = 3pt] (16) to (1);
\draw [color = seagreen, line width = 3pt] (18) to (1);
\draw [color = seagreen, line width = 3pt] (19) to (1);
\draw [color = seagreen, line width = 3pt] (20) to (1);
\draw [color = seagreen, line width = 3pt] (22) to (1);


\node(1003) [node] at (-10.5cm,5cm) {};
\node(1004) [node] at (-10.5cm,2cm) {};
\draw[line width = 1pt] (1001) to (1003);
\draw[line width = 1pt] (1001) to (1004);
\draw[line width = 1pt, dashed] (1004) to (-10cm,3cm);
\node(1007) [node] at (-9cm,6cm) {};
\node(1008) [node] at (-9cm,4.5cm) {};
\draw[line width = 1pt] (1003) to (1007);
\draw[line width = 1pt] (1003) to (1008);
\draw[line width = 1pt, dashed] (1008) to (-8.5cm,5.5cm);
\draw[line width = 1pt, dashed] (1008) to (-8.5cm,3.5cm);
\node(1009) [node] at (-9cm,2.5cm) {};
\node(1010) [node] at (-9cm,1cm) {};
\draw[line width = 1pt, dashed] (1010) to (-8.5cm,2cm);
\draw[line width = 1pt] (1004) to (1009);
\draw[line width = 1pt] (1004) to (1010);

\node(1011) [node] at (-7.5cm,0.6cm) {};
\node(1012) [node] at (-7.5cm,1.4cm) {};
\draw[line width = 1pt] (1011) to (1010);
\draw[line width = 1pt] (1012) to (1010);
\node(1013) [node] at (-7.5cm,2.9cm) {};
\node(1014) [node] at (-7.5cm,2.1cm) {};
\draw[line width = 1pt, dashed] (1013) to (-7cm,3.9cm);
\draw[line width = 1pt] (1013) to (1009);
\draw[line width = 1pt] (1014) to (1009);
\node(1015) [node] at (-7.5cm,6.4cm) {};
\node(1016) [node] at (-7.5cm,5.6cm) {};
\draw[line width = 1pt, dashed] (1015) to (-7cm,7.4cm);
\draw[line width = 1pt] (1015) to (1007);
\draw[line width = 1pt] (1016) to (1007);
\node(1017) [node] at (-7.5cm,4.9cm) {};
\node(1018) [node] at (-7.5cm,4.1cm) {};
\draw[line width = 1pt] (1017) to (1008);
\draw[line width = 1pt] (1018) to (1008);

\draw[line width = 1pt ] (1011) to (-6.2cm,0.4cm);
\draw[line width = 1pt ] (1011) to (-6.2cm,0.8cm);
\draw[line width = 1pt ] (1012) to (-6.2cm,1.6cm);
\draw[line width = 1pt ] (1012) to (-6.2cm,1.2cm);
\draw[line width = 1pt ] (1013) to (-6.2cm,2.7cm);
\draw[line width = 1pt ] (1013) to (-6.2cm,3.1cm);
\draw[line width = 1pt ] (1014) to (-6.2cm,1.9cm);
\draw[line width = 1pt ] (1014) to (-6.2cm,2.3cm);	
\draw[line width = 1pt ] (1015) to (-6.2cm,6.2cm);
\draw[line width = 1pt ] (1015) to (-6.2cm,6.6cm);	
\draw[line width = 1pt ] (1016) to (-6.2cm,5.8cm);
\draw[line width = 1pt ] (1016) to (-6.2cm,5.4cm);	
\draw[line width = 1pt ] (1017) to (-6.2cm,4.7cm);
\draw[line width = 1pt ] (1017) to (-6.2cm,5.1cm);	
\draw[line width = 1pt ] (1018) to (-6.2cm,3.9cm);
\draw[line width = 1pt ] (1018) to (-6.2cm,4.3cm);	

\node(2003) [node] at (-10.5cm,-5cm) {};
\node(2004) [node] at (-10.5cm,-2cm) {};
\draw[line width = 1pt] (1001) to (2003);
\draw[line width = 1pt] (1001) to (2004);
\node(2007) [node] at (-9cm,-6cm) {};
\node(2008) [node] at (-9cm,-4.5cm) {};
\draw[line width = 1pt] (2003) to (2007);
\draw[line width = 1pt] (2003) to (2008);
\node(2009) [node] at (-9cm,-2.5cm) {};
\node(2010) [node] at (-9cm,-1cm) {};
\draw[line width = 1pt] (2004) to (2009);
\draw[line width = 1pt] (2004) to (2010);

\node(2011) [node] at (-7.5cm,-0.6cm) {};
\node(2012) [node] at (-7.5cm,-1.4cm) {};
\draw[line width = 1pt] (2011) to (2010);
\draw[line width = 1pt] (2012) to (2010);
\node(2013) [node] at (-7.5cm,-2.9cm) {};
\node(2014) [node] at (-7.5cm,-2.1cm) {};
\draw[line width = 1pt] (2013) to (2009);
\draw[line width = 1pt] (2014) to (2009);
\node(2015) [node] at (-7.5cm,-6.4cm) {};
\node(2016) [node] at (-7.5cm,-5.6cm) {};
\draw[line width = 1pt] (2015) to (2007);
\draw[line width = 1pt] (2016) to (2007);
\node(2017) [node] at (-7.5cm,-4.9cm) {};
\node(2018) [node] at (-7.5cm,-4.1cm) {};
\draw[line width = 1pt] (2017) to (2008);
\draw[line width = 1pt] (2018) to (2008);

\draw[line width = 1pt ] (2011) to (-6.2cm,-0.4cm);
\draw[line width = 1pt ] (2011) to (-6.2cm,-0.8cm);
\draw[line width = 1pt ] (2012) to (-6.2cm,-1.6cm);
\draw[line width = 1pt ] (2012) to (-6.2cm,-1.2cm);
\draw[line width = 1pt ] (2013) to (-6.2cm,-2.7cm);
\draw[line width = 1pt ] (2013) to (-6.2cm,-3.1cm);
\draw[line width = 1pt ] (2014) to (-6.2cm,-1.9cm);
\draw[line width = 1pt ] (2014) to (-6.2cm,-2.3cm);	
\draw[line width = 1pt ] (2015) to (-6.2cm,-6.2cm);
\draw[line width = 1pt ] (2015) to (-6.2cm,-6.6cm);	
\draw[line width = 1pt ] (2016) to (-6.2cm,-5.8cm);
\draw[line width = 1pt ] (2016) to (-6.2cm,-5.4cm);	
\draw[line width = 1pt ] (2017) to (-6.2cm,-4.7cm);
\draw[line width = 1pt ] (2017) to (-6.2cm,-5.1cm);	
\draw[line width = 1pt ] (2018) to (-6.2cm,-3.9cm);
\draw[line width = 1pt ] (2018) to (-6.2cm,-4.3cm);

\node(3003) [node] at (10.5cm,5cm) {};
\node(3004) [node] at (10.5cm,2cm) {};
\draw[line width = 1pt] (1002) to (3003);
\draw[line width = 1pt] (1002) to (3004);
\draw[line width = 1pt, dashed] (3004) to (10cm,3cm);
\draw[line width = 1pt, dashed] (3004) to (10cm,1cm);
\node(3007) [node] at (9cm,6cm) {};
\node(3008) [node] at (9cm,4.5cm) {};
\draw[line width = 1pt, dashed] (3007) to (8.5cm,7cm);
\draw[line width = 1pt, dashed] (3008) to (8.5cm,3.5cm);
\draw[line width = 1pt] (3003) to (3007);
\draw[line width = 1pt] (3003) to (3008);
\node(3009) [node] at (9cm,2.5cm) {};
\node(3010) [node] at (9cm,1cm) {};
\draw[line width = 1pt, dashed] (3010) to (8.5cm,2cm);
\draw[line width = 1pt] (3004) to (3009);
\draw[line width = 1pt] (3004) to (3010);

\node(3011) [node] at (7.5cm,0.6cm) {};
\node(3012) [node] at (7.5cm,1.4cm) {};
\draw[line width = 1pt] (3011) to (3010);
\draw[line width = 1pt] (3012) to (3010);
\node(3013) [node] at (7.5cm,2.9cm) {};
\node(3014) [node] at (7.5cm,2.1cm) {};
\draw[line width = 1pt] (3013) to (3009);
\draw[line width = 1pt] (3014) to (3009);
\node(3015) [node] at (7.5cm,6.4cm) {};
\node(3016) [node] at (7.5cm,5.6cm) {};
\draw[line width = 1pt] (3015) to (3007);
\draw[line width = 1pt] (3016) to (3007);
\node(3017) [node] at (7.5cm,4.9cm) {};
\node(3018) [node] at (7.5cm,4.1cm) {};
\draw[line width = 1pt] (3017) to (3008);
\draw[line width = 1pt] (3018) to (3008);

\draw[line width = 1pt ] (3011) to (6.2cm,0.4cm);
\draw[line width = 1pt ] (3011) to (6.2cm,0.8cm);
\draw[line width = 1pt ] (3012) to (6.2cm,1.6cm);
\draw[line width = 1pt ] (3012) to (6.2cm,1.2cm);
\draw[line width = 1pt ] (3013) to (6.2cm,2.7cm);
\draw[line width = 1pt ] (3013) to (6.2cm,3.1cm);
\draw[line width = 1pt ] (3014) to (6.2cm,1.9cm);
\draw[line width = 1pt ] (3014) to (6.2cm,2.3cm);	
\draw[line width = 1pt ] (3015) to (6.2cm,6.2cm);
\draw[line width = 1pt ] (3015) to (6.2cm,6.6cm);	
\draw[line width = 1pt ] (3016) to (6.2cm,5.8cm);
\draw[line width = 1pt ] (3016) to (6.2cm,5.4cm);	
\draw[line width = 1pt ] (3017) to (6.2cm,4.7cm);
\draw[line width = 1pt ] (3017) to (6.2cm,5.1cm);	
\draw[line width = 1pt ] (3018) to (6.2cm,3.9cm);
\draw[line width = 1pt ] (3018) to (6.2cm,4.3cm);

\node(4003) [node] at (10.5cm,-5cm) {};
\node(4004) [node] at (10.5cm,-2cm) {};
\draw[line width = 1pt] (1002) to (4003);
\draw[line width = 1pt] (1002) to (4004);
\node(4007) [node] at (9cm,-6cm) {};
\node(4008) [node] at (9cm,-4.5cm) {};
\draw[line width = 1pt] (4003) to (4007);
\draw[line width = 1pt] (4003) to (4008);
\node(4009) [node] at (9cm,-2.5cm) {};
\node(4010) [node] at (9cm,-1cm) {};
\draw[line width = 1pt] (4004) to (4009);
\draw[line width = 1pt] (4004) to (4010);

\node(4011) [node] at (7.5cm,-0.6cm) {};
\node(4012) [node] at (7.5cm,-1.4cm) {};
\draw[line width = 1pt] (4011) to (4010);
\draw[line width = 1pt] (4012) to (4010);
\node(4013) [node] at (7.5cm,-2.9cm) {};
\node(4014) [node] at (7.5cm,-2.1cm) {};
\draw[line width = 1pt] (4013) to (4009);
\draw[line width = 1pt] (4014) to (4009);
\node(4015) [node] at (7.5cm,-6.4cm) {};
\node(4016) [node] at (7.5cm,-5.6cm) {};
\draw[line width = 1pt] (4015) to (4007);
\draw[line width = 1pt] (4016) to (4007);
\node(4017) [node] at (7.5cm,-4.9cm) {};
\node(4018) [node] at (7.5cm,-4.1cm) {};
\draw[line width = 1pt] (4017) to (4008);
\draw[line width = 1pt] (4018) to (4008);

\draw[line width = 1pt ] (4011) to (6.2cm,-0.4cm);
\draw[line width = 1pt ] (4011) to (6.2cm,-0.8cm);
\draw[line width = 1pt ] (4012) to (6.2cm,-1.6cm);
\draw[line width = 1pt ] (4012) to (6.2cm,-1.2cm);
\draw[line width = 1pt ] (4013) to (6.2cm,-2.7cm);
\draw[line width = 1pt ] (4013) to (6.2cm,-3.1cm);
\draw[line width = 1pt ] (4014) to (6.2cm,-1.9cm);
\draw[line width = 1pt ] (4014) to (6.2cm,-2.3cm);	
\draw[line width = 1pt ] (4015) to (6.2cm,-6.2cm);
\draw[line width = 1pt ] (4015) to (6.2cm,-6.6cm);	
\draw[line width = 1pt ] (4016) to (6.2cm,-5.8cm);
\draw[line width = 1pt ] (4016) to (6.2cm,-5.4cm);	
\draw[line width = 1pt ] (4017) to (6.2cm,-4.7cm);
\draw[line width = 1pt ] (4017) to (6.2cm,-5.1cm);	
\draw[line width = 1pt ] (4018) to (6.2cm,-3.9cm);
\draw[line width = 1pt ] (4018) to (6.2cm,-4.3cm);	

\end{tikzpicture}
}
\end{center}
\caption{Structure of the proof in a picture}
\label{fig-proof}
\end{figure}
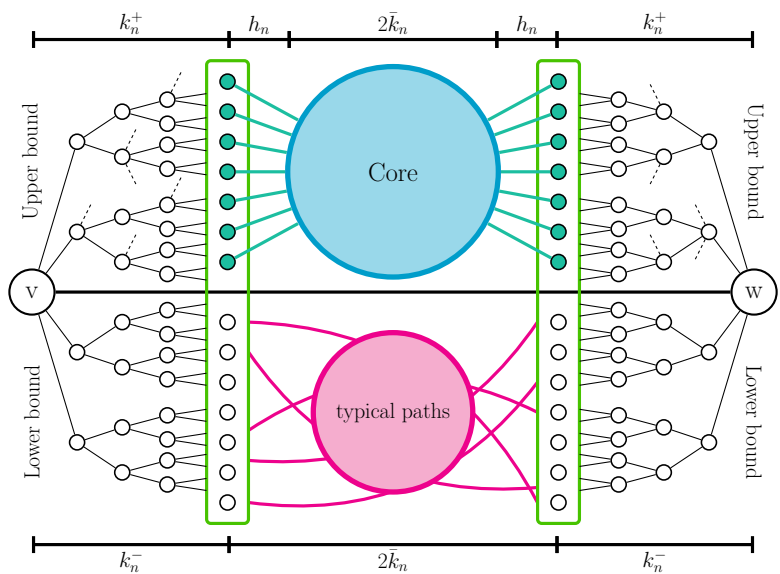

\subsection{Links to the literature and comments}
This paper studies the diameter in $\CMnd$ and $\PA{t}$ when the degree power-law exponent $\tau$ satisfies $\tau\in(2,3)$, which means the degrees have finite mean but
 infinite variance. Both in \eqref{diam-cmnd} and \eqref{diam-prefatt}, the explicit constant is the sum of two terms, one depending on $\tau$, and the other depending on the minimal forward degree {\blue (see \eqref{d_tree})}, which is $\dmin-1$ for $\CMnd$ and $m$ for $\PA{t}$. We remark that the term depending on $\tau$ is related to the typical distances, while the other is related to the periphery of the graph. 

There are several other works that have already studied typical distances and diameters of such models.
Van der Hofstad, Hooghiemstra and Znamenski \cite{vdHHGZD2} analyze typical distances in $\CMnd$  for $\tau\in(2,3)$, while 
Van der Hofstad, Hooghiemstra and Van Mieghem \cite{HofHooVan05a} study $\tau>3$. They prove that for $\tau\in(2,3)$ typical distances are of order $\log\log n$, while for $\tau>3$ is of order $\log n$, and it presents the explicit  constants of asymptotic growth. Van der Hofstad, Hooghiemstra and Znamensky \cite{vdHHGZD1} shows for $\tau>2$ and when vertices of degree $1$ or $2$ are present, that with high probability the diameter of $\CMnd$ is bounded from below by a constant times $\log n$, while when $\tau\in(2,3)$ and the minimal degree is $3$, the diameter is bounded from above by a constant times $\log\log n$.  {\blue In \cite{HofKom16}, Van der Hofstad and Komj\'athy investigate typical distances for configuration models and $\tau\in(2,3)$ in great generality, extending the results in \cite{vdHHGZD1} beyond the setting of i.i.d.\ degrees. Interestingly, they also investigate the effect of truncating the degrees at $n^{\beta_n}$ for values of $\beta_n\rightarrow 0$. It would be of ineterest to also extend our diameter results to this setting.}

{\blue We significantly improve upon the result in \cite{vdHHGZD1}  for $\tau\in(2,3)$. We do make use of similar ideas in our proof of the upper bound on the diameter. Indeed, we again define a core consisting of vertices with high degrees, and use the fact that the diameter of this core can be computed exactly (for a definition of the 
core, see \eqref{core_n-def}). The novelty of our current approach is that we quantify precisely how far the further vertex is from this core in the configuration model. It is a pair of such remote vertices that realizes the graph diameter.}

Fernholz and Ramachandran \cite{FeRama} prove that the diameter of $\CMnd$ is equal to an
 explicit constant times $\log n$ plus $o(\log n)$ when $\tau\in (2,3)$ but there are vertices of degree $1$ or $2$ present in the graph, by studying the longest paths in the configuration model that are not part of the 2-core (which is the part of the graph for which all vertices have degree at least 2). Since our minimal degree is at least 3, the 2-core is whp the entire graph, and thus this logarithmic phase vanishes. 
Dereich, M\"onch and M\"orters \cite{DSMCM} prove that typical distances in $\PA{t}$ are asymptotically equal to an explicit constant times $\log\log t$, using path counting techniques. We use such path counting techniques as well, now for the lower bound on the diameters. Van der Hofstad \cite{vdH2} studies the diameter of $\PA{t}$ when $m=1$, and proves that the diameter still has logarithmic growth. Dommers, van der Hofstad and Hooghiemstra \cite{DSvdH} prove an upper bound on the diameter of $\PA{t}$, but the explicit constant is not sharp. 

{\blue  Again, we significantly improve upon that result. Our proof uses ideas from \cite{DSvdH}, in the sense that we again rely on an appropriately chosen core for the preferential attachment model, but our upper bound now quantifies precisely how the further vertex is from this core, as for the configuration model, but now applied to the much harder preferential attachment model.}

$\CMnd$ and $\PA{t}$ are two different models, in the sense that $\CMnd$ is a static model while $\PA{t}$ is a dynamic model. It is
 interesting to notice that the main strategy to prove Theorems \ref{main-cmnd} and \ref{main-prefatt} is the same. In fact, all the statements formulated in Section \ref{sec-structure} are general and hold for both models. Also the explicit constants appearing in \eqref{diam-cmnd} and \eqref{diam-prefatt} are highly similar, which reflects the same structure of the proofs. 
The differences consist in a factor $2$ in the terms containing $\tau$ and in the presence of $\dmin-1$ and $m$ in the remaining term.
The factor 2 can be understood by noting that in $\CMnd$ pairs of vertices with high degree are likely to be at distance $1$, while in $\PA{t}$ they are at distance $2$. The difference in $\dmin-1$ and $m$ is due to the fact that $\dmin-1$ and $m$ play the same role in the two models, i.e., they are the minimal forward degree (or ``number of children'') of a vertex that is part of a tree contained in the graph. We refer to Section~\ref{sec-structure} for more details. 

{\blue
While the structures of the proofs for both models are identical, the details of the various steps are significantly different. Pairings in the configuration model are uniform, making explicit computations easy, even when already many edges have been paired. In the preferential attachment model, on the other hand, the edge statuses are highly dependent, so that we have ro rely on delicate conditoning arguments. These conditioning arguments are arguably the most significant innovation in this paper. This is formalized in the notion of factorizable events in Definition \ref{def:factorizable}.}

Typical distances and diameters have been studied for other random graphs models as well, showing $\log\log$ behavior. Bloznelis \cite{Blo} investigates the typical distance in power-law intersection random graphs, where such distance, conditioning on being finite, is of order $\log\log n$, while results on diameter are missing. Chung and Lu \cite{ChuLu1, ChuLu2} present results respectively for random graphs with given expected degrees and Erd\H{o}s and R\'enyi random graphs
{\blue $G(n,p)$, see also van den Esker, the last author and Hooghiemstra \cite{vdHEsker} for the rank-1 setting. The setting of the configuration model with finite-variance degrees is studied in \cite{FeRama}.}
In \cite{ChuLu1}, they prove that for the power-law regime with exponent $\tau\in(2,3)$, the diameter is $\Theta(\log n)$, while typical distances are of order $\log\log n$. This can be understood from the existence of a positive proportion of vertices with degree 2, creating long, but thin, paths. In \cite{ChuLu2}, the authors investigate the different behavior of the diameter according to the parameter $p$.

An interesting open problem is the study of fluctuations of the diameters in $\CMnd$ and $\PA{t}$ around the asymptotic mean, i.e., the study of the difference between the diameter of the graph and the asymptotic behavior (for these two models, the difference between the diameter and the right multiple of $\log\log n$). In \cite{vdHHGZD2}, the authors prove that in graphs with i.i.d.\ power-law degrees with $\tau\in(2,3)$, the difference {\blue $\Delta_n$} between 
the typical distance and the asymptotic behavior $2\log\log n/|\log(\tau-2)|$ does not 
converge {\blue in distribution}, even though it is 
{\blue \emph{tight} (i.e., for every $\epsilon > 0$ there is $M < \infty$
such that $\pr(|\Delta_n| \le M) > 1-\epsilon$ for all $n\in\N$). These results have been significantly improved in \cite{HofKom16}.}

{\blue In the literature results on fluctuations for the diameter of random graph models are rare. Bollob\'as in \cite{Bol81}, and, later,  Riordan and Wormald in  \cite{riordan10} give precise estimates on the diameter of the Erd\"os-Renyi random graph.} It would be of interest to investigate whether the diameter has tight fluctuations around $c\log\log{n}$ for the appropriate $c$.

\section{General structure of the proofs}
\label{sec-structure}

We split the proof of Theorems~\ref{main-cmnd} and~\ref{main-prefatt}
into a lower and an upper bound. Remarkably, the strategy is the same
for both models despite the inherent difference in the models. In this section we explain the strategy in detail, formulating
general statements that will be proved for each model separately in the next sections.

Throughout this section, we assume that the assumptions of Theorems~\ref{main-cmnd}
and~\ref{main-prefatt} are satisfied and,
to keep unified notation, we denote the size of 
the preferential attachment
model by $n\in\N$, instead of the more usual $t\in\N$.

{\blue 
Throughout the paper, we treat real numbers as integers when we consider graph distances. By this, we mean that we round real numbers to the closest integer. To keep the notation light and make the paper easier to read, we omit the rounding operation.}

\subsection{Lower bound}
\label{sub-lower}
We start with the structure of the proof of the lower bound
in \eqref{diam-cmnd} and \eqref{diam-prefatt}.
The key notion is that of a
\emph{minimally-$k$-connected vertex}, defined as a vertex whose 
$k$-neighborhood (i.e., the neighborhood up to distance $k$)
is {\blue essentially \emph{a regular tree with the smallest possible degree}}, 
equal to $\dmin$ for the configuration model 
and to $m+1$ for the preferential attachment model.  
{\blue Due to technical reasons, the precise definition of minimally-$k$-connected vertex is slightly different for the two models (see Definitions~\ref{def:minicm} and~\ref{prefatt-MKCdef}).}

Henceforth we fix $\varepsilon>0$ and define, for $n\in\N$,
	\eqn{
	\label{kstar-tree}
	{\blue k^{-}_n} = (1-\varepsilon)\frac{\log\log n}{\log(\dfwd)} ,
	}
where $\dfwd$ denotes the \emph{forward degree}, or ``number of children'':
	\eqn{
	\label{d_tree}
	\dfwd = \begin{cases}
		\dmin-1 & \text{for } \ \CMnd; \\
		m & \text{for } \ \PA{n}.
	\end{cases}
	}
Our first goal is to prove that
the number of minimally-$k^{-}_n$-connected vertices is large enough, as formulated in the following statement:

\medskip
\begin{Statement}[Moments of $M_{k^{-}_n}$]
\label{stat-MKC}
Denote by $M_{k^{-}_n}$ the number of minimally-$k^{-}_n$-connected vertices in the graph 
(either $\CMnd$ or $\PA{n}$). Then, as $n \to \infty$,
	\eqn{
	\label{MKC-moments}
	\E\left[M_{k^{-}_n}\right]\rightarrow \infty, \qquad
	\Var\left(M_{k^{-}_n}\right) = o\left(\E\left[M_{k^{-}_n}\right]^2\right),
	}
where $\Var(X) := \E[X^2] - \E[X]^2$ denotes the variance of the random variable $X$.
\end{Statement}

{\blue The proof for the preferential attachment model makes use of conditioning arguments. Indeed, we describe as much information as necessary to be able to bound probabilities that vertices are minimally-$k$ connected. Particularly in the variance estimate, these arguments are quite delicate, and crucial for our purposes.}

The bounds in \eqref{MKC-moments} show that $M_{k^{-}_n} {\xrightarrow[]{\ \pr\ } \ } \infty$
as $n\to\infty$. This will imply that \emph{there is a pair of minimally-$k^{-}_n$-connected vertices 
with disjoint $k^{-}_n$-neighborhoods},\footnote{\blue A justification for this fact
is provided by the following Statement~\ref{stat-distance-boundaries}
(the randomly chosen vertices $W_1^n$ and $W_2^n$ have
disjoint $k_n^-$-neighborhoods, because $\tilde H_n > 0$ with high probability). 
For a more direct justification, see Remark~\ref{rem:disjcm} for the configuration model
and Remark~\ref{prefatt-disjneigh-remark} for the preferential attachment model.
} 
hence
the diameter of the graph is at least $2 k^{-}_n$, which 
explains the first term in \eqref{diam-cmnd} and \eqref{diam-prefatt}. Our next aim is to prove that these minimally connected trees are typically at distance $2\cdist\log\log n/|\log(\tau-2)|$, where $\cdist=1$ for the configuration model and $\cdist=2$ for the preferential attachment model.
\medskip

For this, let us now define
	\eqn{
	\label{general-k_n}
	\bar{k}_n = (1-\varepsilon)\frac{\cdist\log\log n}{|\log(\tau-2)|},
	}
where 
	\eqn{
	\label{c_dist}
	\cdist = \begin{cases}
	1 & \text{for } \ \CMnd;\\
	2 & \text{for } \ \PA{n}.
	\end{cases}
	}
The difference in the definition of $\cdist$ is due to fact that in $\CMnd$ vertices with high degree are likely at distance $1$, while in $\PA{n}$ are at distance $2$. We explain the origin of this effect in more detail in the proofs.
\medskip

It turns out that the distance between the $k^{-}_n$-neighborhoods  of two minimally-$k^{-}_n$-connected vertices is at least $2\bar{k}_n$. More precisely, we have the following statement:

\begin{Statement}[Distance between neighborhoods]
\label{stat-distance-boundaries}
Let $W^n_1$ and $W^n_2$ be two random vertices
chosen independently and uniformly among the minimally-$k^{-}_n$-connected ones.
Denoting by $\tilde{H}_n$ the distance between the $k^{-}_n$-neighborhoods of $W^n_1$ and $W^n_2$,
we have $\tilde{H}_n \geq 2\bar{k}_n$ with high probability.
\end{Statement}

It follows immediately from Statement~\ref{stat-distance-boundaries} that the
distance between the vertices $W^n_1$ and $W^n_2$ is at least $2 k^{-}_n + 2 \bar{k}_n$,
with high probability. This proves the lower bound in \eqref{diam-cmnd} and \eqref{diam-prefatt}.

\smallskip

It is known from the literature that $2 \bar{k}_n$, {\blue see} \eqref{general-k_n}, represents the \emph{typical distance} 
between two vertices chosen independently and uniformly in the graph.  In order to prove Statement~\ref{stat-distance-boundaries}, we collapse the $k^{-}_n$-neighborhoods of $W^n_1$ and $W^n_2$ into single vertices
and show that their distance is roughly equal to the typical distance $2\bar{k}_n$. 
This is a delicate point, because the collapsed vertices have a relatively large degree and thus {\em could} be closer than the typical distance. The crucial point why they are not closer is that the degree of the boundary only grows polylogarithmically. The required justification is provided by the next statement:

\begin{Statement}[Bound on distances]
\label{stat-distance}
Let us introduce the set
	\begin{equation}
	\label{eq:set}
	V_n := \begin{cases}
	\big\{ v \in [n]\colon \ d_v \le \log n \big\} & \text{for } \ \CMnd;\\
	\rule{0pt}{1.5em}\big\{ v \in [n]\colon \ v \ge \frac{n}{(\log n)^2}\} 
	& \text{for } \ \PA{n}.
	\end{cases}
	\end{equation}
Then, denoting the distance in the graph of size $n$ by $\dist_n$,
	\eqn{
	\label{stat-dist-formula}
	\max_{a,b \in V_n}
	\pr\left(\dist_n(a,b)\leq 2\bar{k}_n\right) =
	O \left(\frac{1}{(\log n)^2} \right).
	}
\end{Statement}

The proof of Statement~\ref{stat-distance} is based on {\em path counting techniques.}
These are different for the two models, but the idea is the same:
We split the possible paths 
between the vertices $a$ and $b$ into two sets, called {\itshape good paths} and {\itshape bad paths}.
Here {\itshape good} means that the degrees of vertices 
along the path increase, but {\itshape not too much}. We then separately and directly estimate
the contribution of each set. The details are described in the proof.

\subsection{Upper bound}
\label{sub-upper}
We now describe the structure of the proof for the upper bound,
which is based on two key concepts:
the {\itshape core of the graph} and the $k$-{\itshape exploration graph} of a vertex.

We start by introducing some notation. First of all, fix a constant $\sigma \in (1/(3-\tau), \infty)$. 
We define $\core_n$ as the set of vertices in the graph of size $n$ with degree larger than $(\log n)^\sigma$.
More precisely, {\blue 
denoting by $D_t(v) = D_{t,m}(v)$ the degree of vertex $v$ in the preferential attachment model
after time $t$, i.e.\ in the graph $\PA{t}$
(see the discussion after \eqref{prefatt-attprob}),}
we let
	\eqn{
	\label{core_n-def}
	\core_n:=\begin{cases}
	\{v\in[n]\colon d_v\geq  (\log n)^\sigma\}& \text{for } \ \CMnd;\\
	\{v\in[n]\colon {\blue D_{n/2}(v)}\geq  (\log n)^\sigma\} & \text{for } \ \PA{n}.
	\end{cases}
	}
The fact that we evaluate the degrees at time $n/2$ for the PAM is quite crucial in the proof of Statement \ref{stat-core} below. 
In Section \ref{upperproofs-prefatt}, we also give bounds on $D_v(n)$ for $v\in \core_n$, as well as for $v\not\in \core_n$, that show that the degrees cannot grow too much between time $n/2$ and $n$. The first statement, that we formulate for completeness, upper bounds the diameter of $\core_n$ and is already 
known from the literature for both models:

\begin{Statement}
\label{stat-core}
Define $\cdist$ as in \eqref{c_dist}. Then, for every $\vep>0$, {\blue with high probability}
	\eqn{
	\label{core_diam}
	\frac{\mathrm{diam}(\core_n)}{\log\log n}
	\leq 	(1+\vep)\frac{2\cdist}{|\log(\tau-2)|}.
	}
\end{Statement}
\medskip

Statement \ref{stat-core} for $\CMnd$ is \cite[Proposition 3.1]{vdHHGZD1}, for $\PA{n}$ it is \cite[Theorem 3.1]{DSvdH}.
\medskip

Next we bound the distance between a vertex and $\core_n$.  We define the \emph{$k$-exploration graph} of 
a vertex $v$ as a suitable subgraph of its $k$-neighborhood, built as follows:
We consider the usual exploration process starting at $v$, but
instead of exploring all the edges incident to a vertex, we only explore a 
\emph{fixed} number of them, namely $\dfwd$ defined in \eqref{d_tree}. 
(The choice of which edges to explore is a natural one, and it will be explained in more detail in the proofs.)

We stress that it is possible to
explore vertices that have already been explored,
leading to what we call a {\itshape collision}. If there are no
collisions, then the $k$-exploration graph is a tree. 
In presence of collisions, the $k$-exploration graph is not a tree, and it is clear that 
every collision reduces the number of vertices in the $k$-exploration graph.

Henceforth we fix $\varepsilon > 0$ and, in analogy with \eqref{kstar-tree}, we define, for $n\in\N$,
	\eqn{
	\label{kstar-exp}
	{\blue k^{+}_n} = \left(1+\varepsilon\right) \frac{\log\log n}{\log(\dfwd)} .
	}
Our second statement for the upper bound shows that the $k^{+}_n$-exploration graph of \emph{every}
vertex in the graph either intersects $\core_n$, or it has a bounded number of collisions:

\begin{Statement}[Bound on collisions]
\label{stat-collision}
There is a constant $c < \infty$ such that, with high probability, 
the $k^{+}_n$-exploration graph of \emph{every} vertex in the graph 
has at most $c$ collisions before hitting $\core_n$. 
As a consequence, for some constant $s>0$, the $k^{+}_n$-exploration graph
of \emph{every} vertex in the graph either intersects $\core_n$, or its boundary
has cardinality at least
	\eqn{
	\label{exp-boundary}
	s(\dfwd)^{k^{+}_n} = (\log n)^{1+\varepsilon+o(1)}.
	}
\end{Statement}
\medskip

With a bounded number of collisions, the $k^{+}_n$-exploration graph is not far from being a tree, which explains 
the lower bound \eqref{exp-boundary} on the cardinality of its boundary. Having enough vertices on its boundary, the $k^{+}_n$-exploration is likely to be connected to $\core_n$ \emph{fast}, which for our purpose means in $o(\log\log n)$ steps. This is the content of our last statement:

\begin{Statement}
\label{stat-triplelog}
There {\blue are} constants $B,C < \infty$ such that, with high probability, 
the $k^{+}_n$-exploration graph of every vertex in the graph is at distance at most 
$h_n = \lceil B\log\log\log n+C\rceil$ from $\core_n$.
\end{Statement}

{\blue
The proof for this is novel. For example, for the configuration model, we grow the $k^{+}_n+h_n$ neighborhood of a vertex, and then show that there are so many half-edges at its boundary that it is very likely to connect immediately to the core. The proof for the preferential attachment model is slightly different, but the conclusion is the same. This shows that the vertex is indeed at most
at distance $k^{+}_n+h_n$ away from the core.}

In conclusion, with high probability, the diameter of the graph is at most
	$$
	(k^{+}_n+ h_n) +\mathrm{diam}(\core_n)+ (k^{+}_n +h_n),
	$$
which gives us the expressions in \eqref{diam-cmnd} and \eqref{diam-prefatt} 
and completes the proof of the upper bound.

\section{Lower bound for configuration model}
\label{lowerproofs-cmnd}

In this section we prove Statements~\ref{stat-MKC}, \ref{stat-distance-boundaries}
and~\ref{stat-distance} for the configuration model. By the discussion in
Section~\ref{sub-lower}, this completes the proof of the lower
bound in Theorem~\ref{main-cmnd}.

In our proof, it will be convenient to choose a particular order to pair the half-edges. This is made precise in the following remark:

\begin{remark}[Exchangeability in half-edge pairing]
\rm
\label{cmnd-conditionedlaw}
Given a sequence $\sub{d} = (d_1, \ldots, d_n)$
such that $\ell_n = d_1 +\ldots + d_n$ is even,
the configuration model $\CMnd$ can be built iteratively as follows:
\begin{itemize}
\item[$\rhd$] start with $d_i$ half-edges attached to each vertex 
$i \in [n] = \{1,2,\ldots, n\}$;

\item[$\rhd$] choose an \emph{arbitrary} half-edge and pair it to a uniformly chosen half-edge;

\item[$\rhd$] choose an \emph{arbitrary} half-edge, among the $\ell_n - 2$
that are still unpaired, and pair it to a uniformly chosen half-edge; and so on.
\end{itemize}
The {\em order} in which the arbitrary half-edges are chosen does not matter in the above, by exchangeability (see also \cite[Chapter 7]{vdH1}).
\end{remark}

\subsection{Proof of Statement~\ref{stat-MKC}}
\label{proof-lower-cmnd-1}

With a slight abuse of notation ({\blue see} \eqref{eq:dmin0}), in this section we set
	\begin{equation*}
	\dmin = \min\{d_1, \ldots, d_n\} \,.
	\end{equation*}
Given a vertex $v\in[n]$ and $k\in\N$, we denote the set of vertices at distance at most $k$ from $v$
(in the graph $\CMnd$) by $U_{\leq k}(v)$ and we call it the \emph{$k$-neighborhood of $v$}.

\begin{Definition}[Minimally-$k$-connected vertex]\label{def:minicm}
For $k\in\N_0$, a vertex $v\in[n]$ is called {\em minimally-$k$-connected} when all the vertices in $U_{\leq k}(v)$ have minimal degree, i.e.,
  	$$
	  \begin{array}{ccc}
		  d_i = \dmin & &\mbox{ for all } i\in U_{\leq k}(v) \,,
	  \end{array}
  	$$
and furthermore there are no self-loops, multiple edges or cycles in $U_{\leq k}(v)$. Equivalently, $v$ is minimally-$k$-connected when the graph $U_{\leq k}(v)$ is a {\blue regular} tree with degree $\dmin$.

We denote the (random) set of minimally-$k$-connected vertices by $\mathcal{M}_k \subseteq [n]$,
and its cardinality by $M_k = |\mathcal{M}_k|$, i.e.\ $M_k$ denotes the number of minimally-$k$-connected vertices.
\end{Definition}

\begin{remark}[The {\blue volume} of the $k$-neighborhood of $k$-minimally connected vertices]
\rm
\label{rem:layers}
For a minimally-$k$-connected vertex $v$, since $U_{\leq k}(v)$ is a tree with
degree $\dmin$, the number of \emph{edges} inside $U_{\leq k}(v)$ equals
(assuming $\dmin \ge 2$)
	\eqn{
	\label{i-k}
  	i_k = \sum_{l=1}^k\dmin(\dmin-1)^{l-1} =
  	\begin{cases}
 	\dmin \, k & \text{if } \dmin = 2;\\
 	\rule{0pt}{1.9em} \displaystyle \dmin\frac{(\dmin-1)^k-1}{\dmin-2}
  	& \text{if } \dmin \ge 3.
  	\end{cases}
	}
Moreover, the number of \emph{vertices} inside $U_{\leq k}(v)$ equals $i_k+1$.
By \eqref{i-k}, it is clear why $\dmin>2$, or $\dmin\geq 3$, is crucial. Indeed, this implies that the volume of neighborhoods of minimally-$k$-connected vertices grows exponentially in $k$.
\end{remark}

\begin{remark}[Collapsing minimally-$k$ connected trees]\rm
\label{cmnd-conditionMKC}
By Remarks~\ref{cmnd-conditionedlaw} and~\ref{rem:layers}, conditionally on the event $\{v \in \mathcal{M}_k\}$ that a given vertex $v$ is minimally-$k$-connected, the random graph obtained from $\CMnd$ by collapsing
$U_{\leq k}(v)$ to a single vertex, called $a$, \emph{is still a configuration model} with $n - i_k$ vertices and
with $\ell_n$ replaced by $\ell_n - 2 i_k$, where the new vertex $a$ has degree $\dmin(\dmin-1)^k$.

Analogously, conditionally on the event $\{v \in \mathcal{M}_k, \,
w \in \mathcal{M}_m, \, U_{\leq k}(v) \cap U_{\leq m}(w) = \varnothing\}$ that
two given vertices $v$ and $w$ are minimally-$k$ and minimally-$m$-connected
with \emph{disjoint} neighborhoods, collapsing
$U_{\leq k}(v)$ and $U_{\leq m}(w)$ to single vertices $a$ and $b$
yields again a configuration model with $n - i_k - i_m $ vertices,
where $\ell_n$ is replaced by $\ell_n - 2 i_k - 2 i_m$ and where
the new vertices $a$ and $b$ have degrees equal to $\dmin(\dmin-1)^k$
and $\dmin(\dmin-1)^m$, respectively.
\end{remark}

We denote the number of vertices of degree $k$ in the graph by $n_k$, i.e.,
	\eqn{
	n_k = \sum_{i\in[n]}\I_{\{d_i = k\}}.
	}
We now study the first two moments of $M_k$, where we recall that the total degree $\ell_n$ is defined by \eqref{cmnd-elln}:

\begin{Proposition}[Moments of $M_k$]
\label{cmnd-mk-moments}
Let $\CMnd$ be a configuration model such that $\dmin \ge 2$.
Then, for all $k\in\N$,
	\eqn{
	\label{cmnd-mk-firstmoment}
		\E[M_k] = n_{\dmin}\prod_{i=1}^{i_k}\frac{\dmin(n_{\dmin}-i)}{\ell_n-2i+1},
	}
where $i_k$ is defined in \eqref{i-k}. When, furthermore, $\ell_n > 4i_k$,
	\eqn{
	\label{cmnd-mk-secondmoment}
	\E[M_k^2]\leq \E[M_k]^2+ \E[M_k]\left(
	(i_k + 1) + 
	i_{2k} \,\dmin \,  \frac{n_{\dmin}}{\ell_n - 4 i_k}\right).
	}
\end{Proposition}
\medskip

Before proving Proposition \ref{cmnd-mk-moments}, let us complete the proof of Statement~\ref{stat-MKC} subject to it.  We are working under the assumptions of Theorem~\ref{main-cmnd}, hence $\dmin \ge 3$ and the degree sequence $\sub{d}$ satisfies the degree regularity condition Condition \ref{drc}, as well as the polynomial distribution condition Condition \ref{pol-dis-con} with exponent $\tau \in (2,3)$.
Recalling \eqref{cmnd-elln}-\eqref{emp-deg-dis},
we can write $n_{\dmin} = n \, p_{\dmin}^{\sss(n)}$ and 
$\ell_n = n \, \sum_{k\in\N} k p_k^{\sss(n)}$, so that, as $n\to\infty$,
	\begin{equation} \label{eq:nel}
	n_{\dmin} {\blue =} 
	n \, p_{\dmin}(1+o(1)) \,, \qquad
	\ell_n = n \, \mu (1+o(1))\,, \qquad \text{with} \qquad
	p_{\dmin} > 0 \,, \quad
	\mu := \sum_{k\in\N} k p_k < \infty .
	\end{equation}
Recalling the definition \eqref{kstar-tree}
of $k^{-}_n$ and \eqref{i-k}, for $k= k^{-}_n$,
	\begin{equation} \label{eq:ikn}
	i_{k^{-}_n} = \dmin\frac{(\dmin-1)^{k^{-}_n}-1}{\dmin-2}
	= \frac{\dmin}{\dmin - 2} (\log n)^{1-\varepsilon}(1+o(1)),
	\qquad \text{hence} \qquad
	i_{2k^{-}_n} = O((\log n)^{2(1-\varepsilon)}) .
	\end{equation}
Bounding $\E[M_k] \le n$, it follows
by \eqref{cmnd-mk-secondmoment} that
\begin{equation} \label{eq:bova}
	\Var[M_{k^{-}_n}] \le \E[M_{k^{-}_n}] \left(
	O(i_{k^{-}_n}) + O(i_{2k^{-}_n}) \right) \le
	n \, O((\log n)^{2(1-\varepsilon)}) = n^{1+o(1)} \,.
\end{equation}
On the other hand, applying \eqref{cmnd-mk-firstmoment},
for some $c \in (0,1)$ one has
\begin{equation} \label{eq:bome}
	\E[M_{k^{-}_n}] \ge n \, p_{\dmin} 
	\, \left( \frac{\dmin \, p_{\dmin}}{\mu} + o(1) \right)^{i_{k^{-}_n}}
	\ge n \, p_{\dmin} \, c^{(\log n)^{1-\varepsilon}}
	= n^{1-o(1)} \,.
\end{equation}
Relations \eqref{eq:bova} and \eqref{eq:bome}
show that \eqref{MKC-moments} holds,
completing the proof of Statement~\ref{stat-MKC}.\qed
\medskip

{\blue
\begin{remark}[Disjoint neighborhoods]\rm\label{rem:disjcm}
Let us show that, with high probability, there are vertices
$v, w \in \cM_{k_n^-}$ with 
$U_{\le k_n^-}(v) \cap U_{\le k_n^-}(w) = \varnothing$.
We proceed by contradiction: fix $v \in \cM_{k_n^-}$ and assume that,
for every vertex $w \in \cM_{k_n^-}$, one has
$U_{\le k_n^-}(v) \cap U_{\le k_n^-}(w) \ne \varnothing$.
Then, for any $w \in \cM_{k_n^-}$
there must exist \emph{a self-avoiding path from $v$ to $w$ of length $\le 2 k_n^-$
which only visits vertices with degree $\dmin$}
(recall that $U_{\le k_n^-}(v)$ and $U_{\le k_n^-}(w)$ are regular trees).
However, for fixed $v$, the number of such
paths is $O((\dmin - 1)^{2 k_n^-}) = O((\log n)^{2(1-\epsilon)})$, see \eqref{kstar-tree},
while by Statement~\ref{stat-MKC}
the number of vertices $w \in \cM_{k_n^-}$ is much larger,
since $M_{k_n^-} \sim \E[M_{k_n^-}] = n^{1-o(1)}$, see \eqref{eq:bome}.
\end{remark}
}

\begin{proof}[Proof of Proposition~\ref{cmnd-mk-moments}]
To prove \eqref{cmnd-mk-firstmoment} we write
	\eqn{
	\label{cmnd-mkmom-1}
	M_k = \sum_{v\in[n]\colon \, d_v = \dmin}\I_{\{v\in\mathcal{M}_k\}},
	}
and since every vertex in the sum has the same 
probability of being minimally-$k$-connected, 
	\eqn{
	\label{cmnd-mkmom-2}
	\E\left[M_k\right] = n_{\dmin}\pr(v\in\mathcal{M}_k).
	}
A vertex $v$ with $d_v=\dmin$ is in $\mathcal{M}_k$ when 
{\blue  all the half-edges in $U_{\le k}(v)$}
 are paired to half-edges incident to distinct vertices having minimal degree, without generating cycles. 
By Remark~\ref{cmnd-conditionedlaw}, we can start pairing a half-edge incident to $v$ to a half-edge incident to another vertex of degree $\dmin$. Since there are $n_{\dmin}-1$ such vertices, this event has probability
	\begin{equation*}
	\frac{\dmin(n_{\dmin}-1)}{\ell_n-1}
	\end{equation*}
We iterate this procedure, and suppose that we have already successfully paired $(i-1)$ couples of half-edges; then the next half-edge can be paired to a distinct vertex of degree $\dmin$ with probability
	\eqn{
	\label{cmnd-mkmom-3}
	\frac{ \dmin(n_{\dmin}-i)}{\ell_n-2(i-1)-1} = \frac{\dmin(n_{\dmin}-i)}{\ell_n-2i+1}.
	}
Indeed, every time that we use a half-edge of a vertex of degree $\dmin$, we cannot use its remaining 
half-edges, and every step we make reduces the total number of possible
half-edges by two. By \eqref{i-k}, exactly $i_k$ couples of half-edges need to be paired for
$v$ to be minimally-$k$-connected, so that
	\eqn{
	\label{cmnd-mkmom-4}
	\E[M_k] = n_{\dmin} \pr(v\in\mathcal{M}_k) = n_{\dmin}
	\prod_{i=1}^{i_k}\frac{ \dmin(n_{\dmin}-i)}{\ell_n-2i+1}.
	}
which proves \eqref{cmnd-mk-firstmoment}. If $n_{\dmin} \le i_k$ the right hand side vanishes, in agreement with the fact that there cannot be any minimally-$k$-connected vertex in this case (recall \eqref{i-k}).

To prove \eqref{cmnd-mk-secondmoment}, we notice that
	\eqn{
	\label{cmnd-mkmom-5}
	\E[M_k^2] = \sum_{v,w\in[n]\colon \,
	d_v = d_w = \dmin}\pr(v,w\in\mathcal{M}_k).
	}
We distinguish different cases: the $k$-neighborhoods of $v$ and $w$ might be disjoint
or they may overlap, in which case $w$ can be included in $U_{\leq k}(v)$ or not. Introducing the events
	\begin{equation} \label{eq:AB}
	A_{v,w} = \left\{ U_{\leq k}(v)\cap U_{\leq k}(w) \neq \varnothing\right\} ,
	\qquad B_{v,w} = \left\{ w\in U_{\leq k}(v)\right\} ,
	\end{equation}
we can write the right hand side of \eqref{cmnd-mkmom-5} as
	\eqn{
	\label{cmnd-mkmom-6}
	\sum_{\substack{v,w\in[n]\\
	d_v = d_w = \dmin}} \!\! \!\!\!\! \big[
	\pr\left( v,w\in\mathcal{M}_k, A_{v,w}^c\right)
	+\pr\left(v,w\in\mathcal{M}_k, A_{v,w}, B_{v,w}\right)+
 	\pr\left(v,w\in\mathcal{M}_k, A_{v,w},B_{v,w}^c\right)\big].
	}

Let us look at the first term in \eqref{cmnd-mkmom-6}. By Remarks~\ref{rem:layers} and~\ref{cmnd-conditionMKC}, 
conditionally on $\{v\in\mathcal{M}_k\}$, the probability of $\{w\in\mathcal{M}_k, A_{v,w}^c\}$ equals the
probability that $w$ is minimally-$k$-connected in a new configuration model, with $\ell_n$ replaced by $\ell_n - 2 i_k$ and with the number of vertices with minimal degree reduced from $n_{\dmin}$ to $n_{\dmin}- (i_k+1)$.
Then, by the previous analysis ({\blue see} \eqref{cmnd-mkmom-4}),
	\eqn{
	\label{cmnd-mkmom-7}
   	\pr\left(v, w\in\mathcal{M}_k, A_{v,w}^c  \right) =
    	\prod_{i=1}^{i_k}\frac{\dmin(n_{\dmin}-i-i_k-1)}{\ell_n-2i-2i_k+1} 
    	\, \pr\left(v \in \mathcal{M}_k\right) .
	}
By direct computation, 
the ratio in the right hand side of \eqref{cmnd-mkmom-7}
is always maximized for $i_k=0$ (provided $\ell_n \ge 2 n_{\dmin} - 3$, which is
satisfied since $\ell_n \ge \dmin n_{\dmin} \ge 3 n_{\dmin}$ by assumption).
Therefore, setting $i_k = 0$ in the ratio and recalling \eqref{cmnd-mkmom-4}, we get the upper bound
	\eqn{
	\label{cmnd-mkmom-8}
  	\pr\left(v,w\in\mathcal{M}_k, A_{v,w}^c\right) \leq \left[\prod_{i=1}^{i_k}
  	\frac{\dmin(n_{\dmin}-i)}{\ell_n-2i+1}\right] \, \pr(v\in\mathcal{M}_k)=
		\pr\left(v\in\mathcal{M}_k\right)^2.
	}
Since there are at most $n_{\dmin}^2$ pairs of vertices of degree $\dmin$, it follows from \eqref{cmnd-mkmom-8} that 
	\eqn{
	\label{cmnd-mkmom-9}
	\sum_{\substack{v,w\in[n]\\
	d_v = d_w = \dmin}} \pr\left(v,w\in\mathcal{M}_k, A_{v,w}^c\right) \leq
	n_{\dmin}^2 \, \pr\left(v\in\mathcal{M}_k\right)^2 = \E[M_k]^2,
	}
which explains the first term in \eqref{cmnd-mk-secondmoment}. 

For the second term in \eqref{cmnd-mkmom-6}, $v$ and $w$ are
minimally-$k$-connected with overlapping neighborhoods, and $w\in U_{\leq k}(v)$. 
Since $\left\{v,w\in\mathcal{M}_k\right\}\cap A_{v,w}\cap B_{v,w}\subseteq \left\{v\in\mathcal{M}_k\right\}\cap B_{v,w}$, we can bound
	\eqn{
	\label{cmnd-mkmom-11}
	\sum_{\substack{v,w\in[n]\\
	d_v = d_w = \dmin}}\pr\left(v,w\in\mathcal{M}_k, A_{v,w},B_{v,w}\right) 
	\leq \E\Big[\sum_{v\in[n]: \, d_v = \dmin}
	\I_{\{v\in\mathcal{M}_k\}}\sum_{w\in[n]: \, d_w = \dmin}\I_{B_{v,w}}\Big] ,
	}
and note that $\sum_{w\in[n]}\I_{B_{v,w}} = \left|U_{\leq k}(v)\right| = i_k + 1$, by Remark~\ref{rem:layers}.
Therefore
	\eqn{
	\label{cmnd-mkmom-12}
	\sum_{\substack{v,w\in[n]\\
	d_v = d_w = \dmin}} \pr\left(v,w\in\mathcal{M}_k ,A_{v,w},B_{v,w}\right)
	\leq \E[M_k] \, (i_k+1) ,
	}
which explains the second term in \eqref{cmnd-mk-secondmoment}.

For the third term in \eqref{cmnd-mkmom-6}, $v$ and $w$ are minimally-$k$-connected vertices with overlapping neighborhoods, but $w\not\in U_{\leq k}(v)$. This means that $\dist(v,w) = l+1$ for some $l\in \{k, \ldots, 2k-1\}$,
so that $U_{\leq k}(v) \cap U_{\leq l-k}(w) = \varnothing$ and, moreover, a half-edge on the boundary of $U_{\leq (l-k)}(w)$ is paired to a half-edge on the boundary of $U_{\leq k}(v)$, an event that we call $F_{v,w; l,k}$.
Therefore
	\eqn{	
	\label{cmnd-mkmom-13}
  	\left\{w\in\mathcal{M}_k\right\}\cap A_{v,w}\cap B_{v,w}^c \subseteq 
  	\bigcup_{l=k}^{2k-1}
 	 \left\{w\in\mathcal{M}_{l-k}\right\} \cap 
  	\{U_{\leq k}(v) \cap U_{\leq l-k}(w) = \varnothing\} \cap F_{v,w;l,k} .
	}
and we stress that in the right hand side $w$ is only minimally-$(l-k)$-connected (in case $l=k$ this just
means that $d_w = \dmin$). Then
	\eqn{
	\label{cmnd-mkmom-14}
	\pr\left(v,w\in\mathcal{M}_k, A_{v,w},B_{v,w}^c\right)\leq\sum_{l=k}^{2k-1}
	\E\left[\I_{\{v\in\mathcal{M}_k, \, w\in \mathcal{M}_{l-k}, \,
	U_{\leq k}(v) \cap U_{\leq l-k}(w) = \varnothing\}}
	\I_{F_{v,w;l,k}} \right].
	}
By Remark~\ref{cmnd-conditionMKC}, conditionally on
$\{v \in \mathcal{M}_k, \,
w \in \mathcal{M}_{l-k}, \, U_{\leq k}(v) \cap U_{\leq l-k}(w) = \varnothing\}$,
we can collapse $U_{\leq k}(v)$ and $U_{\leq l-k}(w)$ to single vertices
$a$ and $b$ with degrees respectively
$\dmin(\dmin-1)^k$ and $\dmin(\dmin-1)^{l-k}$, getting a new configuration
model with $\ell_n$ replaced by $\ell_n - 2 i_k - 2 i_{l-k}$.
Bounding the probability that a half-edge of $a$ is paired
to a half-edge of $b$, we get
	\eqn{
	\label{cmnd-mkmom-15}
	\begin{split}
	\pr(F_{v,w;l,k} \,|\, v \in \mathcal{M}_k, \, &
	w \in \mathcal{M}_{l-k}, \, U_{\leq k}(v) \cap U_{\leq l-k}(w) = \varnothing ) \\
	& \le
	\frac{\dmin(\dmin-1)^k \dmin (\dmin-1)^{l-k}}{\ell_n - 2 i_k - 2 i_{l-k} - 1}
	\le \frac{\dmin^2(\dmin-1)^{l}}{\ell_n - 4 i_k} \,,
	\end{split}
	}
because $l \le 2k-1$ and, consequently, $i_{l-k} \le i_{k-1} \le i_k - 1$.
Plugging \eqref{cmnd-mkmom-15} into \eqref{cmnd-mkmom-14}, and then
forgetting the event
$\{w \in \mathcal{M}_{l-k}, \, U_{\leq k}(v) \cap U_{\leq l-k}(w) = \varnothing\}$, leads to
	\begin{equation} \label{eq:lastpro}
	\begin{split}
	\sum_{\substack{v,w\in[n]\\ d_v = d_w = \dmin}}
	\pr\left(v,w\in\mathcal{M}_k, \, A_{v,w}, \, B_{v,w}^c\right) 
	& \leq
	\left(\sum_{l=k}^{2k-1} \frac{\dmin^2(\dmin-1)^l}{\ell_n - 4 i_k}\right)
	\sum_{\substack{v,w\in[n]\\ d_v = d_w = \dmin}} 
	\pr(v\in\mathcal{M}_{k}) \\
	& \leq \frac{\dmin(\dmin-1)}{\ell_n - 4 i_k} \, i_{2k-1} \,
	n_{\dmin} \, \E[M_k] \,,
	\end{split}
	\end{equation}
where we have used the definition \eqref{i-k} of $i_{2k-1}$.
Since $(\dmin - 1) i_{2k-1} \le i_{2k}$, again by \eqref{i-k}, we have
obtained the third term in \eqref{cmnd-mk-secondmoment}.
\end{proof}

\subsection{Proof of Statement \ref{stat-distance-boundaries}}
\label{lower-proof-cmnd-2}
We recall that $W^n_1$ and $W^n_2$ are two independent random vertices
chosen uniformly in $\mathcal{M}_{k^{-}_n}$ (the set of
minimally-$k^{-}_n$-connected vertices), assuming that $\mathcal{M}_{k^{-}_n}\neq \varnothing$
(which, {\blue as we have shown,} occurs {\blue with high probability}). Our goal is to show that
	\begin{equation}
	\label{eq:goalst2}
	\lim_{n\to\infty} 
	\pr(E_n) = 0 ,
	\end{equation}
where we set
	\begin{equation}
	\label{En-def}
	E_n := \left\{ \dist\big( U_{\le k^{-}_n}(W^n_1),
	U_{\le k^{-}_n}(W^n_2)\big) \le 2 \bar{k}_n \right\} =
	\left\{\dist(W_1^n,W_2^n)\leq 2k^{-}_n+2\bar{k}_n\right\} .
	\end{equation}

We know from Statement~\ref{stat-MKC} that as $n\to\infty$
	\eqn{
	\label{concl-split-opiccolo}
	\pr\left( M_{k^{-}_n} \le \frac{1}{2}\E[M_{k^{-}_n}] \right) \le
	\pr\left( |M_{k^{-}_n}-\E[M_{k^{-}_n}]|> \frac{1}{2}\E[M_{k^{-}_n}] \right)
	\le \frac{\Var[M_{k^{-}_n}]}{\frac{1}{4} \E[M_{k^{-}_n}]^2} = o(1).
	}
Therefore,
	\begin{equation}
	\label{concl-sumtobound}
	\begin{split}
	\pr(E_n) & = \pr\left(E_n\cap\{M_{k^{-}_n} > \tfrac{1}{2}\E[M_{k^{-}_n}]\}\right)
	+ o(1) \\
	& = \E\Big[\sum_{v_1,v_2\in [n]}
	\I_{\{W_1^n=v_1,W_2^n=v_2\}}\I_{\{\dist(v_1,v_2)\leq 2k^{-}_n+2\bar{k}_n\}}
	\I_{\{M_{k^{-}_n} > \frac{1}{2}\E[M_{k^{-}_n}]\}}\Big]+ o(1) \\
	& \leq \E\Big[\sum_{v_1,v_2\in [n]}\frac{\I_{\{v_1\in \mathcal{M}_{k^{-}_n},v_2\in \mathcal{M}_{k^{-}_n}\}}}
	{M_{k^{-}_n}^2} \I_{\{\dist(v_1,v_2)\leq 2k^{-}_n+2\bar{k}_n\}}
			\I_{\{M_{k^{-}_n} > \frac{1}{2}\E[M_{k^{-}_n}]\}}\Big]+ o(1) \\
 	& \leq \sum_{v_1,v_2\in[n]}
	\frac{\pr\left(v_1,v_2\in \mathcal{M}_{k^{-}_n}, \, \dist(v_1,v_2)\leq 2k^{-}_n+2\bar{k}_n \right)}
	{\frac{1}{4}\E[M_{k^{-}_n}]^2} + o(1).
	\end{split}
	\end{equation}

In analogy with \eqref{eq:AB}, we introduce the event
	\begin{equation*}
	A_{v_1, v_2} := \big\{
	U_{\le k^{-}_n}(v_1) \cap U_{\le k^{-}_n}(v_2) \ne \varnothing \big\} \,,
	\end{equation*}
and show that it gives a negligible contribution.  Recalling the proof of Proposition~\ref{cmnd-mk-moments}, in particular \eqref{cmnd-mkmom-12} and \eqref{eq:lastpro}, the sum restricted to $A_{v_1, v_2}$ leads
precisely to the second term in the right hand side of \eqref{cmnd-mk-secondmoment}:
	\begin{equation} \label{eq:preci}
	\begin{split}
	\sum_{v_1,v_2\in[n]}
	\frac{\pr\left(v_1,v_2\in \mathcal{M}_{k^{-}_n}, \,
	A_{v_1, v_2} \right)}{\frac{1}{4}\E[M_{k^{-}_n}]^2} 
	& \le \frac{\E[M_{k^{-}_n}] 
	\left( (i_{k^{-}_n}+1) + i_{2{k^{-}_n}} \frac{\dmin \, n_{\dmin}}{\ell_n - 4 i_{k^{-}_n}} \right)}
	{\frac{1}{4}\E[M_{k^{-}_n}]^2} \\
	& = \frac{O(i_{k^{-}_n}) + O(i_{2k^-_{n}})}{\E[M_{k^{-}_n}]} 
	= \frac{O((\log n)^2)}{n^{1-o(1)}} = o(1) \,,
	\end{split}
	\end{equation}
where we have used \eqref{eq:ikn} and \eqref{eq:bome}
(see also \eqref{eq:nel}).

We can thus focus on the event  $A_{v_1, v_2}^c = \{U_{\le k^{-}_n}(v_1) \cap U_{\le k^{-}_n}(v_2) = \varnothing\}$.
By Remark~\ref{cmnd-conditionMKC},
	\begin{equation} \label{eq:rema}
	\pr\left(\dist(v_1,v_2)\leq 2k^{-}_n+2\bar{k}_n\mid
	v_1,v_2\in \mathcal{M}_{k^{-}_n}, \, A_{v_1, v_2}^c \right)
	 = \hat\pr\left(\dist(a,b)\leq 2\bar{k}_n\right),
	\end{equation}
where $\hat\pr$ is the law of the new configuration model which
results from collapsing the neighborhoods $U_{\le k^{-}_n}(v_1)$ and $U_{\le k^{-}_n}(v_2)$
to single vertices $a$ and $b$, with degrees $\dmin(\dmin-1)^{k^{-}_n}
= O(\log n)$ (recall \eqref{kstar-tree}-\eqref{d_tree}). The degree sequence $\sub{\hat d}$ of 
this new configuration model is a slight modification of the original degree sequence $\sub{d}$: 
two new vertices of degree $O(\log n)$ have been added, while $2 (i_{k^{-}_n} + 1) = O(\log n)$ 
vertices with degree $\dmin$ have been removed  (recall \eqref{eq:ikn}). 
Consequently $\sub{\hat d}$ still satisfies the assumptions of Theorem~\ref{main-cmnd},
hence Statement~\ref{stat-distance} (to be proved in Section~\ref{lower-proof-cmnd-3})
holds for $\hat \pr$ and we obtain
	\begin{equation} \label{eq:o1}
	\hat\pr\left(\dist(a,b)\leq 2\bar{k}_n\right) = o(1) .
	\end{equation}

We are ready to conclude the proof of Statement \ref{stat-distance-boundaries}. By \eqref{concl-sumtobound}-\eqref{eq:preci}-\eqref{eq:rema},
	\begin{equation*}
	\begin{split}
	\pr(E_n) & = \sum_{v_1,v_2\in[n]}
	\frac{\pr\left(v_1,v_2\in \mathcal{M}_{k^{-}_n}, \,
	\dist(v_1,v_2)\leq 2k^{-}_n+2\bar{k}_n, \, A_{v_1, v_2}^c
	\right)}{\frac{1}{4}\E[M_{k^{-}_n}]^2} + o(1) \\
	& \le \hat \pr\left(\dist(a,b)\leq 2\bar{k}_n\right)\sum_{v_1,v_2\in[n]}
	\frac{\pr\left(v_1,v_2\in \mathcal{M}_{k^{-}_n}\right)}{\frac{1}{4}\E[M_{k^{-}_n}]^2} + o(1) \\
	& =	\hat\pr\left(\dist(a,b)\leq 2\bar{k}_n\right)
	\frac{\E[(M_{k^{-}_n})^2]}{\frac{1}{4}\E[M_{k^{-}_n}]^2} + o(1) .
	\end{split}
	\end{equation*}
Observe that $\E[(M_{k^{-}_n})^2] = \E[M_{k^{-}_n}]^2 + \Var(M_{k^{-}_n})= O(\E[M_{k^{-}_n}]^2)$,
by the second relation in \eqref{MKC-moments}. Applying \eqref{eq:o1},
it follows that $\pr(E_n) = o(1)$, completing the proof of Statement \ref{stat-distance-boundaries}.
\qed

\subsection{Proof of Statement \ref{stat-distance}}
\label{lower-proof-cmnd-3}
In this section, we give a self-contained proof of Statement \ref{stat-distance} for $\CMnd$, as used in the proof of Statement \ref{stat-distance-boundaries}.

Given two vertices $a, b \in [n]$, let $\mathcal{P}_k(a,b)$ be the set of all {\blue self-avoiding} paths of length $k$ from $a$ to $b$, i.e.\ of all sequences $(\pi_0, \pi_1, \ldots, \pi_k) \in [n]^{k+1}$ with $\pi_0 = a$, $\pi_k = b$ and such that $(\pi_{i-1}, \pi_i)$ is an edge in the graph, for all $i=1,\ldots, k$. Analogously, let $\mathcal{P}_k(a) = \cup_{b\in[n]} \mathcal{P}_k(a,b)$  denote the set of all paths of length $k$ starting at $a$.

Let us fix an arbitrary increasing sequence $(g_l)_{l\in\N_0}$ (that will be specified later). {\blue Define, for $a,b\in\R$, $a\wedge b := \min\{a,b\}$.} We say that a path $\pi\in\mathcal{P}_k(a,b)$ is {\itshape good}  when $d_{\pi_l}\leq g_l\wedge g_{k-l}$ for every $l=0,\ldots,k$, and {\itshape bad} otherwise.  In other words, a path is good when the degrees along the path do not increase too much from $\pi_0$ to $\pi_{k/2}$, and similarly they do not increase too much in the backward direction, from $\pi_k$ to $\pi_{k/2}$.

For $k \in \N_0$, we introduce the event
	\begin{equation}
	\label{eq:gopa}
		\mathcal{E}_k(a,b) = \left\{
	\exists  \pi\in \mathcal{P}_k(a,b): \
	\pi \text{ is a good path}\right\} .
	\end{equation}
To deal with bad paths, we define
	\begin{equation} \label{eq:bapa}
	\mathcal{F}_k(a) = \left\{
	\exists  \pi\in \mathcal{P}_k(a): \
	d_{\pi_k} > g_k \ \text{ but } \ d_{\pi_i} \le g_i \ \forall i \le k-1\right\} .
	\end{equation}
If $\dist_{\CMnd}(a,b) \le 2 \bar{k}$, then there must be a path in $\mathcal{P}_k(a,b)$
for some $k \le \bar{k}$, and this path might be good or bad. This leads to the simple bound
	\begin{equation}\label{eq:keyinc}
	\pr(\dist_{\CMnd}(a,b)\leq 2\bar{k})\leq \sum_{k=0}^{2\bar{k}}
  	\pr(\mathcal{E}_k(a,b))+\sum_{k = 0}^{\bar{k}}
  	\left[\pr(\mathcal{F}_k(a))+\pr(\mathcal{F}_k(b))\right].
	\end{equation}

We give explicit estimates for the two sums in the right hand side. We introduce the \emph{size-biased distribution function} $F^*_n$ associated to the degree sequence $\sub{d} = (d_1, \ldots, d_n)$ by
	\begin{equation}\label{eq:sizebias}
	F^*_n(t) = \frac{1}{\ell_n} \sum_{v \in [n]} d_v \, \I_{\{d_v \le t\}} .
	\end{equation}
If we choose uniformly one of the $\ell_n$ half-edges in the graph, and call $D^*_n$ the degree of the vertex incident to this half-edge, then $F^*_n(t) = \pr(D^*_n \le t)$.
We also define the truncated mean
	\begin{equation}\label{eq:nu}
	  \nu_n(t) = \E\big[ (D^*_n-1) \I_{\{D^*_n \le t\}} \big] =
  	\frac{1}{\ell_n}\sum_{v\in[n]}d_v(d_v-1)\I_{\{d_v\leq t\}}.
	\end{equation}
Now we are ready to bound \eqref{eq:keyinc}.

\begin{Proposition}[Path counting for configuration model]
\label{cmnd-lolog lower bound}
Fix $\sub{d} = (d_1, \ldots, d_n)$ (such that
$\ell_n = d_1 + \ldots + d_n$ is even) and an increasing sequence
$(g_l)_{l\in\N_0}$.
For all distinct vertices $a,b \in [n]$ with $d_a \le g_0$, $d_b \le g_0$, 
and for all $\bar k \in \N$,
	\begin{equation}\label{cmnd-distance-bound}
	\begin{split}
 	\pr\left(\dist_{\CMnd}(a,b)\leq 2\bar{k}\right) {\leq}
		& \frac{d_a d_b}{\ell_n}
		  \sum_{k=1}^{2\bar{k}}
		   \bigg( 1 - \frac{2k}{\ell_n} \bigg)^{-k}\,
		  \prod_{l=1}^{k-1}\nu_n(g_l\wedge g_{h-l})\\
  	& \displaystyle + (d_a+d_b) \sum_{k=1}^{\bar{k}}
	 \bigg( 1 - \frac{2k}{\ell_n} \bigg)^{-k}\,
  	(1-F^*_{n}(g_k)) \prod_{l=1}^{k-1}\nu_n(g_l).
	\end{split}
	\end{equation}
\end{Proposition}

\begin{proof}
Fix an arbitrary sequence of vertices $\pi = (\pi_i)_{0 \le i \le k} \in [n]^{k+1}$. The probability that vertex $\pi_0$ is connected to $\pi_1$ is at most
	\begin{equation*}
	\frac{d_{\pi_0} d_{\pi_1}}{\ell_n-1} \,,
	\end{equation*}
because there are $d_{\pi_0} d_{\pi_1}$ ordered couples of half-edges, each of which can be paired with probability $1/(\ell_n-1)$ (recall Remark~\ref{cmnd-conditionedlaw}), and we use the union bound. By similar arguments, conditionally on a specific half-edge incident to $\pi_0$ being paired to a specific half-edge incident to $\pi_1$, the probability that another half-edge incident to $\pi_1$ is paired to a half-edge incident to $\pi_2$ is by the union bound bounded from above by
	\begin{equation*}
	\frac{(d_{\pi_1}-1) d_{\pi_2}}{\ell_n-3} .
	\end{equation*}
Iterating the argument, the probability that $\pi$ is a path in $\CMnd$ is at most
	\begin{equation} \label{eq:proba}
	\frac{d_{\pi_0} d_{\pi_1}}{\ell_n-1}
	\frac{(d_{\pi_1}-1) d_{\pi_2}}{\ell_n-3}
	\frac{(d_{\pi_2}-1) d_{\pi_3}}{\ell_n-5} \cdots
	\frac{(d_{\pi_{k-1}}-1) d_{\pi_k}}{\ell_n-(2k-1)} .
	\end{equation}

Let us now fix $a,b \in [n]$ with $a \ne b$. Recalling \eqref{eq:gopa}-\eqref{eq:nu},
choosing $\pi_0 = a$, $\pi_k = b$ and summing \eqref{eq:proba} over all vertices $\pi_1, \ldots, \pi_{k-1}$ satisfying
$d_{\pi_i} \le g_i \wedge g_{k-i}$ yields
	\eqn{
	\label{cmnd-lowdistfix-3}
 	 \pr(\mathcal{E}_k(a,b))\leq 
  	d_a d_b 
   	\frac{(\ell_n-2k-1)!!}{(\ell_n-1)!!}
  	\left(\prod_{i=1}^{k-1}
  	\ell_n \, \nu_n(g_i\wedge g_{k-i})\right) .
	}
 Bounding $(\ell_n-2k-1)!!/(\ell_n-1)!! \le (\ell_n - 2k)^{-k}$ yields the
first term in the right hand side of \eqref{cmnd-distance-bound}.
The bound for $\pr(\mathcal{F}_k(a))$ is similar.
Recalling \eqref{eq:bapa}-\eqref{eq:sizebias}, choosing $\pi_0 = a$ and summing
\eqref{eq:proba} over vertices $\pi_1, \ldots, \pi_{k-1}, \pi_k$ such that
$d_{\pi_i} \le g_i$ for $i \le k-1$ while $d_{\pi_k} > g_k$ gives
	\eqn{
	\label{cmnd-lowdistfix-6}
  	\pr(\mathcal{F}_k(a))\leq 
   	d_a 
    	\frac{(\ell_n-2k-1)!!}{(\ell_n-1)!!}
    	\left(\prod_{i=1}^{k-1}
  	\ell_n \, \nu_n(g_i)\right)
 	\big\{ \ell_n \, \big(1 - F^*_n(g_k) \big) \big\} ,
	}
and the same holds for $\pr(\mathcal{F}_k(b))$. Plugging \eqref{cmnd-lowdistfix-3}
and \eqref{cmnd-lowdistfix-6} into \eqref{eq:keyinc} proves \eqref{cmnd-distance-bound}.
\end{proof}

In order to exploit \eqref{cmnd-distance-bound}, we need estimates
on $F_n^*$ and $\nu_n$, provided by the next lemma:

\begin{Lemma}[Tail and truncated mean bounds for $D_n^*$]
\label{lem-tailFn*-bd}
Assume that Condition \ref{pol-dis-con} holds. Fix $\eta>0$, then there exist two constants $C_1=C_1(\eta)$ and $C_2=C_2(\eta)$ such that, for every $x\geq 0$,
	\begin{align}
		1-F^*_n(x)\leq C_1x^{-(\tau-2-\eta)},&  & \nu_n(x)\leq C_2x^{(3 -\tau+\eta)}.
	\end{align}
\end{Lemma}

\begin{proof}
For every $x\geq 0$ and $t\geq 0$ we can see that 
	\eqn{
	1-F^*_n(x) = \frac{1}{\ell_n}\sum_{v\in[n]}d_v\I_{\{d_v>x\}} 
	= \frac{n}{\ell_n}\Big[\frac{1}{n}\sum_{v\in[n]}d_v\I_{\{d_v>x\}}\Big] 
	= \frac{n}{\ell_n}\E\left[D_n\I_{\{D_n>x\}}\right],
	}
where we recall that $D_n$ is the degree of a uniformly chosen vertex. This means that
	\eqn{
	\label{Fn*-bd}
	\begin{split}
		\frac{n}{\ell_n}\E[D_n\I_{\{D_n>x\}}] = &
			\frac{n}{\ell_n}\sum_{j=0}^\infty \pr\left(D_n\I_{\{D_n>x\}}>j\right)
			= \frac{n}{\ell_n}\sum_{j=0}^\infty\pr\left(D_n>j,D_n>x\right)  \\
			& = \frac{n}{\ell_n}\sum_{j=0}^\infty\pr\left(D_n>j\vee x\right) 
			= \frac{n}{\ell_n}\sum_{j=0}^\infty \big( 1-F_{\sub{d},n}(j\vee x) \big)\\
			& = \frac{n}{\ell_n}\Big[x(1-F_{\sub{d},n}(x))+\sum_{j=x}^\infty
			\big( 1-F_{\sub{d},n}(j)\big) \Big]\\
			& \leq \frac{n}{\ell_n}C\Big[x^{-(\tau-2-\eta)}
			+\sum_{j=x}^\infty j^{-(\tau-1-\eta)}\Big]
			\leq C_1 x^{-(\tau-2-\eta)},
	\end{split}
	}
where we have used Condition \ref{pol-dis-con} in the second last step
(recall that $2 < \tau < 3$).

For $\nu_n$, we can instead write
	\eqn{
	\begin{split}
		\nu_n(x) & =  \frac{1}{\ell_n}\sum_{v\in[n]}d_v(d_v-1)\I_{\{d_v\leq x\}} 
				= \frac{n}{\ell_n}\Big[\frac{1}{n}\sum_{v\in[n]}d_v(d_v-1)\I_{\{d_v\leq x\}}\Big] \\
		& = \frac{n}{\ell_n}\E\left[D_n(D_n-1)\I_{\{D_n\leq x\}}\right] 
		\leq \frac{n}{\ell_n}\E\left[D_n^2\I_{\{D_n\leq x\}}\right],
	\end{split}
}
where $D_n$ is again the degree of a uniformly chosen vertex. The claim now follows from 
	\eqn{
	\begin{split}
		\frac{n}{\ell_n}\E\left[D_n^2\I_{\{D_n\leq x\}}\right]
		& = \frac{n}{\ell_n}\sum_{j=0}^\infty(2j +  1)\pr\left(D_n\I_{\{D_n\leq x\}}>j\right)  \\
			& = \frac{n}{\ell_n}\sum_{j=0}^\infty(2j + 1)\pr\left(D_n>j,D_n\leq x\right)\leq \frac{n}{\ell_n}\sum_{j=0}^{x-1} (2j + 1)\pr\left(D_n>j\right)\\
			& = \frac{n}{\ell_n}\sum_{j=0}^{x-1}(2j + 1)[1-F_{\sub{d},n}(j)]\leq
			 \frac{n}{\ell_n}\sum_{j=0}^{x-1}Cj^{-(\tau-2-\eta)}\leq \frac{n}{\ell_n}C_2
			 x^{3-\tau+\eta}.
	\end{split}
}
\vskip-1cm
\end{proof}
\bigskip
\bigskip

We are finally ready to complete the proof of Statement \ref{stat-distance}:

\begin{proof}[Proof of Statement \ref{stat-distance}]
As in \eqref{general-k_n}, we take
\begin{equation}\label{eq:kbarn}
	\bar{k}_n=(1-\varepsilon)\frac{\log\log n}{|\log(\tau-2)|} \,,
\end{equation}
and our goal is to show that, as $n\rightarrow \infty$,
\eqn{
	\max_{a,b \in [n]: \, d_a, d_b\leq \log n}
  \pr\left(\dist_{\CMnd}(a,b)\leq 2\bar{k}_n\right)\longrightarrow 0.
}

We stress that $\tau \in (2,3)$ and $\varepsilon > 0$ are fixed.
Then we choose $\eta>0$ so small that
\begin{equation} \label{eq:usefu}
		2\eta < \tau - 2 \qquad \text{and} \qquad 
	\frac{|\log(\tau-2-2\eta)|}{|\log|\log(\tau-2)|} \le 
	\frac{1-\varepsilon/2}{1-\varepsilon} .
\end{equation}
We use the inequality \eqref{cmnd-distance-bound} given by Proposition~\ref{cmnd-lolog lower bound},
with the following choice of $(g_k)_{k\in\N_0}$:
\begin{equation}\label{eq:gk}
	g_k := (g_0)^{p^k} \,, \qquad 
	\text{where} \qquad
	\begin{cases} g_0 := (\log n)^{\log \log n}; \\
	p := \frac{1}{\tau-2-2\eta}>1.
	\end{cases}
\end{equation}

Let us focus on the first term in the right hand side of \eqref{cmnd-distance-bound}, that is
\begin{equation}\label{eq:firstterm}
		\frac{d_a d_b}{\ell_n}
		  \sum_{k=1}^{2\bar{k}}
		  \bigg( 1 - \frac{2k}{\ell_n} \bigg)^{-k}\,
		  \prod_{l=1}^{k-1}\nu_n(g_l\wedge g_{h-l}) \,.
\end{equation}
Since $\ell_n =\mu n (1+o(1))$ by \eqref{eq:nel}, for $k \le 2 \bar{k}_n$
we have
\begin{equation}\label{eq:eas}
		\bigg( 1 - \frac{2k}{\ell_n} \bigg)^{-k} \le 
	\bigg( 1 - \frac{4\bar{k}_n}{\ell_n} \bigg)^{-{2\bar{k}_n}}
	= 1 + O\bigg(\frac{\bar k_n^2}{\ell_n}\bigg)
	= 1 + O\bigg(\frac{(\log \log n)^2}{n}\bigg) = 1 + o(1) \,.
\end{equation}
Then observe that, by Lemma~\ref{lem-tailFn*-bd} and \eqref{eq:gk},
for $k \le 2 \bar{k}_n$
\begin{equation}\label{eq:quat}
	\begin{split}
	\prod_{l=1}^{k-1}\nu_n(g_l\wedge g_{k-l}) & = \prod_{l=1}^{k/2}\nu_n(g_l)^2 
	\le C_2^{k/2} \prod_{l=1}^{k/2}(g_l)^{2(3-\tau+\eta)}
	= C_2^{k/2} (g_0)^{2(3-\tau+\eta) \sum_{l=1}^{k/2} p^l} \\
	& \le C_2^{\bar{k}_n} (g_0)^{2(3-\tau+\eta) C \, p^{\bar{k}_n}},
\end{split}
\end{equation}
with $C = \frac{p}{p-1}$.
Note that $C_2^{\bar k_n} = O((\log n)^c)$ for some $c \in (0,\infty)$,
{\blue see} \eqref{eq:kbarn}, while by \eqref{eq:usefu}
	\eqn{
	\label{pkn-bd}
	p^{\bar{k}_n} 
	= \mathrm{exp}\Big(|\log(\tau-2-2\eta)|(1-\varepsilon)
	\frac{\log\log n}{|\log(\tau-2)|)}\Big) 
	= (\log n)^{(1-\varepsilon)\frac{|\log(\tau-2-2\eta)|}{|\log(\tau-2)|}}\leq 
		(\log n)^{(1-\varepsilon/2)},
	}
hence the right hand side of \eqref{eq:quat}
is $n^{o(1)}$ (since $g_0 = (\log n)^{\log \log n}$). Then, for $d_a, d_b \le \log n$,
\begin{equation*}
	\eqref{eq:firstterm} \le 
		  \frac{(\log n)^2}{\ell_n} (2\bar k_n)
		  \, \big( 1+o(1) \big) \, n^{o(1)} =
		  O \bigg( \frac{(\log n)^2}{n} \, (\log \log n) \, n^{o(1)} \bigg) = o(1) \,.
\end{equation*}

It remains to look at the second sum in \eqref{cmnd-distance-bound}:
\begin{equation}\label{eq:2su}
		(d_a+d_b) \sum_{k=1}^{\bar{k}_n}
 	 \bigg( 1 - \frac{2k}{\ell_n} \bigg)^{-k}
  	(1-F^*_{n}(g_k)) \prod_{l=1}^{k-1}\nu_n(g_l).
\end{equation}
By Lemma  \ref{lem-tailFn*-bd} ,we can bound $ 1-F^*_{n}(g_k)\leq C_1 (g_k)^{-(\tau-2-\eta)}$.
By \eqref{eq:eas} and $C_1^{\bar k_n} = O((\log n)^c)$ for some $c \in (0,\infty)$,
{\blue see} \eqref{eq:kbarn},
bounding the product in \eqref{eq:2su} like we did in \eqref{eq:quat} yields
\begin{equation}\label{eq:sesu}
		O\big((\log n)^c\big) \,(d_a+d_b) \, 
	\sum_{k=1}^{\bar{k}_n}(g_k)^{-(\tau-2-\eta)}(g_0)^{(3-\tau+\eta) C p^{k-1}},
\end{equation}
where $p=1/(\tau-2-2\eta)$ and $C= \frac{p}{p-1}$. By \eqref{eq:gk}
	\eqn{
	(g_k)^{-(\tau-2-\eta)}(g_0)^{-\frac{p}{p-1}(3-\tau+\eta)p^{k-1}} 
	= (g_{k-1})^{-p(\tau-2-\eta)}(g_{k-1})^{\frac{p}{p-1}(3-\tau+\eta)},
	}
where
	\eqn{
	p(\tau-2-\eta) = \frac{\tau-2-\eta}{\tau-2-2\eta}>1,
	\qquad
	\mbox{and}
	\qquad
	\frac{p}{p-1}(3-\tau+\eta) = \frac{3-\tau+\eta}{3-\tau+2\eta} <1.
	}
This means that, setting $D := p(\tau-2-\eta) - \frac{p}{p-1}(3-\tau+\eta) > 0$,
by \eqref{eq:gk},
\begin{equation}\label{eq:sesu2}
		\eqref{eq:sesu} = O\big((\log n)^c\big) \, (d_a+d_b) \, 
	\sum_{k=1}^{\bar{k}_n} (g_0)^{-Dp^{k-1}} \le
	O\big((\log n)^c\big) \, \, \frac{d_a+d_b}{(g_0)^D} \,.
\end{equation}
Since $g_0 = (\log n)^{\log\log n}$ while $d_a, d_b \le \log n$,
the right hand side of \eqref{eq:sesu2} is $o(1)$.
\end{proof}

\section{Lower bound for preferential attachment model}
\label{lowerproofs-prefatt}

In this section we prove Statements~\ref{stat-MKC}, \ref{stat-distance-boundaries}
and~\ref{stat-distance} for the preferential attachment model. By the discussion in
Section~\ref{sub-lower}, this completes the proof of the lower
bound in Theorem~\ref{main-prefatt}.

We recall that, given $m\in\N$ and $\delta \in (- m, \infty)$,
the preferential attachment model $\PA{t}$ is a random graph
with vertex set $[t] = \{1,2,\ldots, t\}$, 
where each vertex $w$ has $m$ outgoing edges, 
which are attached to vertices $v \in [w]$ with probabilities
given in \eqref{prefatt-attprob}. 
In the next subsection we give a more detailed construction 
using random variables. This equivalent reformulation will be used in a few places, when we need
to describe carefully some complicated events. However, for most of the exposition
we will stick to the intuitive description given in Section~\ref{subsect-prefatt-intro}.

\subsection{Alternative construction of the preferential attachment model}
\label{sec:altprefatt}

We introduce 
random variables $\xi_{w,j}$ 
to represent the vertex to which the $j$-th edge of vertex $w$ is attached, i.e.
\begin{equation} \label{eq:nota}
	\xi_{w,j} = v \qquad \iff \qquad w \overset{j}{\to} v \,.
\end{equation}
The graph $\PA{t}$
is a \emph{deterministic} function of these random variables: two vertices $v,w\in [t]$ with
$v \le w$ are connected in $\PA{t}$ if and only if $\xi_{w,j} = v$ for some $j \in [m]$.
In particular, the degree of a vertex $v$ after the $k$-th edge of vertex $t$ has been attached,
denoted by $D_{t,k}(v)$, is
\begin{equation}\label{eq:degree}
	D_{t,k}(v) := \sum_{(s,i) \le (t,k)} \big( \indic{\xi_{s,i} = v} +
	\indic{s = v} \big) \,,
\end{equation}
where we use the natural order relation
\begin{equation*}
	(s,i) \le (t,j) \qquad \iff \qquad s < t \quad \ \text{or} \quad \ s=t, \
	i \le j \,.
\end{equation*}

Defining the preferential attachment model amounts to giving a joint law for the
sequence $\xi = (\xi_{w,j})_{(w,j) \in \N \times [m]}$.
In agreement with \eqref{prefatt-attprob}, we set
$\xi_{1,j} = 1$ for all
$j \in [m]$, and for $t \ge 2$ 
	\begin{equation}
	\label{prefatt-attprob2}
		\pr\left( \left. \xi_{t,j} = v \, \right| \xi_{ \le (t, j-1)} \right) = 
		\begin{cases}
			\grosso \frac{D_{t,j-1}(v) + 1 +j\delta/m}
			{c_{t,j}} & \text{ if }v=t; \\
			\rule{0pt}{2em}\grosso \frac{ D_{t,j-1}(v)+\delta}
			{c_{t,j}} & \text{ if }v < t,
		\end{cases}
\end{equation}
where $\xi_{ \le (t, i-1)}$
is a shorthand for the vector $(\xi_{s,i})_{(s,i) \le (t,i-1)}$
(and we agree that $(t,0) := (t-1, m)$).
The normalizing constant $c_{t,j}$ in 
\eqref{prefatt-attprob2} is indeed given by \eqref{eq:normali}, because
by \eqref{eq:degree},
\begin{equation*}
	\sum_{v \in [t]} D_{t,j-1}(v)
	= \sum_{(s,i) \le (t,j-1)} (1+1) = 2 ((t-1)m + (j-1)) \,.
\end{equation*}

The factor $j\delta/m$ in the first line of \eqref{prefatt-attprob2}
is commonly used in the literature (instead of the possibly more natural $\delta$).
The reason is that, with such a definition, the graph $\PAmd{t}$ can be obtained from
the special case $m=1$, where every vertex has only one outgoing edge: one first
generates the random graph $\PAONE{mt}$, whose vertex set is $[mt]$,
and then collapses the block of vertices $[m(i-1)+1,m i)$ into a single vertex $i \in [t]$
(see also \cite[Chapter 8]{vdH1}).

\begin{remark}\rm\label{rem:order}
It is clear from the construction that $\PA{t}$
is a \emph{labeled directed graph}, because any
edge connecting sites $v, w$, say with $v \le w$,
carries a label $j \in [m]$ and a direction, from the newer vertex
$w$ to the older one $v$ ({\blue see} \eqref{eq:nota}). 
Even though our final result, the asymptotic behavior
of the diameter, only depends on the {\blue underlying} undirected graph,
it will be convenient to exploit the labeled directed structure of the graph in the proofs.
\end{remark}

\subsection{Proof of Statement~\ref{stat-MKC}}
\label{lower-proof-prefatt-1}

We denote by $U_{\leq k}(v)$ the $k$-neighborhood in $\PA{t}$
of a vertex $v\in[t]$, i.e.\ the set of vertices at distance at most $k$ from $v$,
viewed as a labeled directed subgraph (see Remark~\ref{rem:order}).
We denote by $D_t(v) = D_{t,m}(v)$ the degree of vertex $v$ after time $t$, i.e. in the graph
$\PA{t}$ (recall \eqref{eq:degree}).

We define the notion of \emph{minimally-$k$-connected vertex}
in analogy with the configuration model ({\blue see} Definition~\ref{def:minicm}),
up to minor technical restrictions made for later convenience.

\begin{Definition}[Minimally-$k$-connected vertex]
\label{prefatt-MKCdef}
For $k\in\N_0$, a vertex $v\in[ t] \setminus [ t/2]$ is called {\em minimally-$k$-connected} 
when $D_{ t}(v)=m$, 
all the other vertices $i\in U_{\leq k}(v)$ are in $[t/2]\setminus[t/4]$ 
and have degree $D_{t}(i)=m+1$, and 
there are no self-loops, multiple edges or cycles in $U_{\leq k}(v)$.
The graph $U_{\leq k}(v)$ 
is thus a tree with degree $m+1$, except for the root $v$ which has degree $m$.

We denote the (random) set of minimally-$k$-connected vertices by 
$\mathcal{M}_k \subseteq [t] \setminus [t/2]$,
and its cardinality by $M_k = |\mathcal{M}_k|$.
\end{Definition}

For the construction of a minimally-$k$-connected neighborhood in the preferential attachment model 
we remind that the vertices are added to the graph at different times, so that the vertex 
degrees change while the graph grows. The relevant degree for Definition~\ref{prefatt-MKCdef}
is the one at the final time $t$.
To build a minimally-$k$-connected neighborhood, we need
	\eqn{
	\label{prefatt-i_k}
  	i_k = 1+\sum_{i=1}^km^{i} = \frac{m^{k+1}-1}{m-1}
	}
many vertices. The center $v$ of the neighborhood is the youngest vertex in $U_{\leq k}(v)$, and it has degree $m$, while all the other vertices have degree $m+1$. 

\smallskip

Our first goal is to evaluate the probability $\pr(v \in \cM_k)$
that a given vertex $v \in [t] \setminus [t/2]$
is minimally-$k$-connected. 
The analogous question for the configuration model could be answered
quite easily in Proposition~\ref{cmnd-mk-moments}, 
because the configuration model can be built exploring its vertices
in an arbitrary order, in particular starting from $v$, {\blue see} 
Remark~\ref{cmnd-conditionedlaw}.
This is no longer true for the preferential attachment model,
whose vertices have an order, the chronological one,
along which the conditional probabilities take the explicit 
form~\eqref{prefatt-attprob} or \eqref{prefatt-attprob2}.
This is why the proofs for the preferential attachment model are
harder than for the configuration model.

\smallskip

As it will be clear in a moment, to get explicit formulas
it is convenient to evaluate the probability $\pr(v \in \cM_k, \, U_{\le k}(v) = H)$,
where $H$ is a fixed \emph{labeled directed} subgraph, i.e.\ it comes with the specification of which
edges are attached to which vertices. To avoid trivialities, we restrict to those $H$ for which
the probability does not vanish, i.e.\ which 
satisfy the constraints in Definition~\ref{prefatt-MKCdef},
and we call them \emph{admissible}.

Let us denote by $H^o := H \setminus \partial H$ the set of vertices in $H$ that are not on the boundary
(i.e.\ they are at distance at most $k-1$ from $v$). 
With this notation, we have the following result:

\begin{Lemma}
\label{prefatt-MKC-determnistic}
Let $\{\PA{ t}\}_{t\in\N}$ be a preferential attachment model.
For any vertex $v \in [t] \setminus [t/2]$
and any directed labeled graph $H$ which is admissible,
 	\eqn{
	\label{prob-fixed-neigh}
   	\pr\left(v\in\mathcal{M}_k,U_{\leq k}(v)=H\right) = L_1(H) \, L_2(H) \,,
 	}
where
 	\begin{eqnarray}
	\label{MKC-L_1}
   	L_1(H) 
	&:=& \prod_{u\in H^o}\prod_{j=1}^m\frac{m+\delta}{ c_{u,j}},\\
	\label{MKC-L_2}
 	L_2(H) &:=& \prod_{u\not\in H^o} \
    	\prod_{j=1}^m\left[1-\frac{ D_{u-1}(H)+|H \cap [u-1]|\delta}{ c_{u,j}}	\right],
	\end{eqnarray}
and $D_{u-1}(H) = \sum_{w \in H} D_{u-1,m}(w)$ is the total degree of $H$ before vertex
$u$ is added to the graph, 
and the normalization constant $c_{u,j}$ is defined in \eqref{eq:normali}.
\end{Lemma}

\begin{proof} 
We recall that $\{a\stackrel{i}{\rightarrow} b\}$ denotes the event that the $i$-th edge 
of $a$ is attached to $b$ ({\blue see} \eqref{eq:nota}).
Since $H$ is an admissible labeled directed subgraph, for all $u \in H^o$ and $j \in [m]$,
the $j$-th edge of $u$ is connected to a vertex in $H$, that we denote by $\theta_j^H(u)$.
We can then write
	\eqn{
	\label{prefatt-mkmom-1}
  	\{v\in\mathcal{M}_k,U_{\leq k}(v)=H\} 
	= \Big(\bigcap_{u\in H^o}\bigcap_{j=1}^m\{u\stackrel{j}{\rightarrow}\theta^H_j(u)\}\Big)\cap
	\Big(\bigcap_{u\not\in H^o} \bigcap_{j = 1}^m \{u\stackrel{j}{\not\rightarrow} H\}\Big),
	}
where of course $\{u\stackrel{j}{\not\rightarrow} H\} 
:= \bigcup_{w \not\in H} 
\{ u \overset{j}{\rightarrow} w\}$.
The first term in \eqref{prefatt-mkmom-1} is exactly the event that the edges present in $H$ 
are connected in $\PA{t}$ as they should be. The second term is the event that the vertices $u\not\in H^o$ are not attached to $H$, so that $U_{\leq k}(v)=H$. Notice that in \eqref{prefatt-mkmom-1} every vertex and every edge of the graph appears.
For a vertex $u\in H^o$, by \eqref{prefatt-attprob}
	\eqn{
	\label{prefatt-mkmom-2}
	\pr\Big(u\stackrel{j}{\rightarrow}\theta^H_j(u)\mid \PA{ u,j-1}\Big) = 
		\frac{m+\delta}{c_{u,j}},
	}
because the vertex $\theta^H_j(u)$ has degree precisely $m$ (when $u$ is not 
already present in the graph). For $u\not\in H^o$, we have to evaluate the probability that 
its edges do no attach to $H$, which is 
	\eqn{
	\label{prefatt-mkmom-3}
	\pr\Big(u\stackrel{j}{\not\rightarrow}H\mid \PA{u-1,j-1}\Big)=
		1-\frac{D_{u-1}(H)+|H \cap [u-1]|\delta}{c_{u,j}}.
	}
Using conditional expectation iteratively, we obtain \eqref{prefatt-mkmom-2} or \eqref{prefatt-mkmom-3} for every edge in the graph, depending on whether the edge is part of $H$ or not. This proves \eqref{MKC-L_1} and \eqref{MKC-L_2}.
\end{proof}

The event $\{v\in\mathcal{M}_k,U_{\leq k}(v)=H\}$ is an example of a class of events,
called \emph{factorizable}, that will be used throughout this section and Section \ref{upperproofs-prefatt}. 
For this reason we define it precisely.

It is convenient to use the random variable $\xi_{w,j}$, introduced 
in Section~\ref{sec:altprefatt}, to denote
the vertex to which the $j$-th edge of vertex $w$ is attached ({\blue see} \eqref{eq:nota}).
Any event $A$ for $\PA{t}$ can be characterized iteratively, specifying
a set $A_{s,i} \subseteq [s]$ of values for $\xi_{s,i}$,  for all $(s,i) \le (t,m)$:
\begin{equation*}
	A = \bigcap_{(s,i) \le (t,m)} \big\{ \xi_{s,i} \in A_{s,i} \big\} \,.
\end{equation*}
Of course, the set $A_{s,i}$ is allowed to
depend on the ``past'', i.e.\ $A_{s,i} = A_{s,i}\big( \xi_{\le (s,i-1)} \big)$,
or equivalently $A_{s,i} = A_{s,i}\big(\PA{s,i-1}\big)$. 
Let us set $A_{\le (s,i)} := \bigcap_{(u,j) \le (s,i)} A_{u,j}$.

\begin{Definition}[Factorizable events]\label{def:factorizable}
An event $A$ for $\PA{t}$ is called \emph{factorizable} when
the conditional probabilities
of the events $\{\xi_{s,i} \in A_{s,i}\}$, given the past, are deterministic.
More precisely, for any $(s,i)$ there is a (non-random) $p_{s,i} \in [0,1]$ such that
\begin{equation} \label{eq:nonra}
	\pr\left( \left. \xi_{s,i} \in A_{s,i} \,\right| \xi_{\le(s,i-1)} \right) 
	= p_{s,i} 
\end{equation}
on the event $\xi_{\le(s,i-1)} \in A_{\le (s,i-1)}$.
As a consequence, the chain rule for probabilities yields
\begin{equation*}
	\pr(A) = \prod_{(s,i) \le (t,m)} p_{s,i} \,.
\end{equation*}
\end{Definition}

\begin{remark}\rm\label{rem:factorizable}
Relations
\eqref{prefatt-mkmom-2} and \eqref{prefatt-mkmom-3} show that
$A = \{v\in\mathcal{M}_k,U_{\leq k}(v)=H\}$ is a factorizable event. In fact,
$A_{s,i}$ is either the single vertex $\theta^H_i(s)$ (if $s \in H^o$) or the
set $[s-1] \setminus H$ (if $s \not\in H^o$). In both cases, the set $A_{s,i} \subseteq [s-1]$
has \emph{a fixed total degree and a fixed cardinality}, hence the conditional probabilities
\eqref{eq:nonra} are specified in a deterministic way (recall \eqref{prefatt-attprob2}).

Note that the event $\{v \in \mathcal{M}_k\}$ is not factorizable. This is the reason for
specifying the realization of the $k$-neighborhood $U_{\leq k}(v)=H$.
\end{remark}

\medskip

Henceforth we fix $\varepsilon > 0$. We recall that
$k^-_n$ was defined in \eqref{kstar-tree}. Using the more customary $t$ instead of $n$,
we have
	\eqn{
	\label{kstar-MKC-prefatt}
	k^{-}_t = (1-\varepsilon)\frac{\log\log t}{\log m}.
	}
We recall that $M_{k^{-}_t}=|\mathcal{M}_{k^{-}_t}|$ denotes
the number of minimally-$k^{-}_t$-connected vertices in $\PA{t}$
({\blue see} Definition~\ref{prefatt-MKCdef}).
We can now prove half of Statement~\ref{stat-MKC}  for the preferential attachment model,
more precisely the first relation in
equation \eqref{MKC-moments}.

\begin{Proposition}[First moment of $M_{k^{-}_t}$]
\label{prefatt-prop-firstmoment}
Let $(\PA{t})_{t\geq 1}$ be a preferential attachment model, with $m\geq 2$ and $\delta\in(-m,0)$. Then, for $k^{-}_t$ as in \eqref{kstar-MKC-prefatt}, as $t\rightarrow\infty$,
	\eqn{
	\label{firstmoment-prefatt}
  	\E[M_{k^{-}_t}]\longrightarrow\infty.
	}
\end{Proposition}
\begin{proof} Similarly to the proof of \eqref{cmnd-mk-firstmoment}, we write
	\eqn{
	\label{prefatt-mkmom-4}
	\E[M_k] =  \sum_{v\in[t]\setminus[t(2]}\pr\left(v\in\mathcal{M}_k\right) 
	= \sum_{v\in[t]\setminus[t/2]}\sum_{H\subseteq[ t]\setminus[t/4]} 	
		\pr\big(v\in\mathcal{M}_k,U_{\leq k}(v)=H\big),
	}
where the sum is implicitly restricted to admissible $H$ (i.e., to $H$ that are possible
realizations of $U_{\le k}(v)$).

Since we will use \eqref{prob-fixed-neigh}, we need a lower bound on \eqref{MKC-L_1} and \eqref{MKC-L_2}. 
Recalling \eqref{eq:normali},
it is easy to show, since the number of vertices in $H^o$ equals $i_k-m^k = i_{k-1}$,
and $u \le v$ for $u \in H^o$,
	\eqn{
	\label{prefatt-mkmom-5}
	L_1(H) \geq \left[\frac{m+\delta}{v(2m+\delta)+1+\delta/m}\right]^{m i_{k-1}}.
	}
Note that for $u \le t/4$ all the factors in the product in \eqref{MKC-L_2} equal $1$,
because $H \subseteq [t] \setminus [t/4]$. Restricting to $u > t/4$ and bounding
$D_{u-1}(H)+|H \cap [u-1]|\delta \leq (m+1+\delta)i_k$, we get
	\eqn{
	\label{prefatt-mkmom-6}
	L_2(H) \geq \left[1-\frac{(m+1+\delta)i_k}{\frac{t}{4}(2m+\delta)}\right]^{3mt/4}.
	}
	
Let us write $H = \{v\} \cup H'$ where
$H'$ is a subset of $[t/2]\setminus[t/4]$ with $|H'|= i_k-1$. Clearly, for any such subset 
there is at least one 
way to order the vertices to generate an admissible $H$.
The number of possible subsets 
in $[t/2]\setminus[t/4]$ is at least $\binom{t/4}{i_k-1}$. Then, we obtain
	\eqn{
	\label{prefatt-mkmom-7}
	\E[M_k] \geq \sum_{v\in[t]\setminus[t/2]}\binom{t/4}{i_k-1}
	\left[\frac{m+\delta}{v(2m+\delta)+1+\delta/m}\right]^{m i_{k-1}}
		\left[1-\frac{(m+1+\delta)i_k}{\frac{t}{4}(2m+\delta)}\right]^{3mt/4}.
	}
Recalling that 
	\eqn{
	\binom{t/4}{i_k-1}=\frac{t^{i_k}}{4^{i_k}(i_k-1)!}(1+o(1)), 
	}
since $m i_{k-1} \le i_k$, we obtain
	\eqn{
	\label{prefatt-mkmom-8}
	\E[M_k] \geq {\blue \frac{t}{2}}\frac{t^{i_k}}{4^{i_k}(i_k-1)!}\left[\frac{m+\delta}{t(2m+\delta)+1+\delta/m}\right]^{i_{k}}
      \left[1-\frac{(m+1+\delta)i_k}{\frac{t}{4}(2m+\delta)}\right]^{3mt/4}.
	}
Choosing $k = k^{-}_t$ as in \eqref{kstar-MKC-prefatt}
and bounding $1-x \ge \e^{-2x}$ for $x$ small, as well as $m+1 \le 2m$, we obtain
	\eqn{
	\E[M_{k^{-}_t}] \geq
	{\blue \frac{t}{2}} \, \frac{t^{i_{k^{-}_t}}}{4^{i_{k^{-}_t}} \, i_{k^{-}_t}!}
	\left(\frac{m}{C \, t}\right)^{i_{k^{-}_t}}
	\mathrm{exp}\left(- 3 \, c \, m \, i_{k^{-}_t}\right)
	\ge \frac{1}{(C')^{i_{k^{-}_t}}} \,  \frac{t}{{\blue 2 \, } i_{k^{-}_t}!}
	\, \mathrm{exp}\left(- 3 \, c \, m \, i_{k^{-}_t}\right),
	}
where $C$ is a constant and $C' = 4C/m$. 
Recalling that $i_{k}$ is given by \eqref{prefatt-i_k}, and $k^{-}_t$ by  \eqref{kstar-MKC-prefatt},
hence $i_{k^{-}_t}=\frac{m}{m-1} m^{k^{-}_t}(1+o(1)) \le 2 (\log t)^{1-\varepsilon}$, hence
	\eqn{i_{k^{-}_t}! \le
	\lfloor 2(\log t)^{1-\varepsilon} \rfloor! 
	\le \left[2(\log t)^{1-\varepsilon}\right]^{2(\log t)^{1-\varepsilon}}
	= t^{o(1)} \,,
	}
and also $(C' \e^{3 \, c \, m})^{i_{k^{-}_t}} = t^{o(1)}$.
This implies that  $\E[M_k]\rightarrow\infty$, as required.
\end{proof}

\begin{remark}[Disjoint neighborhoods for minimally $k$-connected pairs]
\rm
\label{prefatt-disjneigh-remark}
We observe that, on the event $\{v,w\in\mathcal{M}_k\}$ with $v\neq w$, necessarily
	$$
  	U_{\leq k}(v)\cap U_{\leq k}(w)=\varnothing,
	$$
because if a vertex $x$ is in $U_{\leq k}(v)\cap U_{\leq k}(w)$ and $x\neq v,w$, this means 
that $D_x(t) = m+2$, because in addition to its original $m$ outgoing edges,
vertex $x$ has one incident edge from a younger vertex in
$U_{\leq k}(v)$ and one incident edge from a younger vertex in $U_{\leq k}(u)$, 
which gives a contradiction. Similar arguments apply when $x=v$ or $x=w$.
\end{remark}

We use the previous remark to prove the second relation in Statement \ref{stat-MKC} for the preferential attachment model.

\begin{Proposition}[Second moment of $M_{k^{-}_t}$]
\label{prefatt-prop-secondmoment}
 Let $(\PA{t})_{t\geq 1}$ be a preferential attachment model, with $m\geq 2$ and $\delta\in(-m,0)$. Then, for $k\in\N$,
	 \eqn{
	\label{secondmoment-prefatt}
   	\E[M_k^2] \leq \mathrm{exp}
	\left(32 m i_k^2/t\right) \E[M_k]^2+\E[M_k].
 	}
Consequently, for $k=k^{-}_t$ as in \eqref{kstar-MKC-prefatt},  as $t \to \infty$,
\begin{equation}\label{prefatt-mkmom-8.5}
	\E[M_{k^{-}_t}^2] \le (1+o(1)) \, \E[M_{k^{-}_t}]^2 \,.
\end{equation}
\end{Proposition}

\begin{proof}
We write
	\eqn{
	\label{prefatt-mkmom-9}
  	\E\left[M^2_k\right] 
	= \sum_{v,w\in[t]\setminus[t/2]}\pr\left(v,w\in\mathcal{M}_k\right) 
	= \sum_{v\neq w}\pr\left(v,w\in\mathcal{M}_k\right)
		+\E[M_k].
	}
By Remark \ref{prefatt-disjneigh-remark}, for $v\neq w$  we can write
	\eqn{
	\label{prefatt-mkmom-10}
  	\pr\left(v,w\in\mathcal{M}_k\right) 
	= \sum_{H_v\cap H_w=\varnothing}\pr\left(v,w\in\mathcal{M}_k, U_{\leq k}(v) 
	= H_v,U_{\leq k}(w) = H_w\right).
}

The crucial observation is that the event 
$\{v,w\in\mathcal{M}_k, U_{\leq k}(v) = H_v,U_{\leq k}(w) = H_w\}$ is factorizable
(recall Definition~\ref{def:factorizable} and Remark~\ref{rem:factorizable}).
More precisely, in analogy with \eqref{MKC-L_1} and \eqref{MKC-L_2}:
	\eqn{
	\label{prefatt-mkmom-11}
  	\pr\left(v,w\in\mathcal{M}_k, U_{\leq k}(v) = H_v,U_{\leq k}(w) = H_w\right) = L_1(H_v,H_w)L_2(H_v,H_w),
	}
where now
	\begin{eqnarray}
	\label{prefatt-mkmom-12}
   	L_1(H_v,H_w) 
	&=& \prod_{x\in H_v^o\cup H_w^o}\prod_{j=1}^{m}\frac{m+\delta}{c_{x,j}},\\
	\label{prefatt-mkmom-13}
  	L_2(H_v,H_w) &=& \prod_{x\not\in H_v^o\cup H_w^o}\prod_{j=1}^m\left[1-
   	 \frac{D_{x-1}(H_v \cup H_w)+|(H_v \cup H_w) \cap [x-1]|\delta}{c_{x,j}}\right].
	\end{eqnarray}
To prove \eqref{prefatt-mkmom-11}, notice that in \eqref{prefatt-mkmom-12} 
and \eqref{prefatt-mkmom-13} every edge and every vertex of the graph appear. 
Further, \eqref{prefatt-mkmom-12} is the probability of the event 
$\{U_{\leq k}(v) = H_v,U_{\leq k}(w)=H_w\}$, while \eqref{prefatt-mkmom-13} is the 
probability that all vertices not in the two neighborhoods do not attach to the two trees. 

A look at \eqref{MKC-L_1} shows that $L_1(H_v,H_w) = L_1(H_v)L_1(H_w)$. 
We now show that analogous factorization holds approximately also for $L_2$.
Since, for every $a,b\in[0,1]$, with $a+b<1$, it is true that $1-(a+b)\leq (1-a)(1-b)$, we can bound
	\eqan{
	\label{prefatt-mkmom-14}
	&\left[1-
	\frac{D_{x-1}(H_v \cup H_w)+|(H_v \cup H_w) \cap [x-1]|\delta}{c_{x,j}}\right]\\
 	&\qquad\leq\left[1-\frac{D_{x-1}(H_v)+|H_v \cap [x-1]|\delta}{c_{x,j}}\right]
		\left[1-\frac{D_{x-1}(H_w)+|H_w \cap [x-1]|\delta}{c_{x,j}}\right].\nn
	}
When we plug \eqref{prefatt-mkmom-14} into \eqref{prefatt-mkmom-13},
{\blue we obtain $L_2(H_v)L_2(H_w)$ (recall \eqref{MKC-L_2}) times the following terms:}
	\eqn{
	\label{prefatt-mkmom-15}
	\left(\prod_{x\in H_w^o}\left[1-
	\frac{D_{x-1}(H_v)+|H_v \cap [x-1]|\delta}{c_{x,j}}\right]\right)^{\blue -1}
	\left(\prod_{x\in H_v^o}\left[1-
	\frac{D_{x-1}(H_w)+|H_w \cap [x-1]|\delta }{c_{x,j}}\right]\right)^{\blue -1} .
	}
We can bound $D_{x-1}(H_v)+|H_v \cap [x-1]|\delta \le
D_{x-1}(H_v) \le (m+1)i_k$ (recall that $\delta < 0$)
and analogously for $H_w$. The square brackets in \eqref{prefatt-mkmom-15}
equal $1$ for $x \le t/4$ (since $H_v, H_w \subseteq [t] \setminus [t/4]$ by construction),
and for $x > t/4$ we have $c_{x,j} \ge \frac{t}{4}(2m+\delta) \ge \frac{m}{4} t$ by \eqref{eq:normali}
and $\delta > -m$. We can thus write
	\eqn{
	\label{prefatt-mkmom-16}
	\begin{split}
	L_2(H_v,H_w) & \le L_2(H_v) \, L_2(H_w) \,
	\prod_{x \in H_v^o\cup H_w^o} \prod_{j=1}^m
		\left[1-\frac{(m+1)i_k}{\frac{m}{4}t}\right]^{-1} \\
	& \leq L_2(H_v) \, L_2(H_w) \,
	\mathrm{exp}\left(2 (2 i_k) m \frac{(m+1)i_k}{
	\frac{m}{4}t}\right),
\end{split}
}
where we have used the bound $1-z \ge \e^{-2z}$ for small $z > 0$.
Since $m+1 \le 2m$, we obtain
	\eqn{
	\label{prefatt-mkmom-18}
	\begin{array}{l}
	\grosso \sum_{v\neq w}\left[\sum_{H_v\cap H_w= \varnothing}
	\pr\left(v,w\in\mathcal{M}_k\mbox{, }U_{\leq k}(v) = H_v\mbox{, }U_{\leq k}(w) = H_w\right)\right] \\
	\qquad\grosso \leq \mathrm{exp}\left(32 mi_k^2/t\right)
	\sum_{v\in[t]\setminus[t/2]}\sum_{H_v}L_1(H_v)L_2(H_v)
	\sum_{w\in[t]\setminus[t/2]}\sum_{H_w}L_1(H_w)L_2(H_w) \\
	\qquad\grosso  =\mathrm{exp}\left(32 mi_k^2/t\right)\E[M_k]^2.
	\end{array}
	}
Substituting \eqref{prefatt-mkmom-18} in \eqref{prefatt-mkmom-9} completes the proof of
\eqref{secondmoment-prefatt}. 

Finally, for $k=k^{-}_t$ as in \eqref{kstar-MKC-prefatt}
we have $i_{k^{-}_t} \le 2 (\log t)^{1-\varepsilon}$ (recall that $i_{k}$ is given by \eqref{prefatt-i_k}).
We have already shown in Proposition~\ref{prefatt-prop-firstmoment} that $\E[M_{k^{-}_t}] \to \infty$,
hence \eqref{prefatt-mkmom-8.5} follows.
\end{proof}
\medskip

Together, Propositions \ref{prefatt-prop-firstmoment} and \ref{prefatt-prop-secondmoment} prove 
Statement \ref{stat-MKC}. This means, as for the configuration model, since $\Var(M_{k^{-}_t}^2) = 
o(\E[M_{k^{-}_t}]^2)$, that $M_{k^{-}_t}/\E[M_{k^{-}_t}]{\xrightarrow[\sss\,t\to\infty\,]{\sss \pr} \ } 1$, so in particular $M_{k^{-}_t}{\xrightarrow[\sss\,t\to\infty\,]{\sss \pr} \ }\infty$.
\qed

\subsection{Proof of Statement \ref{stat-distance}}
\label{lower-proof-prefatt-3}
Fix $\varepsilon > 0$ and define, as in \eqref{general-k_n},
	\eqn{
	\label{prefatt-low-kt}
		\bar{k}_t = (1-\varepsilon)\frac{2\log \log t}{|\log(\tau-2)|}.
	} 
Statement~\ref{stat-distance} follows from the following result on distances 
between not too early vertices:

\begin{Proposition}[Lower bound on distances]
\label{prefatt-prop-bounddistance}
 Let $(\PA{t})_{t\geq 1}$ be a preferential attachment model, with $m\geq 2$ and $\delta\in(-m,0)$. Then, there exists a constant $p>0$ such that
	\eqn{
	\label{prefatt-bounddistance}
	\max_{x,y \ge \frac{t}{(\log t)^2}}
   	\pr\left(\dist_{\PA{t}}(x,y)\leq 2\bar{k}_t\right)\leq
			\frac{p}{(\log{t})^2}.
 	}
\end{Proposition}
Inequality \eqref{prefatt-bounddistance} is an adaptation of a result 
proved in \cite[Section 4.1]{DSMCM}. Consequently we just give a sketch of the proof
(the complete proof can be found in \longversion{Appendix~\ref{section-appendix}}\shortversion{\cite[Appendix A]{CarGarHof16ext}}).

Let us denote by $u \leftrightarrow v$ the event that vertices $u,v$
are neighbors in $\PA{t}$, that is
\begin{equation*}
	\{u \leftrightarrow v\} = \bigcup_{j=1}^m \Big( \{u \overset{j}{\to} v\} \cup
	\{v \overset{j}{\to} u\} \Big) \,.
\end{equation*}
(As a matter of fact, $\{v \overset{j}{\to} u\}$ is only possible if $v > u$,
while $\{u \overset{j}{\to} v\}$ is only possibly if $v < u$.)
Given a sequence
$\pi = (\pi_0, \pi_1, \ldots, \pi_k) \in [t]^{k+1}$ of distinct vertices,
we denote by $\{\pi\subseteq \PA{t}\}$ the event that \emph{$\pi$ is a path in $\PA{t}$}, that is
\begin{equation*}
	\big\{ \pi\subseteq \PA{t} \big\} = \{\pi_0 \leftrightarrow \pi_1 \leftrightarrow \pi_2
	\cdots \leftrightarrow \pi_k\} = \bigcap_{i=1}^k \{\pi_{i-1} \leftrightarrow \pi_i\} \,.
\end{equation*}

The proof of Proposition~\ref{prefatt-prop-bounddistance} requires the following bound on 
the probability of connection between two vertices from  \cite[Lemma 2.2]{DSvdH}: for $\gamma=m/(2m+\delta)\in (\tfrac{1}{2},1)$, there exists $c \in (0,\infty)$ such that, for all vertices $u,v\in[t]$.
	\eqn{
	\label{prefatt-connectioncondition}
		\pr\left(u\leftrightarrow v\right) \leq c(u\vee v)^{\gamma-1}(u \wedge v)^{-\gamma}.
	}
From \cite[Corollary 2.3]{DSvdH} we know, for any sequence
$\pi = (\pi_0, \pi_1, \ldots, \pi_k) \in [t]^{k+1}$ of distinct vertices,
 	\eqn{
	\label{prefatt-paths-minmax}
    	\pr\big( \pi\subseteq \PA{t} \big) \leq 
	 	p(\pi_0,\pi_1,\ldots,\pi_k)
		:= \prod_{i=0}^{k-1}
	\frac{Cm}{(\pi_i\wedge \pi_{i+1})^\gamma(\pi_i\vee 	\pi_{i+1})^{1-\gamma}},
 	}
where $C$ is an absolute constant. The history of \eqref{prefatt-paths-minmax} is that it was first proved by Bollob\'as and Riordan \cite{BolRio04b} for $\delta=0$ (so that $\gamma=1-\gamma=1/2$), and the argument was extended to all $\delta$ in \cite[Corollary 2.3]{DSvdH}.

\begin{Remark}
	Proposition \ref{prefatt-prop-bounddistance} holds for every random graphs that satisfies \eqref{prefatt-paths-minmax}.
\end{Remark}

We proceed in a similar way as in Section \ref{lower-proof-cmnd-3}. 
Given two vertices $x,y\in[t]$, we consider
paths $\pi = (\pi_0, \pi_1, \ldots, \pi_k)$ between $x = \pi_0$ and $y = \pi_k$.
We fix a decreasing sequence of numbers $(g_l)_{l\in\N_0}$ that 
serve as truncation values for the {\em age} of vertices along the path 
(rather than the degrees as for the configuration model). 
We say that a path $\pi$ is {\itshape good} when $\pi_l \geq g_l \wedge g_{k-l}$
for every $l=0, \ldots,k$, and {\itshape bad} otherwise. In other words, 
a path is good
when the age of vertices does not decrease too much from $\pi_0$ to $\pi_{k/2}$
and, backwards, from $\pi_{k}$ to $\pi_{k/2}$. Intuitively, this also means that their degrees do not grow too fast.
This means that
	\eqn{
	\label{prefatt-lowdistfix-2}
  	\pr(\dist_{\PA{t}}(x,y)\leq 2\bar{k}_t)\leq 
      	\sum_{k=1}^{2\bar{k}_t}\pr(\mathcal{E}_k(x,y))+\sum_{k=1}^{\bar{k}_t}
	\left[\pr(\mathcal{F}_k(x))+\pr(\mathcal{F}_k(y))\right],
	}
where $\mathcal{E}_k(x,y)$ is the event of there being a good path of length $k$,
as in \eqref{eq:gopa}, while $\mathcal{F}_k(x)$ is the event of there being a 
path $\pi$ with $\pi_i \ge g_i$ for $i \le k-1$ but $\pi_k < g_k$, in analogy with
\eqref{eq:bapa}.

Recalling the definition of $p(\pi_0, \pi_1, \ldots, \pi_k)$ in \eqref{prefatt-paths-minmax},
we define for $l\in \N$,
	\eqn{
	\label{prefatt-lowdistfix-4}
  	f_{l,t}(x,w) = \I_{\{x\geq g_0\}}\sum_{\pi_1=g_1}^{t}\sum_{\pi_2=g_2}^{t} 
	\cdots \sum_{\pi_{l-1}=g_{l-1}}^{t}p(x,\pi_1,\ldots,\pi_{l-1},w),
	}
setting $f_{0,t}(x,w) = \I_{\{x\geq g_0\}}$ 
and $f_{1,t}(x,w) = \I_{\{x\geq g_0\}} p(x,w)$.
From \eqref{prefatt-lowdistfix-2} we then obtain
	\eqn{
	\label{prefatt-lowdistfix-6}
	\begin{array}{rl}
  	\pr(\dist_{\PA{t}}(x,y)\leq 2\bar{k}_t)
	\leq &\displaystyle \sum_{k=1}^{2\bar{k}_t}\sum_{l=g_{\lfloor k/2\rfloor}}^{t} 
	f_{\lfloor k/2\rfloor,t}(x,l)f_{\lceil k/2\rceil,t}(y,l)\\
   	& \displaystyle +\sum_{k=1}^{\bar{k}_t}\sum_{l=1}^{g_k-1} f_{k,t}(x,l)
	+\sum_{k=1}^{\bar{k}_t}\sum_{l=1}^{g_k-1}f_{k,t}(y,l).
 	\end{array}
	}
This is the starting point of the proof of Proposition \ref{prefatt-prop-bounddistance}. 

We will show in \longversion{Appendix \ref{section-appendix}}\shortversion{\cite[Appendix A]{CarGarHof16ext}} that the following recursive bound holds
	\eqn{
	\label{prefatt-lowdistfix-11}
   	f_{k,t}(x,l)\leq \alpha_kl^{-\gamma}+\I_{\{l>g_{k-1}\}}\beta_k l^{\gamma-1},
	}
for suitable sequences $(\alpha_k)_{k\in\N}$, $(\beta_k)_{k\in\N}$
and $(g_k)_{k\in\N}$ (see \longversion{Definition~\ref{prefatt-def-sequencerecursive-app}}\shortversion{\cite[Definition~A.2]{CarGarHof16ext}}).
We will prove recursive bounds on these sequences that
guarantee that the sums in \eqref{prefatt-lowdistfix-6} 
satisfy the required bounds. 
We omit further details at this point, and refer the interested reader to
\longversion{Appendix~\ref{section-appendix}}\shortversion{\cite[Appendix A]{CarGarHof16ext}}.

\subsection{Proof of Statement \ref{stat-distance-boundaries}}
\label{lower-proof-prefatt-2}
Consider now two independent random vertices $W^t_1$ and $W^t_2$ that
are uniformly distributed
in the set of minimally-$k^{-}_t$-connected vertices $\mathcal{M}_{k^{-}_t}$.
We set
	\begin{equation}
	\label{En-PAM-def}
	E_t := \left\{ \dist\big( U_{\le k^{-}_t}(W^t_1),
	U_{\le k^{-}_t}(W^t_2)\big) \le 2 \bar{k}_t \right\} =
	\left\{\dist(W_1^t,W_2^t)\leq 2k^{-}_t+2\bar{k}_t\right\} 
	\end{equation}
and, in analogy with Section~\ref{lower-proof-cmnd-2}, our goal is to show that
	\begin{equation}
	\label{eq:PAM-goalst2}
	\lim_{t\to\infty} 
	\pr(E_t) = 0 .
	\end{equation}

We know from Statement~\ref{stat-MKC} that, as $t\to\infty$,
	\eqn{
	\label{concl-PAM-split-opiccolo}
	\pr\left( M_{k^{-}_t} \le \frac{1}{2}\E[M_{k^{-}_t}] \right) \le
	\pr\left( |M_{k^{-}_t}-\E[M_{k^{-}_t}]|> \frac{1}{2}\E[M_{k^{-}_t}] \right)
	\le \frac{\Var(M_{k^{-}_t})}{\frac{1}{4} \E[M_{k^{-}_t}]^2} = o(1).
	}
We also define the event
\eqn{
	\label{eq-PAM-deg}
	B_t := \left\{\max_{v\in[t]}D_t(v)\leq \sqrt{t}\right\} 
}
and note that
it is known (see \cite[Theorem 8.13]{vdH1}) 
that $\lim_{t\rightarrow\infty}\pr(B_t)=1$. 
Therefore,
	\begin{equation}
	\label{concl-PAM-sumtobound}
	\begin{split}
	\pr(E_t) & = \pr\left(E_t\cap\{M_{k^{-}_t} > \tfrac{1}{2}\E[M_{k^{-}_t}]\}\cap B_t\right)
	+ o(1) \\
	& = \E\Big[\sum_{v_1,v_2\in [t]}
	\I_{\{W_1^t=v_1,W_2^t=v_2\}}\I_{\{\dist(v_1,v_2)\leq 2k^{-}_t+2\bar{k}_t\}}
	\I_{\{M_{k^{-}_t} > \frac{1}{2}\E[M_{k^{-}_t}]\}}\I_{B_t}\Big]+ o(1) \\
	& \leq \E\Big[\sum_{v_1,v_2\in [t]  \setminus [t/2]}
	\frac{\I_{\{v_1\in \mathcal{M}_{k^{-}_t},v_2\in \mathcal{M}_{k^{-}_t}\}}}
	{M_{k^{-}_t}^2} \I_{\{\dist(v_1,v_2)\leq 2k^{-}_t+2\bar{k}_t\}}
			\I_{\{M_{k^{-}_t} > \frac{1}{2}\E[M_{k^{-}_t}]\}}\I_{B_t}\Big]+ o(1) \\
 	& \leq \sum_{v_1,v_2\in[t]\setminus [t/2]}
	\frac{\pr\left(v_1,v_2\in \mathcal{M}_{k^{-}_t}, \, 
	\dist(v_1,v_2)\leq 2k^{-}_t+2\bar{k}_t, \, B_t \right)}
	{\frac{1}{4}\E[M_{k^{-}_t}]^2} + o(1).
	\end{split}
	\end{equation}
The contribution of the terms with $v_1 = v_2$ is negligible,
since it gives
\begin{equation*}
	\frac{\sum_{v_1 \in [t] \setminus [t/2]} \pr\left(v_1 \in \mathcal{M}_{k^{-}_t}\right)}
	{\frac{1}{4}\E[M_{k^{-}_t}]^2} = \frac{4}{\E[M_{k^{-}_t}]}
	= o(1) ,
\end{equation*}
because $\E[M_{k^{-}_t}] \to \infty$ by Proposition~\ref{prefatt-prop-firstmoment}.
Henceforth we restrict the sum in \eqref{concl-PAM-sumtobound} to $v_1 \ne v_2$.
Summing over the realizations $H_1$ and $H_2$ of the random neighborhoods
$U_{\leq k^{-}_t}(v_1)$ and $U_{\leq k^{-}_t}(v_2)$, and over paths $\pi$ from an arbitrary vertex
$x \in \partial H_1$ to an arbitrary vertex $y \in \partial H_2$, we obtain
	\eqn{	\label{concl-prefatt-sum}
\begin{split}
	\pr(E_t) \le
	\frac{4}{\E[M_{k^{-}_t}]^2} & \sum_{\substack{v_1,v_2\in[t] \setminus [t/2] \\ v_1 \ne v_2}}
	\ \sum_{H_1,\, H_2 \subseteq [t] \setminus [t/4]} 
	\ \sum_{x \in \partial H_1,\, y \in \partial H_2} 
	\ \sum_{\substack{\pi: x \to y \\
	|\pi| \le 2 \bar{k}_t}}\\
	& \quad \pr\left(
	U_{\leq k^{-}_t}(v_1)=H_1, \, U_{\leq k^{-}_t}(v_2)
	=H_2, \, \pi \subseteq \PA{t}, \, B_t\right) \ + \ o(1) .
\end{split}
	}
The next proposition, proved below, decouples the probability appearing in the last expression:

\begin{Proposition}
\label{pr-PAM-boundsplit}
There is a constant $q \in (1,\infty)$ such that,
for all $v_1, v_2$, 
$H_1, H_2$ and $\pi$, 
\eqn{\label{eq:usbb}
\begin{split}
	&\pr\left(
	U_{\leq k^{-}_t}(v_1)=H_1, \, U_{\leq k^{-}_t}(v_2)
	=H_2, \, \pi \subseteq \PA{t}, \, B_t\right)\\
	&\leq q\, \pr\left(U_{\leq k^{-}_t}(v_1)=H_1,~U_{\leq k^{-}_t}(v_2)=H_2\right)
	\pr\left(\pi \subseteq \PA{t}\right).
\end{split}
}
\end{Proposition}

The proof of Proposition \ref{pr-PAM-boundsplit} reveals that we can take $q=2$ for $t$ sufficiently large.
Using \eqref{eq:usbb} in \eqref{concl-prefatt-sum}, we obtain
	\begin{equation}
\begin{split}
	\label{concl-prefatt-sumsplit}
	\pr(E_t) \le
 	\frac{4q}{\E[M_{k^{-}_t}]^2}\sum_{v_1,v_2\in[t] \setminus [t/2]} \
	& \sum_{H_1,\, H_2 \subseteq [t] \setminus [t/4]} 
	\pr\left(U_{\leq k}(v_1)= H_1\mbox{, }U_{\leq k}(v_2)= H_2\right)\\
	& \times \Bigg\{
	\sum_{x \in \partial H_1,\, y \in \partial H_2} 
	\ \sum_{\substack{\pi: x \to y \\
	|\pi| \le 2 \bar{k}_t}}
	\pr\left(\pi \subseteq \PA{t}\right) \Bigg\} .
\end{split}
	\end{equation}
If we bound $\pr\left(\pi \subseteq \PA{t}\right) \le p(\pi)$ in \eqref{concl-prefatt-sumsplit},
as in \eqref{prefatt-paths-minmax},
\emph{the sum over $\pi$
can be rewritten as the right hand side of \eqref{prefatt-lowdistfix-6}}
(recall \eqref{prefatt-lowdistfix-2}-\eqref{prefatt-lowdistfix-4}).
We can thus apply Proposition~\ref{prefatt-prop-bounddistance} ---because the proof of Proposition \ref{prefatt-prop-bounddistance} really gives a bound on \eqref{prefatt-lowdistfix-6}--- concluding
that the sum over $\pi$ is at most
$p / (\log t)^2$, where the constant $p$ is defined in Proposition \ref{prefatt-prop-bounddistance}.
Since $|\partial H_1| = |\partial H_2| = m^{k^{-}_t} = (\log t)^{1-\varepsilon}$
(recall \eqref{kstar-MKC-prefatt}), we finally obtain
	\eqn{
	\label{concl-prefatt-penultima}
	\pr(E_t) \le
		\frac{4q}{\E[M_{k^{-}_t}]^2} \,
		\frac{p (\log t)^{2(1-\varepsilon)}}{(\log{t})^2} \,
	\E[M_{k^{-}_t}^2] 
	= \big(1+o(1)\big) \frac{4pq}{(\log t)^{2\varepsilon}} ,
	}
where the last step uses Proposition~\ref{prefatt-prop-secondmoment}.
This completes the proof that $\pr(E_t) = o(1)$.
\qed

\begin{proof}[Proof of Proposition \ref{pr-PAM-boundsplit}]
We recall that $H_1 \subseteq [t] \setminus [t/4]$ is a labeled directed subgraph
containing $v_1$, such that it is an admissible realization of the neighborhood
$U_{\le k^{-}_t}(v_1)$ of the minimally-$k^{-}_t$-connected vertex $v_1$
(recall Definition~\ref{prefatt-MKCdef}); in particular, $H_1 \setminus \{v_1\} 
\subseteq [t/2] \setminus [t/4]$.
We also recall that, for all $u \in H_1^o := H_1 \setminus \partial H_1$ and $j \in [m]$,
the $j$-th edge of $u$ is connected to a well specified vertex in $H_1$, denoted by $\theta_j^{H_1}(u)$.
Analogous considerations apply to $H_2$.

We have to bound the probability
\eqn{\label{eq:nofa}
	\pr\left(U_{\leq k^{-}_t}(v_1)=H_1,~U_{\leq k^{-}_t}(v_2)=H_2,~ \pi\subseteq \PA{t},~ B_t\right),
}
where $\pi = (\pi_0, \pi_1, \ldots, \pi_k) \in [t]^{k+1}$ is a given sequence of vertices
with $\pi_0 \in \partial H_1$ and $\pi_k \in \partial H_2$.
The event in \eqref{eq:nofa} is not factorizable, because the degrees of the vertices
in the path $\pi$ are not specified, hence it is not easy to evaluate
its probability. To get a factorizable event, we need to give more information.
For a vertex $v \in [t]$, define its \emph{incoming neighborhood} $\mathcal{N}(v)$ by
\begin{equation} \label{eq:inco}
	\cN(v) := \{(u,j) \in [t] \times [m]: \ u \overset{j}{\to} v\} \,.
\end{equation}
The key observation is that \emph{the knowledge of $\cN(v)$ determines the degree $D_s(v)$
at any time $s \le t$} (for instance, at time $t$ we simply have $D_t(v) = |\cN(v)| + m$).

We are going to fix the incoming neghborhoods $\cN(\pi_1) = K_1$,
\ldots, $\cN(\pi_{k-1}) = K_{k-1}$ of all vertices in the path $\pi$,
except the extreme ones $\pi_0$ and $\pi_k$ (note that $\cN(\pi_0)$
and $\cN(\pi_k)$ reduce to single points in $H_1^o$ and $H_2^o$, respectively,
because $\pi_0 \in \partial H_1$ and $\pi_k \in \partial H_2$).
We emphasize that such incoming neighborhoods
allow us to determine whether $\pi = (\pi_0, \ldots, \pi_k)$ is a path in $\PA{t}$.
Recalling the definition of the event $B_t$ in \eqref{eq-PAM-deg},
we restrict to
\begin{equation}\label{eq:Ksqrt}
	|K_i| \le \sqrt{t}, \qquad \text{for }i\in[k-1],
\end{equation}
and simply drop $B_t$ from \eqref{eq:nofa}. 
We will then prove the following relation:
for all $v_1, v_2$, $H_1, H_2$, $\pi = (\pi_0, \ldots, \pi_k)$,
and for all $K_1, \ldots, K_{k-1}$ satisfying \eqref{eq:Ksqrt}, we have
\begin{equation}\label{eq:usbb2}
\begin{split}
	&\pr\left(
	U_{\leq k^{-}_t}(v_1)=H_1, \ U_{\leq k^{-}_t}(v_2)
	=H_2, \ \{ \cN(\pi_1) = K_1, \, \ldots, \cN(\pi_{k-1}) = K_{k-1} \} \right)\\
	&\leq q\, \pr\left(U_{\leq k^{-}_t}(v_1)=H_1,~U_{\leq k^{-}_t}(v_2)=H_2\right)
	\pr\left(\cN(\pi_1) = K_1, \, \ldots, \cN(\pi_{k-1}) = K_{k-1}\right).
\end{split}
\end{equation}
Our goal \eqref{eq:usbb} follows by summing this relation
over all 
$K_1, \ldots, K_{k-1}$ 
for which $\pi \subseteq \PA{t}$.

The first line of \eqref{eq:usbb2} is the probability
of a factorizable event. In fact, setting for short
\begin{equation*}
	R \; := \; \big( H_1^o \times [m]\big) \; \cup \; 
	\big( H_2^o \times [m] \big) 
	\; \cup \; K_1 \; \cup \; \ldots \; \cup \; K_{k-1} \,,
\end{equation*}
the event in the first line of \eqref{eq:usbb2} is the intersection
of the following four events ({\blue see} \eqref{prefatt-mkmom-1}):
\begin{gather*}
	\bigcap_{u \in H_1^o} \bigcap_{j=1}^m \{u \overset{j}{\to} \theta^{H_1}_j(u) \} \,,
	\qquad
	\bigcap_{u \in H_2^o} \bigcap_{j=1}^m \{u \overset{j}{\to} \theta^{H_2}_j(u) \} \,,
	\qquad
	\bigcap_{i=1}^{k-1} \bigcap_{(u,j) \in K_i} \{u \overset{j}{\to} \pi_i\}, \\
	\bigcap_{(u,j) \in [t] \times [m] \, \setminus \, R} \{u \overset{j}{\not\to} 
	(H_1 \cup H_2 \cup \pi^o) \} \,,
\end{gather*}
where we set $\pi^o := \pi \setminus \{\pi_0, \pi_k\} = (\pi_1, \ldots, \pi_{k-1})$.
Generalizing \eqref{prefatt-mkmom-2}-\eqref{prefatt-mkmom-3},
we can rewrite the first line of \eqref{eq:usbb2} as follows, recalling \eqref{prefatt-attprob}:
\begin{equation} \label{eq:nogat}
\begin{split}
	& \pr\left(
	U_{\leq k^{-}_t}(v_1)=H_1, \ U_{\leq k^{-}_t}(v_2)
	=H_2, \ \{ \cN(\pi_1) = K_1, \, \ldots, \cN(\pi_{k-1}) = K_{k-1} \} \right) \\
	&\quad = \Bigg\{ \prod_{u \in H_1^o} \prod_{j=1}^m \frac{m+\delta}{c_{u,j}} \Bigg\}
	\Bigg\{ \prod_{u \in H_2^o} \prod_{j=1}^m \frac{m+\delta}{c_{u,j}} \Bigg\}
	\Bigg\{ \prod_{i=1}^{k-1} \prod_{(u,j) \in K_i} \frac{D_{u,j-1}(\pi_i)+\delta}{c_{u,j}} \Bigg\} \\
	& \quad \quad \ \ \Bigg\{ \prod_{(u,j) \in [t]\times [m] \;\setminus\; R}
	\bigg(1- \frac{D_{u,j-1}(H_1 \cup H_2 \cup \pi^o)
	+ |(H_1 \cup H_2 \cup \pi^o) \cap [u-1]| \delta}{c_{u,j}} \bigg) \Bigg\} .
\end{split}
\end{equation}
We stress that $D_{u,j-1}(\pi_i)$ is {\em non-random}, because it is determined
by $K_i$. Analogous considerations apply to $D_{u,j-1}(H_1 \cup H_2 \cup \pi^o)$.
We have thus obtained a factorizable event.

Next we evaluate the second line of \eqref{eq:usbb2}.
Looking back at \eqref{prefatt-mkmom-11}-\eqref{prefatt-mkmom-13}, we have
\begin{equation} \label{eq:comp1}
\begin{split}
	& \pr\left(U_{\leq k^{-}_t}(v_1)=H_1,~U_{\leq k^{-}_t}(v_2)=H_2\right) =
	\Bigg\{ \prod_{u \in H_1^o} \prod_{j=1}^m \frac{m+\delta}{c_{u,j}} \Bigg\}
	\Bigg\{ \prod_{u \in H_2^o} \prod_{j=1}^m \frac{m+\delta}{c_{u,j}} \Bigg\} \\
	& \qquad\qquad 
	\Bigg\{ \prod_{(u,j) \in [t]\times [m] \;\setminus\; (H_1^o \cup H_2^o) \times [m]}
	\bigg(1- \frac{D_{u,j-1}(H_1 \cup H_2)
	+ |(H_1 \cup H_2) \cap [u-1]| \delta}{c_{u,j}} \bigg) \Bigg\} \,.
\end{split}
\end{equation}
On the other hand,
\begin{equation} \label{eq:comp2}
\begin{split}
	& \pr\left(\cN(\pi_1) = K_1, \, \ldots, \cN(\pi_{k-1}) = K_{k-1}\right) =
	\Bigg\{ \prod_{i=1}^{k-1} \prod_{(u,j) \in K_i} \frac{D_{u,j-1}(\pi_i)+\delta}{c_{u,j}} \Bigg\} \\
	& \qquad \Bigg\{ \prod_{(u,j) \in [t]\times [m] \;\setminus\; K_1 \cup \ldots \cup K_{k-1}}
	\bigg(1- \frac{D_{u,j-1}(\pi^o)
	+ |\pi^o \cap [u-1]| \delta}{c_{u,j}} \bigg) \Bigg\} .
\end{split}
\end{equation}
Using the bound $(1-(a+b)) \le (1-a)(1-b)$ in the second line of \eqref{eq:nogat},
and comparing with \eqref{eq:comp1}-\eqref{eq:comp2}, we only need
to take into account the missing terms in the product in the last lines.
This shows that relation \eqref{eq:usbb2} holds if one sets $q = C_1 \, C_2$ therein, where
\begin{gather*}
	C_1 := \Bigg\{ \prod_{(u,j) \in K_1 \; \cup \; \ldots \; \cup \; K_{k-1}}
	\bigg(1- \frac{D_{u,j-1}(H_1 \cup H_2)
	+ |(H_1 \cup H_2) \cap [u-1]| \delta}{c_{u,j}} \bigg) \Bigg\}^{-1} , \\
	C_2 := \Bigg\{ \prod_{(u,j) \in (H_1^o \cup H_2^o) \times [m]}
	\bigg(1- \frac{D_{u,j-1}(\pi^o)
	+ |\pi^o \cap [u-1]| \delta}{c_{u,j}} \bigg) \Bigg\}^{-1} .
\end{gather*}

To complete the proof, it is enough to give uniform upper bounds on $C_1$ and $C_2$,
that does not depend on $H_1$, $H_2$, $\pi$. We start with $C_1$.
In the product we may assume $u > t/4$, because the terms with
$u \le t/4$ are identically one, since $H_1, H_2 \subseteq [t] \setminus [t/4]$.
Moreover, for $u > t/4$ we have $c_{u,j} \ge t(2m+\delta)/4 \ge mt/4$
by \eqref{eq:normali} and $\delta > -m$. Since $D_{u,j-1}(H_1 \cup H_2) \le 2(m+1) i_k$,
using $1-x \ge \e^{-2x}$ for $x$ small and recalling that $\delta < 0$, it follows that
\eqn{
\begin{split}
	C_1^{-1}& \ge \prod_{(u,j) \in K_1 \; \cup \; \ldots \; \cup \; K_{k-1}} \left(
	1-\frac{2(m+1) i_k}{\frac{m}{4}t}\right)
	\ge \e^{-\frac{8(m+1)}{tm} |K_{[k-1]} | i_k} ,
\end{split}
}
where $K_{\sss [k-1]}=K_1\; \cup \; \ldots \; \cup \; K_{k-1}$.
Since $i_{k}$ is given by \eqref{prefatt-i_k}, for $k = k^{-}_t$ as in \eqref{kstar-MKC-prefatt}
we have $i_{k}=\frac{m}{m-1} m^{k^{-}_t}(1+o(1)) \le 2 (\log t)^{1-\varepsilon}$.
Recalling also \eqref{eq:Ksqrt} and bounding $m+1\le 2m$, we obtain
\begin{equation*}
	C_1 \le \e^{\frac{8(m+1)}{tm} |K_{\sss [k-1]}| i_k} 
	\le \e^{16k \, i_{k}/\sqrt{t}}  = \e^{O(\log t/\sqrt{t})} = 1 + o(1) \,.
\end{equation*}
For $C_2$, since $D_{u,j-1}(\pi^o) \le D_{t}(\pi^o) = |K_{\sss [k-1]}|
\le k \sqrt{t}$, again by \eqref{eq:Ksqrt}, we get
\eqn{
\begin{split}
	C_2^{-1} &\ge \prod_{(u,j) \in (H_1^o \cup H_2^o) \times [m]}
	\left(1-\frac{k \sqrt{t}}{\frac{m}{4} t}\right) 
	\ge \e^{- \frac{8}{m} \, \frac{k}{\sqrt{t}} \, |H_1^o \cup H_2^o| m}
	\ge \e^{- 16 \, k \, i_k/\sqrt{t}} = 1 - o(1) \,.
\end{split}
}
It follows that $C_1C_2$ is bounded from above by some constant $q$. This completes the proof.
\end{proof}

\subsection{Proof of Theorem \ref{typ-dis-prefatt}}

Dereich, M\"onch and M\"orters \cite{DSMCM} have already proved the upper bound. 
For the lower bound we use Proposition~\ref{prefatt-prop-bounddistance}. In fact, for $\bar{k}_t$ as in \eqref{prefatt-low-kt},
	\eqn{
	\label{prefatt-typ-1}
	\pr\left(H_t\leq 2\bar{k}_t\right) = \sum_{v_1,v_2\in [t]}\pr\left(V_1=v_1,V_2=v_2,\dist(v_1,v_2)\leq 2\bar{k}_t\right).
	}
If $v_1$ and $v_2$ are both larger or equal than $g_0=\lceil\frac{t}{(\log t)^2}\rceil$, 
then we can apply Proposition \ref{prefatt-prop-bounddistance}. The probability that $V_1<g_0$ or $V_2<g_0$ is
	\eqn{
	\label{prefatt-typ-2}
	\pr\left(\{V_1<g_0\}\cup\{V_2<g_0\}\right) \leq 2g_0/t = o(1),
	}
hence we get
	\eqan{
	\label{prefatt-typ-3}
	&\frac{1}{t^2}\sum_{v_1,v_2\in [t]\setminus[g_0]}
	\pr\left(\dist(v_1,v_2)\leq 2\bar{k}_t\right)+o(1)
	\leq \frac{(t-g_0)^2}{t^2}\frac{p}{(\log{t})^2}+o(1)=o(1),\nn
	}
and this completes the proof of Theorem \ref{typ-dis-prefatt}.
\qed

\section{Upper bound for configuration model}
\label{upperproofs-cmnd}

In this section we prove Statements~\ref{stat-collision} and~\ref{stat-triplelog}
for the configuration model. By the discussion in
Section~\ref{sub-upper}, this completes the proof of the upper
bound in Theorem~\ref{main-cmnd}, because the proof
of Statement~\ref{stat-core} is already known in the literature, 
as explained below Statement~\ref{stat-core}.

Throughout this section, the assumptions of Theorem~\ref{main-cmnd}
apply. In particular, we work on a configuration model
$\CMnd$, with $\tau\in(2,3)$ and $\dmin\geq3$.

\subsection{Proof of Statement \ref{stat-collision}}
\label{proof-upper-cmnd-1}

We first recall what $\core_n$ is, and define the $k$-exploration graph.

Recall from \eqref{core_n-def} that, for $\CMnd$, $\core_n$ is defined as
  $$
	  \core_n = \left\{i\in[n]\mbox{ such that }d_i > (\log n)^\sigma\right\},
  $$
where $\sigma>1/(3-\tau)$. Since the degrees $d_i$ are fixed in the configuration model,
$\core_n$ is a deterministic subset.

For any $v \in [n]$, we recall that
$U_{\leq k}(v) \subseteq [n]$ denotes the subgraph of $\CMnd$ consisting of
the vertices at distance at most $k$ from $v$. 
We next consider the \emph{$k$-exploration graph} $\widehat U_{\leq k}(v)$
as a modification of $U_{\leq k}(v)$, where we only explore $\dmin$
half-edges of the starting vertex $v$, and only $\dmin - 1$
for the following vertices:

\begin{Definition}[$k$-exploration graph in $\CMnd$]
\label{cmnd-def-kexp} 
The {\em $k$-exploration graph} of a vertex $v$ is the subgraph $\widehat U_{\leq k}(v)$ 
built iteratively as follows: 

\begin{itemize}
\item[$\rhd$] Starting from $\widehat U_{\leq 0}(v) = \{v\}$, we consider the first $\dmin$ 
half-edges of $v$ and we pair {\blue them, one by one,}
to a uniformly chosen {\blue unpaired} half-edge (see Remark~\ref{cmnd-conditionedlaw}),
to obtain $\widehat U_{\leq 1}(v)$.

\item[$\rhd$] Assume that we have built $\widehat U_{\leq \ell}(v)$, for $\ell \ge 1$,
and set $\widehat U_{=\ell}(v) 
:= \widehat U_{\leq \ell}(v) \setminus \widehat U_{\leq (\ell-1)}(v)$.
For each vertex in $\widehat U_{=\ell}(v)$, we consider
the first $\dmin - 1$ \emph{unpaired} half-edges 
and we pair {\blue them, one by one,} to a uniformly chosen {\blue unpaired} half-edge,
to obtain $\widehat U_{\leq (\ell + 1)}(v)$.
(Note that, by construction, each vertex in $\widehat U_{=\ell}(v)$
has \emph{at least} one already paired half-edge.)

\end{itemize}
\end{Definition}

\begin{Definition}[Collision]
\label{cmnd-def-coll}
In the process of building the $k$-exploration graph $\widehat U_{\leq k}(v)$, we say that there is a
\emph{collision} when a half-edge is paired to a vertex already included in the $k$-exploration graph.
\end{Definition}

We now prove Statement~\ref{stat-collision}. Let us fix $\varepsilon > 0$ and set
	\eqn{
	\label{cmnd-kstar-expl}
	k^{+}_n = \left(1+\varepsilon\right)\frac{\log\log n}{\log(\dmin-1)}.
	}

\begin{Proposition}[At most one collision]
\label{cmnd-prop-collision}
Under the assumption of Theorem~\ref{main-cmnd}, the following holds with high probability:
the $k^{+}_n$-exploration graph of \emph{every} vertex either intersects $\core_n$,
or it has at most {\blue one} collision.
\end{Proposition}

\begin{proof}
Let us fix a vertex $v \in [n]$. We are going to estimate the probability
\begin{equation*}
	q_n(v) := \pr\Big(\text{there are at least $2$ collisions in $\widehat U_{\le k^{+}_n}(v)$
	and }\widehat U_{\le k^{+}_n}(v) \cap \core_n = \varnothing \Big) \,.
\end{equation*}
If we show that $\sup_{v\in [n]} q_n(v) = o(1/n)$, then it follows that
$\sum_{v \in [n]} q_n(v) = o(1)$, completing the proof.

Starting from the vertex $v$, we pair
successively one half-edge after the other,
as described in Definition~\ref{cmnd-def-kexp}
(recall also Remark~\ref{cmnd-conditionedlaw}).
In order to build $\widehat U_{\le k^{+}_n}(v)$,
{\blue we need to make a number of pairings,
denoted by  $\cN$, which is \emph{random}},
because collisions may occur.
In fact, when there are no collisions, $\cN$ is deterministic and takes its maximal value
given by $i_{k^{+}_n}$ in \eqref{i-k}, therefore
\begin{equation} \label{eq:boundcN}
	\cN \le i_{k^{+}_n}
	\le \frac{\dmin}{\dmin-2} (\dmin - 1)^{k^{+}_n} \le 3 \, (\log n)^{1+\varepsilon} \,.
\end{equation}
Introducing the event $C_i :=$ ``there is a collision when pairing
the $i$-th half-edge'', we can write
\begin{equation} \label{eq:luba1}
\begin{split}
	q_n(v) & \le \E \Bigg[ \sum_{1 \le i < j \le \cN} \I_{\{ C_i, \ C_j, \
	\widehat U_{\le k^{+}_n}(v) \cap \core_n = \varnothing \}} \Bigg] \\
	& = \sum_{1 \le i < j \le 3 (\log n)^{1+\varepsilon}} \pr \big( C_i, \ C_j, \
	j \le \cN , \
	\widehat U_{\le k^{+}_n}(v) \cap \core_n = \varnothing \big) \,.
\end{split}
\end{equation}

Let $E_\ell$ be the event that the first $\ell$ half-edges are paired to vertices
with degree $\le (\log n)^\sigma$ (i.e., the graph obtained after
pairing the first $\ell$ half-edges is disjoint from $\core_n$). Then
\begin{equation}\label{eq:luba2}
\begin{split}
	\pr \big( C_i, \ C_j, \
	j \le \cN , \
	\widehat U_{\le k^{+}_n}(v) \cap \core_n = \varnothing \big)
	& \le \pr \big( C_i, \ C_j, \ E_{j-1}\big) \\
	& = \pr ( E_{i-1} ) \, \pr( C_i \,|\, E_{i-1}) \, \pr( C_j \,|\, C_i, \, E_{j-1} ) \,.
\end{split}
\end{equation}
On the event $E_{i-1}$, before pairing the $i$-th half-edge, the graph is composed by at most $i-1$
vertices, each with degree at most $(\log n)^\sigma$, hence, for $i \le 3 (\log n)^{1+\varepsilon}$,
\begin{equation*}
	\pr( C_i \,|\, E_{i-1}) \le
	\frac{(i-1) (\log n)^\sigma}{\ell_n - 2i + 1} \le 
	\frac{3 (\log n)^{1+\varepsilon} (\log n)^\sigma}
	{\ell_n - 6 (\log n)^{1+\varepsilon}}
	\le c \, \frac{(\log n)^{\sigma + 1 + \varepsilon}}{n} \,,
\end{equation*}
for some $c \in (0,\infty)$, thanks to $\ell_n=n \mu(1+o(1))$ (recall \eqref{eq:nel}).
The same arguments show that 
\begin{equation*}
	\pr( C_j \,|\, C_i, \, E_{j-1} ) \le c \, \frac{(\log n)^{\sigma + 1 + \varepsilon}}{n} \,.
\end{equation*}
Looking back at \eqref{eq:luba1}-\eqref{eq:luba2}, we obtain
\begin{equation*}
	\sup_{v\in [n]} q_n(v) \le \sum_{1 \le i < j \le 3 (\log n)^{1+\varepsilon}} 
	c^2 \, \frac{(\log n)^{2(\sigma + 1 + \varepsilon)}}{n^2}
	\le 9 \, c^2 \, \frac{(\log n)^{2\sigma + 4(1 + \varepsilon)}}{n^2}
	= o\bigg(\frac{1}{n}\bigg) \,,
\end{equation*}
which completes the proof.
\end{proof}

\begin{Corollary}[Large boundaries]
\label{cmnd-cor-boundryexpl}
Under the assumptions of Theorem~\ref{main-cmnd} and on the event $\widehat U_{\leq k^{+}_n}(v)\cap \core_n=\varnothing$, with high probability,
the boundary $\widehat U_{= k^{+}_n}(v)$ of the ${k^{+}_n}$-exploration graph
of any vertex $v\in[n]$ 
contains at least $(\dmin - 2)(\dmin-1)^{k^{+}_n-1} \ge \frac{1}{2} (\log n)^{1+\varepsilon}$ vertices,
each one with at least two unpaired half-edges.
\end{Corollary}

\begin{proof} 
By Proposition~\ref{cmnd-prop-collision}, with high probability, 
every ${k^{+}_n}$-exploration graph has at most one collision before hitting $\core_n$. 
The worst case is when the collision happens immediately, i.e.\
a half-edge incident to $v$ is paired to another half-edge incident to $v$: in this case,
removing both half-edges, the $k^{+}_n$-exploration
graph becomes a tree with $(\dmin - 2) (\dmin - 1)^{k^{+}_n - 1}$ vertices
on its boundary,
each of which has at least $(\dmin - 1) \ge 2$ yet unpaired half-edges.
Since $(\dmin - 2)/(\dmin - 1) \ge \frac{1}{2}$ for $\dmin \ge 3$,
and moreover $(\dmin - 1)^{k^{+}_n} = (\log n)^{1+\varepsilon}$ by \eqref{cmnd-kstar-expl},
we obtain the claimed bound.

If the collision happens at a later stage, i.e.\ for a half-edge incident
to a vertex different from the starting vertex $v$, then we just remove the branch from $v$
to that vertex, getting a tree with $(\dmin - 1) (\dmin - 1)^{k^{+}_n - 1} $ vertices
on its boundary. The conclusion follows.
\end{proof}

Together, Proposition \ref{cmnd-prop-collision} and Corollary \ref{cmnd-cor-boundryexpl} prove Statement \ref{stat-collision}.
\qed

\subsection{Proof of Statement \ref{stat-triplelog}}
\label{proof-upper-cmnd-2}
Consider the ${k^{+}_n}$-exploration graph
$\widehat U =
\widehat U_{\le k^{+}_n}(v)$ of a fixed vertex $v\in[n]$, as in Definition~\ref{cmnd-def-kexp},
and let $x_1,\ldots,x_N$ be the (random) vertices on its boundary.
We stress that, by Corollary \ref{cmnd-cor-boundryexpl}, with high probability
$N\geq \frac{1}{2}(\log n)^{1+\varepsilon}$. Set
	\eqn{
	\label{cmnd-hlog}
  	h_n = \big\lceil B\log\log\log n +\, C \big\rceil \,,
	}
where $B, C$ are fixed constants, to be determined later on. 

Henceforth we fix a realization $H$ of $\widehat U = \widehat U_{\le k^{+}_n}(v)$
and we work \emph{conditionally} on the event $\{\widehat U =H\}$.
By Remark~\ref{cmnd-conditionedlaw}, we can complete the construction
of the configuration model $\CMnd$ by pairing uniformly all the yet unpaired half-edges.
We do this as follows: for each vertex $x_1,\ldots,x_N$
on the boundary of $\widehat U$, we explore
its neighborhood, looking for \emph{fresh} vertices with higher and higher degree, up
to distance $h_n$ (we call a vertex \emph{fresh} if it is connected
to the graph for the first time, hence it only has one paired half-edge).
We now describe this procedure in detail: 

\begin{Definition}[Exploration procedure]
\label{cmnd-def-expl}
Let $x_1, \ldots, x_N$ denote the vertices on the boundary of
a $k^{+}_n$-exploration graph $\widehat U = \widehat U_{\le k^{+}_n}(v)$.
We start the exploration procedure from $x_1$.
\begin{itemize}
\item[$\rhd$] Step 1. We set $v_0^{\sss(1)} := x_1$ and we pair all its unpaired half-edges.
\emph{Among the fresh vertices} to which $v_0^{\sss(1)}$ has been connected,
we call $v_1$ the one with maximal degree.

\item[$\rhd$] When there are no fresh vertices at some step, the procedure for $x_1$ stops.

\item[$\rhd$] Step 2. Assuming we have built $v_1^{\sss(1)}$,
we pair all its unpaired half-edges:
among the fresh connected vertices, we denote by
$v_2^{\sss(1)}$ the vertex with maximal degree.

\item[$\rhd$] We continue in this way for (at most) $h_n$ steps, defining $v_j^{\sss(1)}$ 
for $0 \le j \le h_n$ (recall \eqref{cmnd-hlog}).
\end{itemize}
After finishing the procedure for $x_1$, we perform the same procedure for $x_2, x_3, \ldots, x_N$,
defining the vertices $v_0^{\sss(i)}, v_1^{\sss(i)}, \ldots, v_{h_n}^{\sss(i)}$ starting from $v_0^{\sss(i)} = x_i$.
\end{Definition}

\begin{Definition}[Success]
\label{cmnd-def-success}
Let $x_1, \ldots, x_N$ be the vertices on the boundary of
a $k^{+}_n$-exploration graph $\widehat U = \widehat U_{\le k^{+}_n}(v)$.
We define the event $S_{x_i} :=$ ``$x_i$ 
is a {\em success}'' by
\begin{equation*}
	S_{x_i} := \big\{ \{v_0^{\sss(i)}, v_1^{\sss(i)}, \ldots, v_{h_n}^{\sss(i)}\} 
	\cap \core_n \ne \varnothing \big\}
	= \big\{ d_{v_j^{\sss(i)}} > (\log n)^\sigma \ \text{ for some } 0 \le j \le h_n\big\} \,.
\end{equation*}
\end{Definition}

\smallskip
Here is the key result, proved below:

\begin{Proposition}[Hitting the core quickly]
\label{cmnd-lemma-hitting}
There exists a constant $\eta>0$ such that, 
for every $n\in\N$ and
for every realization $H$ of $\widehat U$,
    		\eqn{
		\label{cmnd-success-1}
	    	\pr\big(S_{x_1} \,\big|\, \widehat U = H\big)\geq \eta,
	   	}
     and, for each $i=2,\ldots,N$,
		\eqn{
     		\label{cmnd-success-2}
	    	\pr\big(S_{x_i} \,\big|\, \widehat U = H \,,
		\, S_{x_1}^c \,,\, \ldots \,,\, S_{x_{i-1}}^c \big)\geq \eta.
     		}
\end{Proposition}

\smallskip
This directly leads to the proof of Statement \ref{stat-triplelog}, as the following
corollary shows:

\begin{Corollary}[Distance between periphery and $\core_n$]
\label{cmnd-prop-perifery}
	Under the hypotheses of Theorem \ref{main-cmnd}, with high probability,
	the distance of every vertex in the graph from $\core_n$ is at most
		\eqn{
		\label{cmnd-formula-perifery}
		(1+\varepsilon)\frac{\log\log n}{\log(\dmin-1)}+o\left(\log\log n\right).
		}
\end{Corollary}

\begin{proof}
By Corollary~\ref{cmnd-cor-boundryexpl}, with high probability,
every vertex $v \in [n]$ either is at distance at most $k^{+}_n$ from $\core_n$, or has a $k^{+}_n$-exploration graph $\widehat U = \widehat U_{\le k^{+}_n}(v)$
with at least $N \ge \frac{1}{2}(\log n)^{1+\varepsilon}$ vertices on its boundary. It suffices to consider the latter case.
Conditionally on $\widehat U = H$,
the probability that none of these vertices is a success can be bounded by
Proposition~\ref{cmnd-lemma-hitting}:
	\eqn{
\begin{split}
	\label{cmnd-fastly-16}
	\pr\big(S_{x_1}^c\cap\cdots \cap S_{x_N}^c
	\,\big|\, \widehat U = H \big) 
	& = \pr\big(S_{x_1}^c \,\big|\, \widehat U = H \big) 
	\prod_{j=2}^N 	\pr\big(S_{x_j}^c \,\big|\, \widehat U = H \,,
		\, S_{x_1}^c \,,\, \ldots \,,\, S_{x_{j-1}}^c \big)\\
	& \le (1-\eta)^N \le (1-\eta)^{\frac{1}{2}(\log n)^{1+\varepsilon}}
	= o(1/n) \,.
\end{split}
	}
This is uniform over $H$, hence the probability
that no vertex is a success, without conditioning, is still $o(1/n)$.
It follows that, with high probability, every $v \in [n]$ 
has at least one successful vertex on the boundary of its $k^{+}_n$-exploration graph. 
This means that the distance of every
vertex $v \in [n]$ from $\core_n$ is at most $k^{+}_n + h_n = k^{+}_n + o(\log \log n)$,
by \eqref{cmnd-hlog}. Recalling \eqref{cmnd-kstar-expl},
we have completed the proof of Corollary~\ref{cmnd-prop-perifery} 
and thus of Statement \ref{stat-triplelog}.
\end{proof}

To prove Proposition \ref{cmnd-lemma-hitting}, we need the following technical
(but simple) result: 

\begin{Lemma}[High-degree fresh vertices]
\label{cmnd-lemma-highvertices}
Consider the process of building a configuration model $\CMnd$
as described in Remark~\ref{cmnd-conditionedlaw}.
Let $\cG_l$ be the random graph obtained after $l$ pairings of half-edges
and let $V_{l}$ be the random vertex incident to the half-edge to which 
the $l$-th half-edge is paired.
For all $l,n\in\N$ and $z \in [0,\infty)$ such that
 	\eqn{
	\label{cmnd-zcondition}
  	l\leq \frac{n}{4}(1-F_{\sub{d},n}(z)),
 	}
the following holds:
 	\eqn{
	\label{cmnd-high-relation}
    	\pr\left(\left.d_{V_{l+1}}>z \,,\, 
	V_{l+1} \not\in \cG_l \,\right|
    	\cG_l\right)\geq z[1-F_{\sub{d},n}(z)]\frac{n}{2\ell_n}.
 	}
In particular, when Conditions~\ref{drc} and~\ref{pol-dis-con} hold,
for every $\zeta > 0$ 
there are $c > 0$, $n_0 < \infty$
such that
\begin{equation}\label{eq:cmnd-high-relation+}
	\forall \ n \ge n_0\,, \
	\, 0 \le z \le n^{1/3}\,, \ \, l \le n^{1/3}: \qquad
	\pr\left(\left.d_{V_{l+1}}>z \,,\, V_{l+1} \not\in \cG_l \,\right| \cG_l\right)
	\ge \frac{c}{z^{\tau - 2 + \zeta}} \,.
\end{equation}

\end{Lemma}
\begin{proof} 
By definition of $\CMnd$, the $(l+1)$-st half-edge is paired to a uniformly chosen
half-edge among the $\ell_n - 2l - 1$ that are not yet paired. Consequently
	\eqn{
	\pr\left(\left.d_{V_{l+1}}>z \,,\, 
	V_{l+1} \not\in \cG_l \,\right|
    	\cG_l\right) =
	\frac{1}{\ell_n-2l-1}\sum_{v\not\in \cG_l}d_v\I_{\{d_v > z\}}.
	}
Since $|\cG_l|\leq 2l\leq \frac{n}{2}(1-F_{\sub{d},n}(z))$ by \eqref{cmnd-zcondition}, we obtain
	\eqn{
	\frac{1}{\ell_n-2l-1}\sum_{v\not\in \cG_l}d_v\I_{\{d_v > z\}}
	\geq \frac{z}{\ell_n}\big(n(1-F_{\sub{d},n}(z))-|\cG_l|\big) \geq
	z(1-F_{\sub{d},n}(z))\frac{n}{2\ell_n},
	}
which proves \eqref{cmnd-high-relation}.

Assuming Conditions~\ref{drc} and~\ref{pol-dis-con}, we have
$\ell_n =\mu n(1+o(1)),$ with $\mu \in (0,\infty)$, {\blue see} \eqref{eq:nel},
and there are $c_1 > 0$ and $\alpha > 1/2$ such that
$1-F_{\sub{d},n}(z) \ge c_1 \, z^{-(\tau-1)}$ for $0 \le z \le n^\alpha$. 
Consequently, for $0 \le z \le n^{1/3}$, the right
hand side of \eqref{cmnd-zcondition} is at least $\frac{n}{4} \frac{c_1}{n^{(\tau-1)/3}}$.
Note that $(\tau - 1)/3 < 2/3$ (because $\tau < 3$), hence we can choose $n_0$
so that $\frac{n}{4} \frac{c_1}{n^{(\tau-1)/3}} \ge n^{1/3}$ for all $n \ge n_0$.
This directly leads to \eqref{eq:cmnd-high-relation+}.
\end{proof}

With Lemma \ref{cmnd-lemma-highvertices} in hand, we are able to prove Proposition \ref{cmnd-lemma-hitting}:

\begin{proof}[Proof of Proposition \ref{cmnd-lemma-hitting}]
We fix $v\in[n]$ and a realization $H$ of $\widehat U =\widehat U_{\le k^{+}_n}(v)$.
We abbreviate
	\begin{equation} 
	\label{eq:pstar}
	\pr^* (\,\cdot\,) := \pr(\,\cdot\,|\,\widehat U = H) \,.
	\end{equation}
The vertices on the boundary of $\widehat U$ are denoted by $x_1, \ldots, x_N$.
We start proving \eqref{cmnd-success-1}, hence we focus on $x_1$ and we define
$v_0^{\sss(1)}, v_1^{\sss(1)}, \ldots, v_{h_n}^{\sss(1)}$ as in Definition~\ref{cmnd-def-expl}, with $v_0^{\sss(1)} = x_1$.

\smallskip

We first fix some parameters. Since $2 < \tau < 3$,
we can choose $\zeta, \gamma > 0$ small enough so that
	\begin{equation}
	\label{eq:xifix}
	\xi := 1 - \e^\gamma(\tau-2+\zeta) > 0 \,.
	\end{equation}
Next we define a sequence $(g_\ell)_{\ell \in \N_0}$ that grows \emph{doubly exponentially} fast:
	\begin{equation} 
	\label{eq:gell}
	g_\ell :=
	2^{\e^{\gamma \ell}} = \exp \big( (\log 2) \exp (\gamma\, \ell) \big) \,.
	\end{equation}
Then we fix $B = 1/\gamma$ and $C = \log(\sigma/\log 2)$ in \eqref{cmnd-hlog}, 
where $\sigma$ is the same constant as in $\core_n$, see \eqref{core_n-def}.
With these choices, we have
\begin{equation} \label{eq:gbo}
	g_{h_n} = \e^{\sigma \e^{\lceil \log \log \log n\rceil}}
	> \e^{\sigma \log \log n} = (\log n)^\sigma \,, \qquad \text{while} \qquad
	g_{h_n - 1} < (\log n)^\sigma \,.
\end{equation}

Roughly speaking, the idea is to show that, with positive probability, one has $d_{v_j^{\sss(1)}} > g_j$.
As a consequence, $d_{v_{h_n}^{\sss(1)}} > g_{h_n} \ge (\log n)^\sigma$, that is $v_{h_n}^{\sss(1)}$ belongs to $\core_n$
and $x_1$ is a success. The situation is actually more involved, since
we can only show that $d_{v_j^{\sss(1)}} > g_j$ before reaching $\core_n$.

\smallskip

Let us make the above intuition precise.
Recalling \eqref{eq:pstar}, let us set
	\begin{equation*}
	H_{-1} := \varnothing\,, \qquad H_0 := H \,, \qquad
	H_k := H \cup \{v_1^{\sss(1)}, \ldots, v_k^{\sss(1)}\} \quad \text{for } 1 \le k \le h_n \,.
	\end{equation*}
Then we introduce the events
	\begin{equation} \label{eq:TW}
	T_\ell := \bigcup_{k=0}^{\ell} \big\{ {d_{v_k^{\sss(1)}} > (\log n)^\sigma} \big\} \,, \qquad
	W_\ell := \bigcap_{k=0}^{\ell} \big\{ d_{v_k^{\sss(1)}} >  g_k \,, \
	v_k^{\sss(1)} \not\in H_{k-1} \big\} \,.
	\end{equation}
In words, the event 
$T_\ell$ means that one of the vertices $v_0^{\sss(1)}, \ldots, v_\ell^{\sss(1)}$ has already reached $\core_n$,
while the event $W_\ell$ means that the degrees of vertices $v_0^{\sss(1)}, \ldots, v_\ell^{\sss(1)}$ 
grow at least like $g_0, \ldots, g_\ell$ and, furthermore, each $v_k$
is a fresh vertex (this is actually already implied by Definition~\ref{cmnd-def-expl},
otherwise $v_k$ would not even be defined). We finally set
	\begin{equation*}
	E_0 := W_0 \,, \qquad \quad
	E_j := T_{j-1} \cup W_j \quad \text{for} \ \ 1 \le j \le h_n \,.
	\end{equation*}

Note that $T_{h_n}$ coincides with $S_{x_1} = $ ``$x_1$ is a success''.
Also note that $W_{h_n} \subseteq \{d_{v_{h_n}^{\sss(1)}} > (\log n)^\sigma\}$, because 
$d_{v_{h_n}^{\sss(1)}} > g_{h_n} > (\log n)^\sigma$ by \eqref{eq:gbo}, hence
	\begin{equation*}
	E_{h_n} = T_{h_n - 1} \cup W_{h_n} \subseteq T_{h_n - 1} \cup
	\{d_{v_{h_n}^{\sss(1)}} > (\log n)^\sigma\} = T_{h_n} = S_{x_1} \,.
	\end{equation*}
Consequently, if we prove that $\pr^*(E_{h_n}) \ge \eta$, then our goal $\pr^*(S_{x_1}) \ge \eta$ follows
(recall \eqref{cmnd-success-1}).

\smallskip

The reason for working with the events $E_j$ is that their probabilities can be controlled
by an induction argument. Recalling \eqref{eq:pstar}, we can write
	\begin{equation} \label{eq:doli}
	\begin{split}
	\pr^*(E_{j+1})
	& = \pr^*(T_{j}) + \pr^*(T_{j}^c \cap W_{j+1}) \\
	& = \pr^*(T_{j}) + \pr\big( d_{v_{j+1}^{\sss(1)}} > g_{j+1} \,,\,
	v_{j+1}^{\sss(1)} \not\in H_{j} \,\big|\, 
	\{\widehat U = H\} \cap T_j^c \cap W_j \big)
	\, \pr^*(T_j^c \cap W_j ) \,.
	\end{split}
	\end{equation}
The key point is the following estimate on the conditional probability, proved below:
	\begin{equation} 
	\label{eq:bigpo}
	\begin{split}
	& \pr\big( d_{v_{j+1}^{\sss(1)}} > g_{j+1} \,,\,
	v_{j+1}^{\sss(1)} \not\in H_{j} \,\big|\, 
	\{\widehat U = H\} \cap T_j^c \cap W_j \big) \,\ge\, 1 - \varepsilon_{j} \,, \qquad
	\text{where} \quad \varepsilon_j := \e^{-c (g_j)^{\xi}/2} \,,
	\end{split}
	\end{equation}
with $\xi > 0$ is defined in \eqref{eq:xifix} and $c > 0$ is the constant appearing
in relation \eqref{eq:cmnd-high-relation+}. This yields
	\begin{equation*}
	\begin{split}
	\pr^*(E_{j+1}) & \,\ge\, \pr^*(T_{j}) + (1 - \varepsilon_{j})
	\, \pr^*(T_j^c \cap W_j ) \,\ge\, 
	(1 - \varepsilon_{j}) \big( \pr^*(T_{j}) +
	\, \pr^*(T_j^c \cap W_j ) \big) \\
	& \,=\, (1 - \varepsilon_{j}) \pr^*(T_j \cup W_j) 
	\,\ge\, (1 - \varepsilon_{j}) \pr^*(T_{j-1} \cup W_j) \\
	& \,=\, (1 - \varepsilon_{j}) \pr^*(E_j) \,,
	\end{split}
	\end{equation*}
which leads us to
	\begin{equation*}
	\pr^*(E_{h_n}) \ge \pr^*(E_0) \prod_{j=0}^{h_n - 1} (1-\varepsilon_j)
	\ge \pr^*(E_0) \prod_{j=0}^{\infty} (1-\varepsilon_j) =: \eta  \,.
	\end{equation*}
Since $\sum_{j\geq 0} \varepsilon_j < \infty$ and $\varepsilon_j<1$ for every $j\geq 0$, by \eqref{eq:bigpo} and \eqref{eq:gell},
the infinite product is strictly positive. Also note that $\pr^*(E_0) = \pr^*(d_{v_0^{\sss(1)}} \ge 2) = 1$,
because $g_0 = 2$ and $d_{v_0^{\sss(1)}} \ge \dmin \ge 3$. Then $\eta > 0$, as required.

\smallskip

It remains to prove \eqref{eq:bigpo}. To lighten notation, we rewrite the left hand side
of \eqref{eq:bigpo} as
\begin{equation} \label{eq:qD}
	q_{j+1} := \pr\big( d_{v_{j+1}^{\sss(1)}} > g_{j+1} \,,\,
	v_{j+1}^{\sss(1)} \not\in H_{j} \,\big|\, 
	D_j \big) \,, \qquad \text{where} \quad
	D_j := \{\widehat U = H\} \cap T_j^c \cap W_j \,.
\end{equation}
Note that, on the event $D_j \subseteq W_j$,
vertex $v_j^{\sss(1)}$ is fresh (i.e., it is connected to the graph for the first time),
hence it has $m = d_{v_j^{\sss(1)}} - 1$ unpaired half-edges. 
These are paired uniformly, connecting $v_j^{\sss(1)}$ to
(not necessarily distinct) vertices $w^{\sss(1)}, \ldots,
w^{\sss(m)}$. Let us introduce for $1 \le \ell \le m$ the event
\begin{equation} \label{eq:Cell}
	C_\ell := \bigcap_{k=1}^\ell
	\big\{ d_{w^{\sss(k)}} > g_{j+1} \,,\,
	w^{\sss(k)} \not\in H_{j} \big\}^c \,.
\end{equation}
By Definition~\ref{cmnd-def-expl}, $v_{j+1}^{\sss(1)}$ is the \emph{fresh}
vertex with maximal degree among them, hence
\begin{equation*}
	\big\{ d_{v_{j+1}^{\sss(1)}} > g_{j+1} \,,\,
	v_{j+1}^{\sss(1)} \not\in H_{j} \big\}^c = C_m \,.
\end{equation*}
Since $m = d_{v_j^{\sss(1)}}-1 > g_j-1$ on $W_j \subseteq D_j$,
the left hand side of \eqref{eq:bigpo} can be estimated by
\begin{equation} \label{eq:lastim}
\begin{split}
	q_{j+1} & = 1 - \pr\big( C_m \,\big|\, D_j \big)
	\ge 1 - \prod_{k=1}^{g_j-1} \pr\big( C_k \,\big|\, 
	D_j \cap C_{k-1} \big) \\
	& = 1 - \prod_{k=1}^{g_j-1} \Big( 1 - \pr\big( d_{w^{(k)}} > g_{j+1} \,,\,
	w^{\sss(k)} \not\in H_{j} \,\big|\, 
	D_j \cap C_{k-1} \big) \Big) \,.
\end{split}	
\end{equation}

We claim that we can apply relation
\eqref{eq:cmnd-high-relation+} from Lemma~\ref{cmnd-lemma-highvertices}
to each of the probabilities in the last line of \eqref{eq:lastim}.
To justify this claim, 
we need to look at the conditioning event $D_j \cap C_{k-1}$, recalling \eqref{eq:Cell},
\eqref{eq:qD} and \eqref{eq:TW}. In order to produce it, we have to do the following:
\begin{itemize}
\item[$\rhd$] First we build the $k^{+}_n$-exploration graph $\widehat U_{\le k^{+}_n}(v) = H$,
which requires to pair at most $O( (\dmin-1)^{k^{+}_n} ) = O((\log n)^{1+\varepsilon})$
half-edges (recall Definition~\ref{cmnd-def-kexp});

\item[$\rhd$] Next, starting from the boundary vertex $x_1$, we generate
the fresh vertices $v_0^{\sss(1)}, \ldots, v_j^{\sss(1)}$ \emph{all outside $\core_n$}, because we are on
the event $T_j^c$, and this requires to pair a number of half-edges which is at most
$(\log n)^\sigma j \le (\log n)^\sigma h_n = O((\log n)^{\sigma + 1})$;

\item[$\rhd$] Finally, in order to generate $w^{\sss(1)}, \ldots, w^{\sss(k-1)}$, we pair exactly
$k-1$ half-edges, and note that $k-1 \le g_j - 1 \le g_{h_n} - 1 = O((\log n)^{\sigma})$
(always because $v_j \not\in \core_n$).
\end{itemize}

It follows that the conditioning event $D_j \cap C_{k-1}$ is in the $\sigma$-algebra generated
by $\cG_l$ for $l \le O((\log n)^{1+\sigma+\varepsilon})$
(we use the notation of Lemma~\ref{cmnd-lemma-highvertices}).
In particular, $l \le n^{1/3}$.
Also note that $z = g_{j+1} \le g_{h_n} = O((\log n)^{\sigma})$,
{\blue see} \eqref{eq:gbo}, hence also $z \le n^{1/3}$.
Applying \eqref{eq:cmnd-high-relation+}, we get
	\eqan{
	\label{qj-bd-1}
	q_{j+1} & \ge 1 - \bigg( 1 - \frac{c}{(g_{j+1})^{\tau-2+\zeta}} \bigg)^{g_j-1}
	\ge 1 - \exp \bigg(- c\frac{g_j-1}{(g_{j+1})^{\tau-2+\zeta}} \bigg) \\
	& \ge 1 - \exp \bigg(- \frac{c}{2} \, \frac{g_j}{(g_{j+1})^{\tau-2+\zeta}} \bigg)\nn
	}
because $1-x \le \e^{-x}$ and $n-1 \ge n/2$ for all $n \ge 2$
(note that $g_j \ge g_0 = 2$).
Since $g_{j+1} = (g_j)^{\e^\gamma}$, by \eqref{eq:gell}, we finally arrive at
	\eqn{
	\label{qj-bd-2}
	q_{j+1} \ge 1 - \exp \bigg( -\frac{c}{2} \, (g_j)^{1 - \e^\gamma(\tau-2+\zeta)} \bigg)
	= 1 - \e^{-c \, (g_j)^\xi/2} \,,
	}
which is precisely \eqref{eq:bigpo}. This completes the proof of \eqref{cmnd-success-1}.

\smallskip

In order to prove \eqref{cmnd-success-2}, we proceed in the same way:
for any fixed $2 \le i \le N$, we start
from the modification of \eqref{eq:pstar} given by 
$\pr^* (\,\cdot\,) := \pr(\,\cdot\,|\,\widehat U = H \,,
		\, S_{x_1}^c \,,\, \ldots \,,\, S_{x_{i-1}}^c )$
and we follow the same proof, 
working with the vertices $v_1^{\sss(i)}$, \ldots, $v_{h_n}^{\sss(i)}$
instead of $v_1^{\sss(1)}, \ldots, v_{h_n}^{\sss(1)}$ (recall Definition~\ref{cmnd-def-expl}).
We leave the details to the reader.
\end{proof}

\section{Upper bound for Preferential attachment model}
\label{upperproofs-prefatt}

In this section we prove Statements~\ref{stat-collision} and~\ref{stat-triplelog}
for the preferential attachment model. By the discussion in
Section~\ref{sub-upper}, this completes the proof of the upper
bound in Theorem~\ref{main-prefatt}, because the proof
of Statement~\ref{stat-core} is already known in the literature, 
as explained below Statement~\ref{stat-core}.

\subsection{Proof of Statement \ref{stat-collision}}
\label{proof-upper-prefatt-1}

Recall the definition of $\core_t$ in \eqref{core_n-def}. It is crucial that in $\core_t$, we let $D_{t/2}(v)$ be large. We again continue to define what a $k$-exploration graph and its collisions are, but this time for the preferential attachment model:

\begin{Definition}[$k$-exploration graph]
\label{prefatt-def-kexpl}
Let $(\PA{t})_{t\geq 1}$ be a preferential attachment model. For $v\in[t]$, we call the {\em $k$-exploration graph} of $v$ to be the subgraph of $\PA{t}$, where we consider the $m$ edges originally incident to $v$, and the $m$ edges originally incident to any other vertex that is connected to $v$ in this procedure, up to distance $k$ from $v$.
\end{Definition}

\begin{Definition}[Collision]
\label{prefatt-def-coll}
Let $(\PA{t})_{t\geq 1}$ be a preferential attachment model with $m\geq 2$, and let $v$ be a vertex. We say that we have a {\em collision} in the $k$-exploration graph of $v$ when one of the $m$ edges of a vertex in the $k$-exploration graph of $v$ is connected to a vertex that is already in the $k$-exploration graph of $v$.
\end{Definition}

Now we want to show that every $k$-exploration graph has at most a finite number of collisions before hitting the $\core_t$, as we did for the configuration model. The first step is to use \cite[Lemma 3.9]{DSvdH}:

\begin{Lemma}[Early vertices have large degree]
\label{prefatt-lemma-oldvertices}
Fix $m\geq 1$. There exists $a>0$ such that
 	\eqn{
	\label{prefatt-formula-oldvertices}
  	\pr\Big(\min_{i\leq t^a}D_t(i)\geq (\log t)^\sigma\Big)\longrightarrow 1
 	}
for some $\sigma>1/(3-\tau)$. As consequence, $[t^a]\subseteq \core_t$ with high probability.
\end{Lemma}
{\blue 
In agreement with \eqref{kstar-exp}
(see also \eqref{kstar-MKC-prefatt}), we set
	\eqn{
	\label{kstar-MKC-prefatt-new}
	k^{+}_t = (1+\varepsilon)\frac{\log\log t}{\log m}.
	}
We  want to prove that the exploration }
graph $\widehat{U}_{\leq k^{+}_t}(v)$ has at most a finite number of collisions before hitting $\core_t$, similarly to the case of $\CMnd$, now for $\PA{t}$. As it is possible to see from \eqref{core_n-def}, $\core_t\subseteq [t/2]$, i.e., is a subset defined in $\PA{t}$ when the graph has size $t/2$. As a consequence, we do not know the degree of vertices in $[t/2]$ when the graph has size $t$. 
However, in \cite[Appendix A.4]{DSvdH} the authors
prove that at time $t$
all the vertices $t/2+1,\ldots,t$ have degree smaller than $(\log t)^\sigma$.

We continue by giving a bound on the degree of vertices that are not in $\core_t$.
{\blue For vertices $i\in[t/2] \setminus \core_t$ we know that
$D_{t/2}(i)<(\log t)^\sigma$, see \eqref{core_n-def},
but in principle their degree $D_t(i)$ at time $t$
could be quite high}. We need to prove that this happens with very small probability. Precisely, we prove that, for some $B>0$,
	\eqn{
	\label{eq::poly0}
	\pr\left(\max_{i\in[t/2]\setminus\core_t}D_t(i)\geq (1+B)(\log t)^\sigma\right) =o(1).
	}
This inequality implies that when a degree is at most $(\log t)^\sigma$ at time $t/2$, then it is unlikely to grow by $B(\log t)^\sigma$ between time $t/2$ and $t$. This provides a bound on the cardinality of incoming neighborhoods that we can use in the definition of the exploration processes that we will rely on, in order to avoid $\core_t$. 
We prove \eqref{eq::poly0} in the following lemma that is an adaptation of the proof of \cite[Lemma A.4]{DSvdH}. Its proof is deferred to \longversion{Appendix \ref{app-B}}\shortversion{\cite[Appendix B]{CarGarHof16ext}}: 

\begin{Lemma}[Old vertex not in $\core_t$]
\label{lem-degoutcore}
There exists $B\in (0,\infty)$ such that, for every $i\in[t/2]$,  
	\eqn{
	\pr\left(D_t(i)\geq (1+B)(\log t)^\sigma \mid D_{t/2}(i)<(\log t)^\sigma\right) = o(1/t).
	}
\end{Lemma}

\medskip

{\blue We can now get to the core of the proof of Statement \ref{stat-collision},
that is we show that there are few collisions before reaching $\core_t$:} 

\begin{Lemma}[Few collisions before hitting the core]
\label{prefatt-corol-numbercollision}
Let $(\PA{t})_{t\geq 1}$ be a preferential attachment model, with $m\geq 2$ and $\delta\in(-m,0)$. 
Fix $a \in (0,1)$ and $l\in\N$ such that $l > 1/a$.
With $k^{+}_t$ as in \eqref{kstar-MKC-prefatt-new},
the probability that there exists a vertex $v\in[t]$ such that its $k^{+}_t$-exploration graph has at least $l$ collisions before hitting $\core_t \cup [t^a]$ is $o(1)$.
\end{Lemma}

Next we give a lower bound on the number of vertices 
on the boundary of a $k^+_n$-exploration graph. 
First of all, for any fixed $a \in (0,1)$, we notice that 
the probability of existence of a vertex in $[t]\setminus [t^a]$, 
that has only self loops is $o(1)$. Indeed,
the probability that a vertex $s$ has only self-loops is $O(\frac{1}{s^m})$. Thus, the probability that there exists a vertex in $[t]\setminus [t^a]$ that has only self-loops is bounded above by
	\eqn{
	\sum_{s>t^a} O(\frac{1}{s^m})=O(t^{-a(m-1)})=o(1),
	}
since we assume that $m\geq 2$. 
We can thus assume that no vertex in $[t]\setminus [t^a]$ has only self-loops.
This leads us formulate the following Lemma,
whose proof is also deferred to \longversion{Appendix~\ref{app-B}}\shortversion{\cite[Appendix B]{CarGarHof16ext}}.

\begin{Lemma}[Lower bound on boundary vertices]
\label{prefatt-lemma-boundarybound}
 Let $(\PA{t})_{t\geq 1}$ be a preferential attachment model, with $m\geq 2$ and $\delta\in(-m,0)$. 
 For $a\in(0,1)$, consider a vertex $v\in[t]\setminus(\core_t\cup [t^a])$ 
 and its $k$-exploration graph.
 If there are at most $l$ collisions in the $k$-exploration graph, 
and no vertex in $[t] \setminus [t^a]$ has only self loops,
 then there exists a constant $s = s(m,l)>0$ such that the number of vertices in the boundary of the $k$-exploration graph
 is at least $s(m,l)m^{k}$.
\end{Lemma}

Together, Lemmas~\ref{prefatt-lemma-oldvertices},
\ref{prefatt-corol-numbercollision} 
and~\ref{prefatt-lemma-boundarybound} complete the proof of Statement \ref{stat-collision}.

\medskip

The rest of this section is devoted to the proof of Lemma~\ref{prefatt-corol-numbercollision}.
We first need to introduce some notation, in order to 
be able to express the probability of collisions. We do this in the next subsection. 

\subsubsection{Ulam-Harris notation for trees}
\label{sec-Ulam-Harris}
Define
	\begin{equation*}
	W_\ell := [m]^\ell \,, \qquad W_{\le k} := \bigcup_{\ell = 0}^{k} W_\ell \,,
	\end{equation*}
where $W_0 := \varnothing$. We use $W_{\le k}$ as a universal set
to label any regular tree of depth
$k$, where each vertex has $m$ children. This is sometimes called the {\em Ulam-Harris notation} for trees. 

Given $y \in W_\ell$ and $z \in W_m$, we denote by $(y,z) \in W_{\ell + m}$
the concatenation of $y$ and $z$. Given $x, y \in W_{\le k}$, we write $y \succeq x$ if $y$ 
is a descendant of $x$, that is $y = (x,z)$ for some $z\in W_{\le k}$.

Given a finite number of points $z_1, \ldots, z_m \in W_{\le k}$, abbreviate $\vec{z}_m=(z_1, \ldots, z_m)$, and define
$W_{\le k}^{\sss(\vec z_m)}$ to be the tree obtained from $W_{\le k}$
by cutting the branches starting from any of the $z_i$'s (including the $z_i$'s themselves):
	\begin{equation} \label{eq:subt}
	W_{\le k}^{\sss(\vec{z}_k)} :=
	\big\{x \in W_{\le k}: \ x \not\succeq z_1, \ \ldots,
	x \not\succeq z_m \big\} \,.
	\end{equation}

\begin{remark}[Total order]
\label{def-tot-order}
\rm 
The set $W_{\le k}$ comes with a natural total order relation, 
{\blue called \emph{shortlex order},
in which shorter words precede
longer ones, and words with equal length are ordered lexicographically.}
More precisely, given $x \in W_\ell$ and $y \in W_m$, we say that $x$ precedes $y$ if
{\blue either $\ell < m$, or if $\ell = m$} and 
$x_i \le y_i$ for all $1 \le i \le \ell$.
We stress that this is a \emph{total}
order relation, unlike the descendant relation $\succeq$ which is only a partial order.
(Of course, if $y \succeq x$, then $x$ precedes $y$, but not vice versa).
\end{remark}

\subsubsection{Collisions}
We recall that, given $z \in [t]$ and $j \in [m]$,
the $j$-th half-edge starting from vertex $z$ in $\PA{t}$ is attached to a random vertex,
denoted by $\xi_{z,j}$. We can use the set $W_{\le k}$ to label
the exploration graph
$\widehat{U}_{\leq k}(v)$, as follows:
	\begin{equation} 
	\label{eq:urep}
	\widehat{U}_{\leq k}(v) = \big\{ V_z \big\}_{z \in W_{\le k}} \,,
	\end{equation}
where $V_\varnothing = v$ and, iteratively, $V_z = \xi_{V_x, j}$
for $z = (x,j)$ with $x \in W_{\le k-1}$
and $j \in [m]$.

The first vertex generating a \emph{collision}
is $V_{Z_1}$, where the random 
index $Z_1 \in W_{\le k}$ is given by
\begin{equation*}
	Z_1 := \min\big\{z \in W_{\le k}: \ V_z = V_y \ \text{for some } y 
	\ \text{which precedes} \ z \big\} \,,
\end{equation*}
where ``$\min$'' refers to the total order relation on $W_{\le k}$ as defined in 
Remark~\ref{def-tot-order}.

Now comes a tedious observation.
Since $V_{Z_1} = V_y$ for some $y$ {which precedes} $Z_1$, by definition of $Z_1$,
then all descendants of $Z_1$
will coincide with the corresponding descendants of $y$, that is
$V_{(Z_1, r)} = V_{(y,r)}$ for all $r$. 
In order not to over count collisions, in defining the second collision index $Z_2$,
we avoid exploring the descendants
of index $Z_1$, that is we only look at indices in $W_{\le k}^{\sss(Z_1)}$,
{\blue see} \eqref{eq:subt}.
The second vertex representing a (true) collision
is then $V_{Z_2}$, where we define
	\begin{equation*}
	Z_2 := \min\big\{z \in W_{\le k}^{\sss (Z_1)}\colon z \text{ follows }Z_1, \text{ i.e., }
	\ \ V_z = V_y \ \text{for some } y \ \text{which precedes} \ z \big\} \,,
	\end{equation*}
Iteratively, we define
	\begin{equation*}
	Z_{i+1} := \min\big\{z \in W_{\le k}^{\sss (\vec{Z}_i)} \colon
	z \text{ follows } Z_i, \text{ i.e., }V_z = V_y \ \text{for some } 
	y \ \text{which precedes} \ z \big\} \,,
	\end{equation*}
so that $V_{Z_i}$ is the $i$-th vertex that represents a collision.
The procedure stops when there are no more collisions. Denoting by $\cC$ the (random) number
of collisions, we have a family
	\begin{equation*}
	Z_1, \ Z_2, \, \ldots, \ Z_{\cC}
	\end{equation*}
of random elements of $W_{\le k}$, such that
$(V_{Z_i})_{1 \le i \le \cC}$ are the vertices generating the collisions.

\subsubsection{Proof of Lemma~\ref{prefatt-corol-numbercollision}}
\label{sec-key-bound}
Recalling \eqref{eq:urep} and \eqref{eq:subt},
given arbitrarily $z_1, \ldots, z_l \in W_{\le k}$, we define
	\begin{equation} \label{eq:urepplus}
	\widehat{U}_{\leq k}^{\sss (\vec{z}_l)}(v) 
	= \big\{ V_z \big\}_{z \in W_{\le k}^{\sss (\vec{z}_l)}} \,,
	\end{equation}
that is, we consider a subset of the full exploration graph
$\widehat{U}_{\leq k}(v)$, 
consisting of vertices $V_z$ whose indexes $z \in W_{\le k}$ are not
descendants of $z_1, \ldots, z_l$. The basic observation is that
	\begin{equation} \label{eq:identi}
	\widehat{U}_{\leq k}(v) = \widehat{U}_{\leq k}^{\sss(\vec{z}_l)}(v) 
	\qquad \text{on the event}
	\quad \{\cC = l\,, \ Z_1 = z_1, \ldots, Z_l = z_l\} \,.
	\end{equation}
In words, this means that to recover the full exploration graph
$\widehat{U}_{\leq k}(v)$, it is irrelevant to look at vertices $V_z$ for $z$ that is a descendant of
a collision index $z_1, \ldots, z_l$.

We will bound the probability that there are $l$ collisions before reaching $\core_t \cup [t^a]$, 
occurring at specified indices $z_1, \ldots, z_l \in W_{\le k}$,
for $k = k^{+}_t$ as in \eqref{kstar-MKC-prefatt-new}, as follows:
	\begin{equation}
	\label{eq:key}
	\begin{split}
	\pr\big(\cC = l\,, \ Z_1 = z_1, \ldots, Z_l = z_l, \
	& \widehat{U}_{\leq k}(v) \cap (\core_t \cup [t^a]) = \varnothing\big) 
	\,\le\, \alpha(t)^l \,, 
	\end{split}
	\end{equation}
where, for the constant $B$ given by Lemma \ref{lem-degoutcore}, we define
	\begin{equation} 
	\label{eq:alphat}
	\alpha(t) = \frac{4(1+B)}{m}\frac{(\log t)^{\sigma + 1 + \varepsilon}}{t^a} \,.
	\end{equation}
Summing \eqref{eq:key} over $z_1, \ldots, z_l \in W_{\le k}$ we get
	\begin{equation*}
	\pr(\cC = l ,\ \widehat{U}_{\leq k}(v) \cap (\core_t \cup [t^a]) = \varnothing) 
	\le \alpha(t)^l \, | W_{\le k} |^l \, .
	\end{equation*}
Since, for $k = k^{+}_t$ as in \eqref{kstar-MKC-prefatt-new}, we can bound
	\begin{equation}
	\label{eq:cardW}
	|W_{\le k}| = \frac{m^{k+1}-1}{m-1} \le 2 \, m^k
	\le 2 \, (\log t)^{1+\varepsilon} \,,
	\end{equation}
the probability of having at least $l$ collisions, before reaching $\core_t \cup [t^a]$,
is $O(\alpha(t)^l (\log t)^{2l}) 
= o(1/t)$, because $l > 1/a$ by assumption.
This completes the proof of Lemma~\ref{prefatt-corol-numbercollision}.
It only remains to show that \eqref{eq:key} holds true.

\subsubsection{Proof of \eqref{eq:key}: case $l=1$}
\label{sec-proof-key-bound-1}
We start proving  \eqref{eq:key} for one collision. By \eqref{eq:identi}, we can replace $\widehat{U}_{\leq k}(v)$
by $\widehat{U}_{\leq k}^{(z_1)}(v)$ in the left hand side of \eqref{eq:key},
i.e., we have to prove that
	\eqn{\label{eq:c1}
	\begin{split}
	& \pr(\cC = 1\,, \ Z_1 = z_1, \
	\widehat{U}_{\leq k}^{\sss(z_1)}(v) \cap (\core_t \cup [t^a]) = \varnothing)
	\,\le\, \alpha(t) \,.
	\end{split}
	}

Since $v,k$ and $z_1$ are fixed, let us abbreviate, and recalling \eqref{eq:urepplus},
	\begin{equation}
	\label{eq:forshort}
	\cW := W_{\le k}^{\sss(z_1)}(v) \,, \qquad
	\widehat{U} := \widehat{U}_{\leq k}^{(z_1)}(v) = \big\{ V_z \big\}_{z \in \cW} \,.
	\end{equation}
Note that $V_{z_1}$ is the only collision precisely
when \emph{$\widehat{U}$ is a tree and $V_{z_1} \in \widehat{U}$}.
Then \eqref{eq:c1} becomes
	\begin{equation} 
	\label{eq:impl}
	\begin{split}
	& \pr(\widehat{U} \text{ is a tree}\,, \ V_{z_1} \in \widehat{U}\,, \
	\widehat{U} \cap (\core_t \cup [t^a]) = \varnothing) 
	\le \alpha(t) \, .
	\end{split}
	\end{equation}
We will actually prove a stronger statement: for any fixed \emph{deterministic}
labeled directed tree $H \subseteq [t]$ and for any $y \in H$,
	\begin{equation}\label{eq:c1+}
	\begin{split}
	& \pr(\widehat{U} = H\,, \ V_{z_1} = y\,, \
	H \cap (\core_t \cup [t^a]) = \varnothing)  \le \frac{\alpha(t)}{2 (\log t)^{1+\varepsilon}} \, 
	\pr\big(\widehat{U} =H \,, \ V_{z_1} \not\in H
	\big) \,.
	\end{split}
	\end{equation}
This yields \eqref{eq:impl} by summing over $y \in H$
---note that $|H| \le |W_{\le k}| \le 2 (\log t)^{1+\varepsilon}$
by \eqref{eq:cardW}--- and then summing over all possible realizations of $H$.

\smallskip

It remains to prove \eqref{eq:c1+}. We again use the notion of a {\em factorizable event,} as in the proof of the lower bound. 
Since the events in \eqref{eq:c1+} are not factorizable, we will specify
the incoming neighborhood $\cN(y)$ (recall \eqref{eq:inco}) of all 
$y \in H$. More precisely, by labeling the vertices of $H$,
{\blue see} \eqref{eq:forshort}, as
	\begin{equation} \label{eq:asin}
	H = \{v_s\}_{s \in \cW} \qquad \text{and} \qquad y = v_{\bar s} \,, \qquad
	\text{for some} \ \ \bar s \in \cW \,,
	\end{equation}
we can consider the events $\{\cN(v_s) = N_{v_s}\}$
where $N_{v_s}$ are (deterministic) disjoint subsets of $[t] \times [m]$.
We say that the subsets $(N_{v_s})_{s\in\cW}$ are \emph{compatible} with the tree $H$ when
$(v_{s'}, j) \in N_{v_s}$ whenever $s = (s', j)$ with $s, s' \in \cW$, $j \in [m]$. Then we can write
	\begin{equation} \label{eq:unio}
	\{\widehat{U} = H\} = \bigcup_{\text{compatible } (N_{v_s})_{s\in\cW}}
	\{ \cN(v_s)=N_{v_s}\ \text{ for every } \ s \in \cW\} \,.
	\end{equation}
Since the degree of vertex $v_s$ equals $D_t(v_s) = m + |N_{v_s}|$,
we can ensure that $H \cap (\core_t \cup [t^a]) = \varnothing$ 
by restricting the union in \eqref{eq:unio} to those $N_{v_s}$ satisfying
the constraints
	\begin{equation}
	\label{eq:constraints}
	v_s > t^a \qquad \text{and} \qquad |N_{v_s}| \le (1+B)(\log t)^\sigma - m \,,
	\qquad \forall s \in \cW \,.
	\end{equation}
Finally, if we write
	\begin{equation}
	\label{eq:z1xj}
	z_1 = (x,j) \qquad \text{for some} \quad x \in \cW \,, \ j \in [m] \,,
	\end{equation}
then, since $V_{z_1} = \xi_{V_{x},j}$, the event $\{V_{z_1} = v_{\bar s}\}$ amounts
to require that\footnote{Incidentally, we observe that the constraint
\eqref{eq:lastconst} is not included in the requirement that 
$(N_{v_s})_{s\in\cW}$ are compatible, because $z_1 = (x,j) \not\in \cW$
by definition \eqref{eq:forshort} of $\cW$.}
\begin{equation}\label{eq:lastconst}
	(v_x,j) \in N_{v_{\bar s}} \,.
\end{equation}

Let us summarize where we now stand: When
we fix a family of $(N_{v_s})_{s \in \cW}$ that is compatible
and satisfies the constraints \eqref{eq:constraints} and \eqref{eq:lastconst},
in order to prove \eqref{eq:c1+} it is enough to show that
	\begin{equation}\label{eq:c1++}
	\begin{split}
	& \pr(\cN(v_s)=N_{v_s} \text{ for every } s \in \cW) \\
	& \qquad \le \frac{\alpha(t)}{2 (\log t)^{1+\varepsilon}} \,
	\pr(\cN(v_s)=N_{v_s} \text{ for every } s \in \cW \setminus \{\bar s\}, \ \
	\cN(v_{\bar s})=N_{v_{\bar s}} \setminus \{(v_x,j)\}) \,.
	\end{split}
	\end{equation}
\smallskip

Let us set
	\begin{equation}\label{eq:N}
	N := \bigcup_{s \in \cW} N_{v_s} \subseteq [t] \times [m] \,.
	\end{equation}
The probability on the left-hand side of \eqref{eq:c1++}
can be factorized, using conditional expectations and the tower property, as a product of two kinds of terms:

\begin{itemize}
\item[$\rhd$] For every edge $(u,r)\in N$ ---say $(u,r) \in N_{v_s}$, with $s \in \cW$---
we have the term
	\eqn{
	\label{eq:Dul-1}
	\frac{D_{u,r-1}(v_s)+\delta}{c_{u,r}}
	}
corresponding to the fact that the edge needs to be connected to $v_s$;

\item[$\rhd$] On the other hand, 
for every edge $(u,r)\not\in N$, we have the term
	\eqn{
	\label{eq:outH}
	1-\frac{D_{u,r-1}(H)+|H\cap [u-1]|\delta}{c_{u,r}},
	}
corresponding to the fact that the edge may not connect to any vertex in $H$.
\end{itemize}

(We emphasize that all the degrees $D_{\cdot,\cdot}(\,\cdot\,)$ appearing in
\eqref{eq:Dul-1} and \eqref{eq:outH} are \emph{deterministic}, since they
are fully determined by the realizations of the incoming neighborhoods
$(N_{v_s})_{s\in\cW}$.)

We can obtain the right-hand side in \eqref{eq:c1++} by replacing some terms 
in the product. 
\begin{itemize}
\item[$\rhd$] Among the edges $(u,r)\in N$, whose contribution is \eqref{eq:Dul-1}, we have
the one that creates the collision, namely $(v_x,j)$.
If we want this edge to be connected \emph{outside} $H$, 
as in the right-hand side in \eqref{eq:c1++},
we need to divide the left hand side of \eqref{eq:c1++} by 
	\eqn{
	\label{eq:div1}
	\left(\frac{D_{v_x,j-1}(v_{\bar s})+\delta }{c_{v_x,j}}\right)
	\left(1-\frac{D_{v_x,j-1}(H)
	+|H \cap [v_x-1]|\delta}{c_{v_x,j}}\right)^{-1} .
	}
We also have to replace some other terms corresponding to edges 
$(u,r) \in N_{v_{\bar s}}$, because the degree of 
vertex $v_{\bar s}$ is decreased by one after connecting $(v_x,j)$ outside $H$.
More precisely, for every edge $(u,r)\in N_{v_{\bar s}}$
that is younger than $(v_x,j)$, that is
$(u,r)>(v_x,j)$, we can reduce the degree of $v_{\bar s}$ 
by one by dividing the left-hand side of \eqref{eq:c1++} by 
	\eqn{
	\label{eq:div2}
	\prod_{(u,r)\in N_{v_{\bar s}},\
	(u,r) > (v_x,j)}\frac{D_{u,r-1}(v_{\bar s})+\delta}
	{D_{u,r-1}(v_{\bar s}) - 1 +\delta}
	= \frac{D_{t}(v_{\bar s})+\delta}{D_{v_x,j-1}(v_{\bar s}) +\delta} \,.
	}
Finally, the contribution of the edges $(u,r) \in N_{v_{s}}$ for $s \ne \bar s$ is unchanged.

\item[$\rhd$] For every edge $(u,r)\not\in N$, 
the probability that such edge is not attached to $H$, after we reconnect the edge $(v_x,j)$,
becomes larger, since the degree of $H$ is reduced by one. 
\end{itemize}

It follows that the inequality \eqref{eq:c1++} holds with
$\alpha(t) / (2 (\log t)^{1+\varepsilon})$ replaced by $\beta$, defined by
\begin{equation} \label{eq:betadef}
\begin{split}
	\beta & = \left(\frac{D_{v_x,j-1}(v_{\bar s})+\delta }{c_{v_x,j}}\right)
	\left(1-\frac{D_{v_x,j-1}(H)
	+|H \cap [v_x-1]|\delta}{c_{v_x,j}}\right)^{-1}
	\frac{D_{t}(v_{\bar s})+\delta}{D_{v_x,j-1}(v_{\bar s}) +\delta} \\
	& = \left(\frac{D_{t}(v_{\bar s})+\delta}{c_{v_x,j}}\right)\left(1-\frac{D_{v_x,j-1}(H)
	+|H \cap [v_x-1]|\delta}{c_{v_x,j}}\right)^{-1} \\
	& \le \left(\frac{D_{t}(v_{\bar s})}{c_{v_x,j}}\right)
	\left(1-\frac{D_{v_x,j-1}(H)}{c_{v_x,j}}\right)^{-1}  =: \beta' \,,
\end{split}
\end{equation}
because $\delta \le 0$.
We only need to show that $\beta' \le \alpha(t) / (2 (\log t)^{1+\varepsilon})$.

Since $c_{v,j} \ge m(v-1)$, the first relation in \eqref{eq:constraints} yields
	\begin{equation*}
	c_{v_x, j} \ge  t^a.
	\end{equation*}
Hence, since $D_{t}(v_{\bar s}) \le (1+B)(\log t)^\sigma$
by the second relation in \eqref{eq:constraints}, we can bound
	\begin{equation*}
	\left(\frac{D_{t}(v_{\bar s}) }{c_{v_x,j}}\right)
	\le \frac{(1+B)(\log t)^\sigma}{m t^a} \,.
	\end{equation*}
Likewise, since $D_t(H) \le |H| (1+B) (\log t)^\sigma$, for $k = k^{+}_t$ we get,
by \eqref{eq:cardW},
	\begin{equation*}
	\left(1-\frac{D_{v_x,j-1}(H)}{c_{v_x,j}}\right)^{-1} \le
	\left(1-\frac{2 (\log t)^{1+\varepsilon} (1+B)(\log t)^\sigma}{t^a}\right)^{-1}
	\le 2 \,,
	\end{equation*}
where the last inequality holds for $t$ large enough. 
Recalling \eqref{eq:alphat},
	\begin{equation*}
	\beta' \le 2 \frac{(1+B)(\log t)^\sigma}{m t^a}
	= \frac{\alpha(t)}{2 (\log t)^{1+\varepsilon}} \,.
	\end{equation*}
This completes the proof of \eqref{eq:c1++},  and
hence of \eqref{eq:key}, in the case where $l=1$.
\qed

\subsubsection{Proof of \eqref{eq:key}: general case $l\geq 2$}
The proof for the general case is very similar to that for $l=1$, so we only highlight the (minor) changes.

In analogy with \eqref{eq:c1},  we can replace $\widehat{U}_{\leq k}(v)$
by $\widehat{U}_{\leq k}^{\sss (\vec{z}_l)}(v)$ in the left-hand side of \eqref{eq:key},
thanks to \eqref{eq:identi}. Then, as in \eqref{eq:forshort}, we write
	\begin{equation} 
	\label{eq:forshortl}
	\cW := W_{\le k}^{\sss (\vec{z}_l)}(v) \,, \qquad
	\widehat{U} := \widehat{U}_{\leq k}^{\sss (\vec{z}_l)}(v)
	= \big\{ V_z \big\}_{z \in \cW} \,.
	\end{equation}
The extension of \eqref{eq:c1+} becomes that for any fixed \emph{deterministic}
labeled directed tree $H \subseteq [t]$ and for all $y_1, \ldots, y_l \in H$,
	\begin{equation}\label{eq:c1+l}
	\begin{split}
	& \pr(\widehat{U} = H\,, \ V_{z_1} = y_1\,,\, \ldots, V_{z_l} = y_l\,, \
	H \cap (\core_t \cup [t^a]) = \varnothing)  \\
	& \ \ \le \left(\frac{\alpha(t)}{2(\log t)^{1+\varepsilon}}\right)^l \, 
	\pr\big(\widehat{U} =H\,, \ V_{z_1} \not\in H \,, \,
	V_{z_2} \not\in H\,,\, \ldots \,, \,
	V_{z_l} \not\in H
	\big) \,.
	\end{split}
	\end{equation}
As in \eqref{eq:asin}, we can write
	\begin{equation*}
	H = \{v_s\}_{s \in \cW} \qquad \text{and} \qquad y_1 = v_{\bar s_1} \,,\ \ldots \,, \
	y_l = v_{\bar s_l} \qquad
	\text{for some} \ \ \bar s_1\,, \ \ldots\,, \ \bar s_l \in \cW \,.
	\end{equation*}

To obtaint a factorizable event, we must specify the incoming neighborhoods
$\cN_{v_s} = N_{v_s}$ for all $s\in\cW$, which must be compatible with $H$ and
satisfy the constraint \eqref{eq:constraints}. If we write
	\begin{equation*}
	z_1 = (x_1, j_1)\,, \ \ldots, \ z_l = (x_l, j_l) \,, \qquad
	\text{for some} \quad x_1, \ldots, x_l \in \cW, \ j_1, \ldots, j_l \in [m] \,,
	\end{equation*}
then we also impose the constraint that obviously generalizes \eqref{eq:lastconst}, namely
	\begin{equation*}
	(v_{x_1}, j_1) \in N_{v_{\bar s_1}} \,, \ \ldots \,, \
	(v_{x_l}, j_l) \in N_{v_{\bar s_l}} \,.
	\end{equation*}
The analogue of \eqref{eq:c1++} then becomes
	\begin{equation}\label{eq:c1++l}
	\begin{split}
	& \pr(\cN(v_s)=N_{v_s} \text{ for every } s \in \cW) \\
	& \qquad \le \left(\frac{\alpha(t)}{2(\log t)^{1+\varepsilon}}\right)^l \,
	\pr\big(\cN(v_s)=N_{v_s} \text{ for every } s \in \cW \setminus \{\bar s_1, \ldots, \bar s_l\}, \\
	& \hskip5cm
	\cN(v_{\bar s_i}) = N_{v_{\bar s_i}} \setminus \{(v_{x_i}, j_i)\}
	\text{ for every } i=1,\ldots, l \big) \,.
\end{split}
\end{equation}

When we define $N$ as in \eqref{eq:N},
the probability in the left-hand side of \eqref{eq:c1++l}
can be factorized in a product of terms of two different types, which are given
precisely by \eqref{eq:Dul-1} and \eqref{eq:outH}. 
In order to obtain the probability in the right-hand side of \eqref{eq:c1++l},
we have to divide the left-hand side by a product of factors analogous to
\eqref{eq:div1} and \eqref{eq:div2}. More precisely,
\eqref{eq:div1} becomes
	\eqn{\label{eq:div1l}
	\prod_{i=1}^l \left(\frac{D_{v_{x_i},j_i-1}(v_{\bar s_i})+\delta }{c_{v_{x_i},j_i}}\right)
	\left(1-\frac{D_{v_{x_i},j_i-1}(H)
	+|H \cap [v_{x_i}-1]|\delta}{c_{v_{x_i},j_i}}\right)^{-1} ,
	}
while \eqref{eq:div2} becomes
	\begin{equation*}
	\prod_{i=1}^l \frac{D_{t}(v_{\bar s_i})+\delta}{D_{v_{x_i},j_i-1}(v_{\bar s_i}) +\delta} \,.
	\end{equation*}
We define $\beta$ accordingly,
namely we take the product for $i=1,\ldots, l$
of \eqref{eq:betadef} with $x, j, \bar s$ replaced respectively by $x_i, j_i, \bar s_i$.
Then it is easy to show that
	\begin{equation*}
	\beta \le \left(\frac{\alpha(t)}{2(\log t)^{1+\varepsilon}}\right)^l \,,
	\end{equation*}
arguing as in the case $l=1$. This completes the proof of \eqref{eq:c1++l}.
\qed
\medskip

\subsection{Proof of Statement \ref{stat-triplelog}}
\label{proof-upper-prefatt-2}

The next step is to prove that the boundaries of the $k^{+}_t$-exploration graphs are at most at distance
\begin{equation}\label{eq:ht}
	h_t=\lceil B\log\log\log t+C\rceil
\end{equation}
from $\core_t$, where $B, C$ are constants to be chosen later on.
Similarly to the proof in Section \ref{proof-upper-cmnd-2}, we consider a $k^{+}_t$-exploration graph,
and we enumerate the vertices on the boundary as $x_1, \ldots,x_N$, where $N\geq s(m,l)m^{k^{+}_t}$ from Lemma \ref{prefatt-lemma-boundarybound} and $l$ is chosen as in 
Lemma~\ref{prefatt-corol-numbercollision}. We next define what it means to have a success:

\begin{Definition}[Success]
\label{prefatt-def-success}
Consider the vertices $x_1,\ldots,x_N$ on the boundary of a $k^{+}_t$-exploration graph. We say that 
$x_i$ is a {\em success} when the distance between $x_i$ and $\core_t$ is at most $2h_t$.
\end{Definition}

The next lemma is similar to Lemma \ref{cmnd-lemma-hitting}
(but only deals with vertices in $[t/2]$):

\begin{Lemma}[Probability of success]
\label{prefatt-lemma-hitting}
 Let $(\PA{t})_{t\geq 1}$ be a preferential attachment model, with $m\geq 2$ and $\delta\in(-m,0)$. Consider $v\in [t/2]\setminus \core_t$ and its $k^{+}_t$-exploration graph. Then there
 exists a constant $\eta>0$ such that
 	\eqn{
	\label{prefatt-formula-hitting1}
  	\pr\left(S_{x_1}\mid \PA{t/2}\right) \geq \eta,
 	}
and for all $j=2,\ldots,N$,
	\eqn{
	\label{prefatt-formula-hitting2}
   	\pr\left(S_{x_1} \mid \PA{t/2}, S_{x_1}^c,\ldots,S_{x_{j-1}}^c\right) \geq \eta.
 	}
\end{Lemma}

The aim is to define a sequence of vertices $w_0, \ldots,w_h$ that connects a vertex $x_i$ on the boundary with $\core_t$. In order to do this, we need some preliminar results. We start with the crucial definition of a $t$-connector:

\begin{Definition}[$t$-connector]
\label{prefatt-def-tconn}
Let $(\PA{t})_{t\geq 1}$ be a preferential attachment model, with $m\geq 2$. 
Consider two subsets $A,B\subseteq [t/2]$, with $A\cap B=\varnothing$. We say that a vertex $j\in [t]\setminus[t/2]$ is a {\em $t$-connector for $A$ and $B$} if at least one of the edges incident to $j$ is attached to a vertex in $A$ and at least one is attached to a vertex in $B$.
\end{Definition}

The notion of $t$-connector is useful, because, unlike in the configuration model, in the preferential attachment model typically two high-degree vertices are not directly connected. From the definition of the preferential attachment model, it is clear that the older vertices have with high probability large degree, and the younger vertices have lower degree. When we add a new vertex,
this is typically attached to vertices with large degrees. This means that, 
with high probability, two vertices with high degree can be connected by a young vertex, which is the $t$-connector.

A further important reason for the usefulness of $t$-connectors is that we have effectively {\em decoupled} the preferential attachment model at time $t/2$ and what happens in between times $t/2$ and $t$. When the sets $A$ and $B$ are appropriately chosen, then each vertex will be 
a $t$-connector with reasonable probability, and the events that distinct
vertices are $t$-connectors are close to being independent. Thus, we can use comparisons to binomial random variables to investigate the existence of $t$-connectors. In order to make this work, we need to identify the structure of $\PA{t/2}$ and show that it has sufficiently many vertices of large degree, and we need to show that $t$-connectors are likely to exist. We start with the latter.

In more detail, we will use $t$-connectors to generate the sequence of vertices $w_1,\ldots,w_h$ between the boundary of a $k^+_n$-exploration graph and the $\core_t$, in the sense that we use a $t$-connector to link the vertex $w_i$ to the vertex $w_{i+1}$. (This is why we define a vertex $x_i$ 
to be a success if its distance from $\core_t$ is at most $2h_t$, instead of $h_t$.) 
We rely on a result implying the existence of $t$-connectors between sets of high total degree:

\begin{Lemma}[Existence of $t$-connectors]
\label{prefatt-lemma-existencetconn}
 Let $(\PA{t})_{t\geq 1}$ be a preferential attachment model, with $m\geq 2$ and $\delta\in(-m,0)$. There exists a constant $\mu>0$ such that, for every $A\subseteq[t/2]$, and $i\in[t/2]\setminus A$,
 	\eqn{
	\label{prefatt-formula-tconnprob}
  	\pr\Big(\nexists j\in[t]\setminus[t/2]\colon j\mbox{ is a }t\mbox{-connector for }i\mbox{ and }A\mid \PA{t/2}\Big) 
      	\leq \mathrm{exp}\left(-\frac{\mu D_A(t/2)D_i(t/2)}{t}\right),
 	}
 where $D_A(t/2) = \sum_{v\in A}D_v(t/2)$ is the total degree of $A$ at time $t/2$.
\end{Lemma}

\begin{proof} The proof of this lemma is present in the proof of \cite[Proposition 3.2]{DSvdH}.
\end{proof}

\begin{remark}\rm
Notice that this bound depends on the fact that the number of possible $t$-connectors is of order $t$.
\end{remark}

A last preliminary result that we need is a technical one, which plays the role of Lemma \ref{cmnd-lemma-highvertices} for the configuration model and shows that at time $t/2$ there are sufficiently many vertices of high degree, uniformly over a wide range of what `large' could mean:

\begin{Lemma}[Tail of degree distribution]
\label{prefatt-lemma-tail}
 Let $(\PA{t})_{t\geq 1}$ be a preferential attachment model, with $m\geq 2$ and $\delta\in(-m,0)$. Then, for all $\zeta>0$ there exists a constant $c = c(\zeta)$ such that,
 for all $1\leq x\leq (\log t)^{q}$, for any $q>0$, and uniformly in $t$,
 	\eqn{
	\label{prefatt-formula-tail}
   	P_{\geq x}(t) = \frac{1}{t}\sum_{v\in[t]}\I_{\{D_v(t)\geq x\}} \geq cx^{-(\tau-1+\zeta)}.
 	}
\end{Lemma}

\begin{proof}
The degree distribution sequence $(p_k)_{k\in\N}$ in \eqref{prefatt-empdegdist} satisfies a power law with exponent $\tau\in(2,3)$. As a consequence,
for all $\zeta>0$ there exists a constant $\bar{c}= \bar{c}(\zeta)$ such that
	\eqn{
	\label{tail-pk}
  	p_{\geq x} := \sum_{k\geq x}p_k \geq \bar{c}x^{-(\tau-1+\zeta)}.
	}
We now use a concentration result on the empirical degree distribution (for details, see \cite[Theorem 8.2]{vdH1}), which assures us that there exists a second constant $C>0$ such that, with high probability, for every $x\in\N$, 
	\eqn{
	\left|P_{\geq x}-p_{\geq x}\right|\leq C\sqrt{\frac{\log t}{t}}.
	}
Fix now $\zeta>0$, then from this last bound we can immediately write, for a suitable constant $\bar{c}$ as in \eqref{tail-pk}, 
	\eqn{
	P_{\geq x}\geq p_{\geq x}-C\sqrt{\frac{\log t}{t}}\geq \bar{c}x^{-(\tau-1+\zeta)}-C\sqrt{\frac{\log t}{t}}\geq 
		\frac{\bar{c}}{2}x^{-(\tau-1+\zeta)},
	}
if and only if 
	\eqn{
	C\sqrt{\frac{\log t}{t}} = o\left(x^{-(\tau-1+\zeta)}\right).
	}
This is clearly true for $x\leq (\log t)^q$, for any positive $q$. Taking $c = \bar{c}/2$ completes the proof.
\end{proof}

With the above tools, we are now ready to complete the proof of Lemma \ref{prefatt-lemma-hitting}:
\medskip
\begin{proof}[Proof of Lemma \ref{prefatt-lemma-hitting}]
As in the proof of Proposition \ref{cmnd-lemma-hitting},
we define the super-exponentially growing sequence $g_\ell$ as in \eqref{eq:gell},
where $\gamma > 0$ is chosen small enough, as well as $\zeta > 0$,
so that \eqref{eq:xifix} holds. The constants $B$ and $C$ in the definition
\eqref{eq:ht} of $h_t$ are fixed as prescribed below \eqref{eq:gell}.

We will define a sequence of vertices $w_0,\ldots,w_h$ such that, for $i=1,\ldots,h$, $D_{w_i}(t)\geq g_{i}$ and $w_{i-1}$ is connected to $w_i$. For this, we define, for $i=1,\ldots,h-1$,
	\eqn{
	\label{prefatt-fastly-1}
  	H_i = \Big\{u\in[t] \colon D_{u}(t/2)\geq g_i\Big\}
	\subseteq [t/2],
	}
so that we aim for $w_i\in H_i$.
\medskip

We define the vertices recursively, and start with $w_0=x_1$. Then, we consider $t$-connectors between $w_0$ and $H_1$, and denote by $w_1$ the vertex in $H_1$ with minimal degree among the ones that are connected to $w_0$ by a $t$-connector. Recursively, consider $t$-connectors between $w_i$ and $H_{i+1}$, and denote by $w_{i+1}$ the vertex in $H_{i+1}$ with minimal degree among the ones that are connected to $w_i$ by a $t$-connector.  
Recall \eqref{eq:gbo}  to see that $g_{h_t}\geq (\log{t})^{\sigma}$, where $h_t$ is defined in \eqref{eq:ht}.
The distance between $w_0$ and $\core_t$ is at most $2h_t = 2\lceil B\log \log \log t+C\rceil$. If we denote the event that there exists a $t$ connector between $w_{i-1}$ and $H_i$ by $\{w_{i-1}\sim H_i\}$, then we will bound from below
	\eqn{
	\label{prefatt-fastly-2}
  	\pr(S_{x_1}\mid \PA{t/2}) \geq \E\Big[\prod_{i=1}^{h_t} \I_{\{w_{i-1}\sim H_i\}}\mid \PA{t/2}\Big].
	}
In Lemma \ref{prefatt-lemma-existencetconn}, the bound on the probability that a vertex  $j\in[t]\setminus[t/2]$ is a $t$-connector between two subsets of $[t]$ is independent of the fact that the other vertices are $t$-connectors or not. This means that, with $\mathcal{F}_i$ the $\sigma$-field generated by the path formed by $w_0,\ldots,w_i$ and their respective $t$-connectors,
	\eqn{
	\label{prefatt-fastly-3}
  	\E\Big[\I_{\{w_{i-1}\sim H_i\}}\mid\PA{t/2},\mathcal{F}_{i-1}\Big]
      	\geq 1-\e^{-\mu D_{w_{i-1}}(t/2)D_{H_i}(t/2)/t},
	}
where $D_{H_i}(t)=\sum_{u\in H_i}D_u(t/2)$. This means that
	\eqn{
	\label{prefatt-fastly-4}
  	\E\Big[\prod_{i=1}^{h_t} \I_{\{w_{i-1}\sim H_i\}}\mid \PA{t/2}\Big] \geq 
      	\prod_{i=1}^{h_t}\Big(1-\e^{-\mu D_{w_{i-1}}(t/2)D_{H_i}(t/2)/t}\Big).
	}
We have to bound every term in the product. Using Lemma \ref{prefatt-lemma-tail}, for $i=1$,
	\eqn{
	\label{prefatt-fastly-5}
  	1-\e^{-\mu D_{w_0}(t/2)D_{H_1}(t/2)/t}\geq 1-\e^{-\mu D_{w_0}(t/2) g_1 P_{\geq g_1}(t/2)},
	}
while, for $i=2,\ldots,h-1$
	\eqn{
	\label{prefatt-fastly-6}
	1-\e^{-\mu D_{w_{i-1}}(t/2)D_{H_i}(t/2)/t}\geq 1-\e^{-\mu g_{i-1}g_iP_{\geq g_i}(t/2)}.
	}
Applying \eqref{prefatt-formula-tail} and  recalling \eqref{qj-bd-1}--\eqref{qj-bd-2}, the result is
	\eqn{
	\label{prefatt-fastly-7}
	\begin{array}{cl}
  	\grosso\pr( S_{x_1}\mid \PA{t}) 
	& \grosso \geq \Big(1-\e^{-\mu D_{w_0}(t/2)g_1P_{\geq g_1}(t/2)}\Big)
	\prod_{i=2}^{h_t} \Big(1-\e^{-\mu g_{i-1}g_iP_{\geq g_i}(t/2)}\Big)\\
  	& \grosso\geq \Big(1-\e^{-\mu m g_1 P_{\geq g_1}(t/2)}\Big)
	\prod_{i=2}^\infty \Big(1 - \e^{-\tilde c \, (g_i)^\xi}\Big) ,
	\end{array}
	}
for some constant $\tilde c$.
Since $h_t= \lceil B\log \log \log t+C\rceil$, and 
	\eqn{
	\grosso P_{\geq g_1}(t/2) \rightarrow \sum_{k\geq g_1}p_k>0
	}
with high probability as $t\rightarrow\infty$, we can find a constant $\eta$ such that
	\eqn{
	\label{prefatt-fastly-8}
  	\Big(1-\e^{-\eta m g_1 P_{\geq g_1}(t/2)}\Big)\prod_{i=2}^{h_t} 
	\Big(1 - \e^{-\tilde c \, (g_i)^\xi}\Big)> \eta >0,
	}
which proves \eqref{prefatt-formula-hitting1}.

To prove \eqref{prefatt-formula-hitting2}, we observe that all the lower bounds that we have used on the probability of existence of $t$-connectors only depend on the existence of sufficiently many potential $t$-connectors. Thus, it suffices to prove that, on the event $S_{x_1}^c\cap \cdots\cap S_{x_{j-1}}^c$, we have not used too many vertices as $t$-connectors. On this event, we have used at most $h_t\cdot (j-1)$
vertices as $t$-connectors, which is $o(t)$. Thus, this means that, when we bound the probability of $S_{x_j}$, we still have 
$t-h_t\cdot (j-1)$ possible $t$-connectors, where $j$ is at most $(\log t)^{1+\varepsilon}$. Thus, with the same notation as before,
	\eqn{
	\label{prefatt-fastly-9}
   	\E\Big[\I_{\{w_{i-1}\sim H_i\}}\mid\PA{t/2},S_{x_1}^c,\ldots,S_{x_{j-1}}^c\Big]\geq 
   	1-\e^{-\mu D_{w_{i-1}}(t/2)D_{H_i}(t/2)/t},
	}
so that we can proceed as we did for $S_{x_1}$. We omit further details.
\end{proof}

We are now ready to identify the distance between the vertices outside the core and the core:

\begin{Proposition}[Distance between periphery and $\core_t$]
 Let $(\PA{t})_{t\geq 1}$ be a preferential attachment model with $m\geq 2$ and $\delta\in(-m,0)$. 
Then, with high probability and for all $v\in[t]\setminus\core_t$,
 	\eqn{
	\label{prefatt-formula-perifery}
  	\dist_{\PA{t}}(v,\core_t)\leq k^{+}_t+2h_t.
	 }
\end{Proposition}

\begin{proof}
We start by analyzing $v\in [t/2]$.  By Lemma \ref{prefatt-lemma-oldvertices}, with high probability there exists $a\in(0,1]$ such that $[t^a]\subseteq\core_t$.
Consider $l>1/a$, and fix a vertex $v\in[t/2]$. Then, by 
Lemma~\ref{prefatt-corol-numbercollision} and with high probability, the $k^{+}_t$-exploration graph starting from $v$  has at most 
$l$ collisions before hitting $\core_t$. By Lemma \ref{prefatt-lemma-boundarybound} and with high probability, the number of vertices on the boundary of the $k^{+}_t$-exploration graph is at least
$N = s(m,l)(\log t)^{1+\varepsilon}$. It remains to bound the probability that none of the $N$ vertices on the boundary is a success, meaning that it does not reach $\core_t$ in at most $2h_t=2\lceil B\log\log{t}+C\rceil$ steps. 

By Lemma \ref{prefatt-lemma-hitting},
	\eqn{
	\label{prefatt-fastly-10}
	\pr(S_{x_1}^c\cap\cdots\cap S_{x_N}^c\mid\PA{t/2}) \leq (1-\eta)^N = o(1/t),
	}
thanks to the bound $N \ge s(m,l) (\log t)^{1+\epsilon}$.
This means that the probability that there exists a vertex $v\in[t/2]$ such that 
its $k^+_n$-exploration graph is at distance more than $A\log\log\log t$ from $\core_t$ is $o(1)$. 
This proves the statement for all $v\in [t/2]$.

Next, consider a vertex $v\in[t]\setminus[t/2]$. 
Lemma~\ref{prefatt-corol-numbercollision} implies that the probability that there exists a vertex $v\in[t]\setminus[t/2]$ such that its $k^{+}_t$-exploration graph contains more than one collision before hitting $\core_t\cup[t/2]$ is $o(1)$. As before, the number of vertices on the boundary of a $k^{+}_t$-exploration graph starting at $v\in[t]\setminus[t/2]$ is at least $N \geq  s(m,1)m^{k^+_n} = s(m,1)(\log t)^{1+\varepsilon}$.
We denote these vertices by $x_1,\ldots,x_N$. We aim to show that, with high probability,
	\eqn{
	\label{prefatt-fastly-11}
  	\Delta_{\sss N} = \sum_{i=1}^N \I_{(x_i\in[t/2])}\geq N/4.
	}
For every $i=1,\ldots,N$, there exists a unique vertex $y_i$ such that $y_i$ is in the $k^{+}_t$-exploration graph and it is attached to $x_i$. Obviously, if $y_i\in[t/2]$ then also $x_i\in[t/2]$, since
$x_i$ has to be older than $y_i$. If $y_i\not\in [t/2]$, then
	\eqn{
	\label{prefatt-fastly-12}
  	\pr\left(x_i\in[t/2]\mid \PA{y_i-1}\right) = \pr\left(y_i\rightarrow[t/2]\mid \PA{y_i-1} \right)\geq \frac{1}{2},
	}
and this bound does not depend on the attaching of the edges of the other vertices $\{y_j\colon j\neq i\}$. 

This means that we obtain the stochastic domination
	\eqn{
	\label{prefatt-fastly-13}
  	\Delta_{\sss N}\geq  \sum_{i=1}^N \I_{(x_i\in[t/2])} \succeq \mathrm{Bin}\big(N,\frac{1}{2}\big),
	}
where we write that $X \succeq Y$ when the random variable $X$ is stochastically larger than $Y$. By concentration properties of the binomial, $\mathrm{Bin}\big(N,\frac{1}{2}\big)\geq N/4$
with probability at least 
	\eqn{
	\label{prefatt-fastly-14}
 	1-\e^{-N/4} 
	= 1-\e^{-s(m,1)(\log t)^{1+\varepsilon}/4}
  	=1-o(1/t).
	}
Thus, the probability that none of the vertices on the boundary intersected with $[t/2]$ is a success is bounded by
	\eqn{
	\label{prefatt-fastly-15}
	\pr\big(S_{x_1}^c\cap \cdots \cap S_{x_{\Delta_N}}^c\mid \PA{t/2}\big) \leq (1-\eta)^{N/4} +o(1/t)=
	o(1/t).
	}
We conclude that the probability that there exists a vertex in $[t]\setminus[t/2]$ such that it is at distance more than $k^{+}_t+2h_t$ from $\core_t$ is $o(1)$.
\end{proof}
\medskip

This completes the proof of Statement \ref{stat-triplelog}, and thus of Theorem \ref{main-prefatt}.
\qed

\longversion{\appendix
\section{Proof of Proposition \ref{prefatt-prop-bounddistance}}
\label{section-appendix}

We prove Proposition \ref{prefatt-prop-bounddistance}.
As mentioned in Section \ref{lower-proof-prefatt-3}, the proof of
\eqref{prefatt-bounddistance} is an adaptation of an argument in \cite[Section 4.1]{DSMCM}.
The final aim is to prove that \eqref{prefatt-lowdistfix-6} is $o(1)$. 

\smallskip

First, we start with a technical lemma. Let us fix $R \in (0,\infty)$ and define
	\begin{equation}\label{eq:pnl}
	p(n,l) = R(n\wedge l)^{-\gamma}(n\vee l)^{\gamma-1} \,.
	\end{equation}
Our interest is for $\gamma=m/(2m+\delta)\in (1/2,1)$, so that $\gamma \in (1/2, 1)$
(because $\delta \in (-m,0)$).

\begin{Lemma}
\label{prefatt-appendix-result1}
 Let $\gamma\in(1/2,1)$ and suppose that $2\leq g\leq t$, $\alpha,\beta\geq0$ and $q\colon [t]\rightarrow[0,\infty)$ satisfy
 	\eqn{
	\label{prefatt-appendix-1}
   	q(n)\leq \I_{\{n\geq g\}}\left(\alpha n^{-\gamma}+\beta n^{\gamma-1}\right)
 	}
for all $n\in[t]$. Then 
there exists a constant $c= c(R,\gamma)>1$ such that,  for all $l\in[t]$,
 	\eqn{
	\label{prefatt-appendix-2}
   	\sum_{n=1}^tq(n)p(n,l)\leq c\big(\alpha\log(t/g)+\beta t^{2\gamma-1}\big)l^{-\gamma}
	+c\I_{\{l>g\}}
      	\big(\alpha g^{1-2\gamma}+\beta\log(t/g)\big)l^{\gamma-1}.
 	}
\end{Lemma}

\begin{proof}
We split
	\eqn{
	\label{prefatt-appendix-3}
  	\sum_{n=1}^tq(n)p(n,l) = \sum_{n = g\vee l}^tq(n)p(n,l) 
	+\I_{\{l>g\}}\sum_{n=g}^{l-1}q(n)p(n,l),
	}
because $q(n)=0$ when $n<g$. Therefore,
	\eqn{
	\label{prefatt-appendix-4}
  	\sum_{n = g\vee l}^tq(n)p(n,l) 
	= \sum_{n=g\vee l}^t q(n)R(n\wedge l)^{-\gamma}(n\vee l)^{\gamma-1}
	\leq 
    	\sum_{n=g\vee l}^t\left(\alpha n^{-\gamma}+\beta n^{\gamma-1}\right)Rn^{\gamma-1}l^{-\gamma},
	}
because the sum is over $n\geq g\vee l\geq l$. 
For the other term, since we may assume that $l>g$,
	\eqn{
	\label{prefatt-appendix-5}
  	\I_{\{l>g\}}\sum_{n=g}^{l-1}q(n)p(n,l) 
	\leq \I_{\{l>g\}}\sum_{n=g}^{l-1} \left(\alpha n^{-\gamma}
	+\beta n^{\gamma-1}\right)Rn^{-\gamma}l^{\gamma-1}.
	}
This means that $\sum_{n=1}^tq(n)p(n,l)$ is bounded above by
	\eqan{
	\label{prefatt-appendix-6}
	&
  	\grosso R\Big[\alpha\sum_{n=g\vee l}^tn^{-1}
	+\beta\sum_{n=g\vee l}^tn^{2\gamma-2}\Big] l^{-\gamma} 
	+\I_{\{l>g\}}R\Big[\alpha\sum_{n=g}^{l-1}n^{-2\gamma}
	+\beta\sum_{n=g}^{l-1}n^{-1}\Big]l^{\gamma-1}\\
 	&\grosso \quad \leq c_1\left[\alpha\log\left(t/g\right)
	+\beta t^{2\gamma-1}\right]l^{-\gamma}
	+\I_{\{l>g\}}c_2\left[\alpha g^{1-2\gamma}
	+\beta\log\left(t/g\right)\right] l^{\gamma-1},\nn
	}
where we have used that $\gamma > 1/2$.
We take $c=\max(c_1,c_2)$ to obtain the statement.
\end{proof}
\medskip

We now define recursively the sequences $(\alpha_k)_{k\in\N}$, $(\beta_k)_{k\in\N}$
and $(g_k)_{k\in\N}$, for which we will prove the bound \eqref{prefatt-lowdistfix-11}.
This will allow us to control \eqref{prefatt-lowdistfix-6}.

\begin{Definition}
\label{prefatt-def-sequencerecursive-app}
We define
	\eqn{
	\label{prefatt-lowdistfix-7-app}
  	\begin{array}{ccccc}
    	\grosso g_0 = \left\lceil\frac{t}{(\log t)^2}\right\rceil, 
	& & \grosso\alpha_1 = R\left(g_0\right)^{\gamma-1}, 
	& &
	  \grosso\beta_1 = R\left(g_0\right)^{-\gamma},
 	\end{array}
	} 
and recursively, for $k\geq 1$:
\begin{enumerate}
	\item $g_k$ is the smallest integer such that
    		\eqn{
		\label{prefatt-lowdistfix-8-app}
      		\frac{1}{1-\gamma}\alpha_k g_k^{1-\gamma}\geq \frac{6}{\pi^2k^2(\log t)^2};
    		}
	\item
		\eqn{
		\label{prefatt-lowdistfix-9-app}
  		\alpha_{k+1} = c\big(\alpha_k\log(t/g_k)+\beta_k t^{2\gamma-1}\big);
		}
	\item 
		\eqn{
		\label{prefatt-lowdistfix-10-app}
  		\beta_{k+1} = c\big(\alpha_k g_k^{1-2\gamma}+\beta_{k}\log(t/g_k)\big),
		}
\end{enumerate}
where $c= c(R,\gamma)>1$ is the same 
constant appearing in Lemma~\ref{prefatt-appendix-result1}.
\end{Definition}
\medskip

One can check that $k\mapsto g_k$ is non-increasing, while $k\mapsto \alpha_k, \beta_k$ 
are non-decreasing.

We recall that $f_{k,t}(x,l)$ was introduced in \eqref{prefatt-lowdistfix-4},
with $p(z,w)$ defined in \eqref{prefatt-paths-minmax} (where we set
$R = Cm$, to match with \eqref{eq:pnl}). As a consequence,
the following recursive relation is satisfied:
	\eqn{
	\label{prefatt-lowdistfix-5-app}
	\forall k \ge 1: \qquad
  	f_{k+1,t}(x,w) = \sum_{z=g_k}^tf_{k,t}(x,z)p(z,w),
	}
where $p(z,w)$ is given in \eqref{eq:pnl}.
The following lemma derives recursive bounds on $f_{k,t}$.

\begin{Lemma}[Recursive bound on $f_{k,t}$]
\label{prefatt-appendix-result2}
For the sequences in Definition \ref{prefatt-def-sequencerecursive-app}, for every $l\in[t]$ and $k\in\N$, 
 	\eqn{
	\label{prefatt-appendix-7}
   	f_{k,t}(x,l)\leq \alpha_kl^{-\gamma}+\I_{\{l>g_{k-1}\}}\beta_k l^{\gamma-1}.
 	}
\end{Lemma}

\begin{proof} We prove \eqref{prefatt-appendix-7} by induction on $k$. For $k=1$, using $\alpha_1 = Rg_0^{\gamma-1}$ and $\beta_1=Rg_0^{-\gamma}$,
	\eqn{
	\label{prefatt-appendix-8}
  	f_{1,t}(x,l) = p(x,l) \I_{\{x\geq g_0\}}
	\leq R(g_0)^{\gamma-1}l^{-\gamma}+\I_{\{l>g_0\}}R(g_0)^{-\gamma}l^{\gamma-1} 
	= \alpha_1l^{-\gamma}+\I_{\{l>g_0\}}\beta_1l^{\gamma-1},
	}
as required. This initiates the induction hypothesis. 
We now proceed with the induction: suppose that $g_{k-1}$, $\alpha_k$ and $\beta_{k}$ are such that
	\eqn{
	\label{prefatt-appendix-9}
  	f_{k,t}(x,l)\leq \alpha_k l^{-\gamma}+\I_{(l>g_{k-1})}\beta_kl^{\gamma-1}.
	}
We use the recursive property of $f_{k,t}$ in \eqref{prefatt-lowdistfix-5-app}.
We apply Lemma \ref{prefatt-appendix-result1}, with $g = g_k$ and $q(n) = f_{k,t}(x,n)\I_{\{n\geq g_k\}}$,
so, by Definition \ref{prefatt-def-sequencerecursive-app}, 
	\eqan{
	\label{prefatt-appendix-12}
	f_{k+1,t}(x,l)&\leq c\Big[\alpha_k\log(t/g_k)+\beta_k t^{2\gamma-1}\Big]l^{-\gamma}+
	c\I_{\{l>g_k\}}\Big[\alpha_k g_k^{1-2\gamma}+\beta_k\log(t/g_k)\Big] l^{\gamma-1}\\
	&\qquad
  	\grosso = \alpha_{k+1} l^{-\gamma}+\I_{\{l>g_k\}}\beta_{k+1} l^{\gamma-1}.\nn
	}
This advances the induction hypothesis, and thus completes the proof.
\end{proof}

In order to proceed, we define $\eta_k = t/g_k$.
We aim to derive a bound on the growth of $\eta_k$.

\begin{Lemma}[Recursive relation of $\eta_k$]
\label{prefatt-appendix-result3}
Let $\eta_k=t/g_k$ be defined as above. Then there exists a constant $C>0$ such that
 	\eqn{
	\label{prefatt-appendix-13}
    	\eta_{k+2}^{1-\gamma}\leq C\left[\eta_k^\gamma +\eta_{k+1}^{1-\gamma}\log \eta_{k+1}\right].
	 }
\end{Lemma}

\begin{proof}
By definition of $g_k$ in \eqref{prefatt-lowdistfix-8-app},
	\eqn{
	\label{prefatt-appendix-14}
  	\eta_{k+2}^{1-\gamma} = t^{1-\gamma}g_{k+2}^{\gamma-1}\leq t^{1-\gamma} 
	\frac{1}{1-\gamma}\frac{\pi^2(\log t)^2}{6}(k+2)^2\alpha_{k+2}.
	}
By definition of $\alpha_k$ in \eqref{prefatt-lowdistfix-9-app},
	\eqan{
	\label{prefatt-appendix-15}
 	& t^{1-\gamma}\frac{1}{1-\gamma}\frac{\pi^2(\log t)^2}{6}(k+2)^2\alpha_{k+2}\\
	&\qquad= t^{1-\gamma}\frac{c}{1-\gamma}\frac{\pi^2(\log t)^2}{6}(k+2)^2\left[\alpha_{k+1}\log\eta_{k+1}
	+\beta_{k+1}t^{2\gamma-1}\right].\nn
	}
By definition of $g_{k}$, relation \eqref{prefatt-lowdistfix-8-app}
holds with the opposite inequality if we replace $g_k$ by $g_k-1$ in the left hand side.
This, with $k+1$ instead of $k$, yields
	\eqn{
	\label{prefatt-appendix-16}
  	\alpha_{k+1}\leq \frac{6(1-\gamma)}{\pi^2(k+1)^2(\log t)^2}(g_{k+1}-1)^{\gamma-1}.
	}
Since $\alpha_{k+1}\geq 2$, we must have that $g_{k+1}\geq 2$, so that
	\eqn{
	\label{prefatt-appendix-16bis}
	(g_{k+1}-1)^{\gamma-1} \leq 2^{1-\gamma}g_{k+1}^{\gamma-1}.
	}
We conclude that
	\eqan{
	\label{prefatt-appendix-17}
  	&t^{1-\gamma}\frac{c}{1-\gamma}\frac{\pi^2(\log t)^2}{6}
	(k+2)^2\alpha_{k+1}\log\eta_{k+1}\\
  	&\qquad\leq t^{1-\gamma} \frac{c2^{1-\gamma}}{1-\gamma}
	\frac{\pi^2(\log t)^2}{6}(k+2)^2
	\frac{6(1-\gamma)}{\pi^2(k+1)^2(\log t)^2}g_{k+1}^{\gamma-1}\log\eta_{k+1}\nn\\
  	&\qquad= c2^{1-\gamma} \frac{(k+2)^2}{(k+1)^2}\eta_{k+1}^{1-\gamma}\log\eta_{k+1}.\nn
	}
We now have to bound the remaining term in \eqref{prefatt-appendix-15}, which equals
	\eqn{
	\label{prefatt-appendix-18}
  	t^{1-\gamma}\frac{c}{1-\gamma}\frac{\pi^2(\log t)^2}{6}(k+2)^2\beta_{k+1}t^{2\gamma-1} 
	= \frac{c}{1-\gamma}\frac{\pi^2(\log t)^2}{6}(k+2)^2\beta_{k+1}t^{\gamma}.
	}
We use the definition of $\beta_k$ in \eqref{prefatt-lowdistfix-10-app} to write
	\eqn{
	\label{prefatt-appendix-19}
  	\frac{c}{1-\gamma}\frac{\pi^2(\log t)^2}{6}(k+2)^2\beta_{k+1}t^{\gamma} 
	= \frac{c}{1-\gamma}\frac{\pi^2(\log t)^2}{6}(k+2)^2t^{\gamma}
      	c\left[\alpha_k g_k^{1-2\gamma}+\beta_k\log\eta_k\right],
	}
and again use the fact that $\alpha_k \leq 2^{1-\gamma}\frac{6(1-\gamma)}{\pi^2k^2 (\log t)^2}g_k^{\gamma-1}$,
so that
	\eqan{
	\label{prefatt-appendix-20}
	&\frac{c}{1-\gamma}\frac{\pi^2(\log t)^2}{6}(k+2)^2t^{\gamma} c \alpha_k
	g_k^{1-2\gamma}\\
	&\qquad\leq \frac{c2^{1-\gamma}}{1-\gamma}\frac{\pi^2(\log t)^2}{6}
	(k+2)^2t^{\gamma} c
      \frac{6(1-\gamma)}{\pi^2k^2(\log t)^2}g_k^{\gamma-1}g_k^{1-2\gamma} 
	= c^2 2^{1-\gamma}\frac{(k+2)^2}{k^2}\eta^{\gamma}_k.\nn
	}
By Definition \ref{prefatt-def-sequencerecursive-app}, we have $c\beta_kt^{2\gamma-1}\leq \alpha_{k+1}$,
so that, using \eqref{prefatt-appendix-16} and \eqref{prefatt-appendix-16bis},
	\eqan{
	\label{prefatt-appendix-21}
	&\frac{c}{1-\gamma}\frac{\pi^2(\log t)^2}{6}(k+2)^2t^{\gamma}c\beta_k\log\eta_k\\
	&\qquad\leq \frac{c}{1-\gamma}\frac{\pi^2(\log t)^2}{6}(k+2)^2\alpha_{k+1}t^{1-\gamma}\log\eta_k 
	\leq c2^{1-\gamma}\frac{(k+2)^2}{(k+1)^2}g_{k+1}^{\gamma-1}t^{1-\gamma}\log \eta_k\nn\\
	&\qquad= c2^{1-\gamma}\frac{(k+2)^2}{(k+1)^2}\eta_{k+1}^{1-\gamma}\log \eta_k.\nn
	}
Since $k\mapsto\eta_k$ is increasing,
	\eqn{
	\label{prefatt-appendix-22}
  	c2^{1-\gamma}\frac{(k+2)^2}{(k+1)^2}\eta_{k+1}^{1-\gamma}\log \eta_k
	\leq c2^{1-\gamma}\frac{(k+2)^2}{(k+1)^2}\eta_{k+1}^{1-\gamma}\log \eta_{k+1}.
	}
Putting together all the bounds and taking a different constant $C= C(\gamma)$, we obtain \eqref{prefatt-appendix-13}.
\end{proof}

We can now obtain a useful bound on the growth of $\eta_k$.

\begin{Lemma}[Inductive bound on $\eta_k$]
\label{prefatt-appendix-result4}
Let $(\eta_k)_{k\in\N}$ be given by $\eta_k = t/g_k$ and let $\kappa=\gamma/(1-\gamma)\in (1,\infty)$. Then, there exists a constant $B\geq 2$ such that
 	\eqn{
	\label{prefatt-appendix-23}
  	\eta_k\leq \mathrm{exp}\big(B(\log\log t)\kappa ^{k/2}\big).
	}
for any $k=O(\log\log t)$.
\end{Lemma}

\begin{proof}
We prove the lemma by induction on $k$, and start by initializing the induction. For $k=0$,
	\eqn{
	\label{prefatt-appendix-24}
  	\eta_0 = t/g_0= \frac{t}{\left\lceil\frac{t}{(\log t)^2}\right\rceil}\leq (\log t)^2 
	= \e^{2\log\log t} \leq \e^{B\log\log t},
}
for any $B\geq 2$, which initializes the induction.

We next suppose that the statement is true for $l=1,\ldots,k-1,$ and will prove it for $k$. 
Using that
	\eqn{
	\label{prefatt-appendix-26}
  	(z+w)^{\frac{1}{1-\gamma}}\leq 2^{\frac{1}{1-\gamma}}
	\left(z^{\frac{1}{1-\gamma}}+w^{\frac{1}{1-\gamma}}\right),
	}
we can write by Lemma~\ref{prefatt-appendix-result3}, for a different constant $C$,
	\eqn{
	\label{prefatt-appendix-27}
  	\eta_k\leq C\big[\eta_{k-2}^{\frac{\gamma}{1-\gamma}}
	+\eta_{k-1}(\log \eta_{k-1})^{\frac{1}{1-\gamma}}\big]
	=C\big[\eta_{k-2}^{\kappa}
	+\eta_{k-1}(\log \eta_{k-1})^{\frac{1}{1-\gamma}}\big].
	}
Using this inequality, we can write
	\eqn{
	\label{prefatt-appendix-28}
  	\eta_{k-2}\leq C\big[\eta_{k-4}^{\kappa}+\eta_{k-3}(\log \eta_{k-3})^{\frac{1}{1-\gamma}}\big],
	}
so that, by $(z+w)^\kappa \le 2^\kappa (z^\kappa + w^\kappa)$,
	\eqn{
	\label{prefatt-appendix-29}
  	\eta_k\leq C (2 C)^{\kappa}\big[\eta_{k-4}^{\kappa^2}+\eta_{k-3}^{\kappa}
      	(\log \eta_{k-3})^{\frac{\kappa}{1-\gamma}}\big]+C\eta_{k-1}(\log \eta_{k-1})^{\frac{1}{1-\gamma}}.
	}
Renaming $2 C$ as $C$ for simplicity,
and iterating these bounds, we obtain
	\eqn{
	\label{prefatt-appendix-30}
  	\eta_k\leq C^{\sum_{l=0}^{k/2}\kappa^l}
	\eta_0^{\kappa^{k/2}}+
      	\sum_{i=1}^{k/2}C^{\sum_{l=0}^{i-1}\kappa^l}\eta_{k-2i+1}^{\kappa^{i-1}}
	  (\log\eta_{k-2i+1})^{\frac{\kappa^{i-1}}{1-\gamma}}.
	}
For the first term in \eqref{prefatt-appendix-30}, we use the precise expression for $\eta_0$ to obtain
	\eqan{
	\label{prefatt-appendix-31}
	C^{\sum_{l=0}^{k/2}\kappa^l}\eta_0^{\kappa^{k/2}} 
	&\leq 
       C^{\sum_{l=0}^{k/2}\kappa^l}\mathrm{exp}\left(2(\log\log t)\kappa^{k/2}\right) \\
	& \grosso \leq  \frac{1}{2}\mathrm{exp}\left(B(\log\log t)\kappa^{k/2}\right),\nn
	}
for a constant $B\geq 2$ large enough. 

For the second term in \eqref{prefatt-appendix-30}, we use the induction hypothesis to obtain
	\eqan{
	\label{prefatt-appendix-32}
	&\sum_{i=1}^{k/2}C^{\sum_{l=0}^{i-1}\kappa^l}\eta_{k-2i+1}^{\kappa^{i-1}}
	  (\log\eta_{k-2i+1})^{\frac{\kappa^{i-1}}{1-\gamma}} \\
  	&\qquad\leq \sum_{i=1}^{k/2}C^{\sum_{l=0}^{i-1}\kappa^l}
	\mathrm{exp}\left(B(\log\log t)\kappa^{(k-1)/2}\right)
	\Big[B(\log\log t)\kappa^{(k-2i+1)/2}\Big]^{\frac{\kappa^{i-1}}{1-\gamma}}.\nn
	}
We can write
	\eqan{
	\label{prefatt-appendix-32b}
	\mathrm{exp}\Big(B(\log\log t)\kappa^{(k-1)/2}\Big) 
	&=\mathrm{exp}\left(B(\log\log t)\kappa^{k/2}\right)
	\mathrm{exp}\left(B(\log\log t) \kappa^{k/2}\big(\sqrt{1/\kappa}-1\big)\right).
	}
Since $\sqrt{1/\kappa}-1 <0$, for $k = O(\log \log t)$
we can take $B$ large enough such that 
	\eqn{
	\label{prefatt-appendix-33}
  	\sum_{i=1}^{k/2}C^{\sum_{l=0}^{i-1}\kappa^l} 
	\mathrm{exp}\left(B(\log\log t) \kappa^{k/2}\big(\sqrt{1/\kappa}-1\big)\right)
	\Big[B(\log\log t)\kappa^{(k-2i+1)/2}\Big]^{\frac{\kappa^{i-1}}{1-\gamma}}<\frac{1}{2}.
	}
We can now sum the bounds in \eqref{prefatt-appendix-31} and \eqref{prefatt-appendix-32}--\eqref{prefatt-appendix-33} to obtain
	\eqn{
	\label{prefatt-appendix-34}
  	\eta_{k}\leq\Big(\frac{1}{2}+\frac{1}{2}\Big)
	\mathrm{exp}\left(B(\log\log t)\kappa^{k/2}\right),
	}
as required. This completes the proof of Lemma \ref{prefatt-appendix-result4}.
\end{proof}

Now we are ready to complete the proof of  Proposition \ref{prefatt-prop-bounddistance}:
\begin{proof}[Proof of Proposition \ref{prefatt-prop-bounddistance}]
Recall the definition of $\bar{k}_t$ in \eqref{prefatt-low-kt}. By \eqref{prefatt-lowdistfix-6},
	\eqn{
	\label{prefatt-appendix-45}
  	\pr(\dist_{\PA{t}}(x,y)\leq \bar{k}_t)\leq \sum_{k=1}^{\bar{k}_t}\sum_{l=1}^{g_k-1} f_{k,t}(x,l)
	+\sum_{k=1}^{\bar{k}_t}\sum_{l=1}^{g_k-1}f_{k,t}(y,l) 
	+\sum_{k=1}^{2\bar{k}_t}\sum_{l=g_{\lfloor k/2\rfloor}}^t f_{\lfloor k/2\rfloor,t}(x,l)
	f_{\lceil k/2\rceil,t}(y,l).
	}
We start with the first two sums, which are equal except that $x$ is replaced by $y$ in the second. We use \eqref{prefatt-appendix-7}, together with the fact that $l\leq g_k-1$, to obtain
	\eqn{
	\sum_{k=1}^{\bar{k}_t}\sum_{l=1}^{g_k-1} f_{k,t}(x,l)
	\leq \sum_{k=1}^{\bar{k}_t}\alpha_k\sum_{l=1}^{g_k-1} l^{-\gamma}.
	} 
Since $\gamma\in(1/2,1)$, there exists a constant $b$ such that
	\eqn{
	\sum_{k=1}^{\bar{k}_t}\alpha_k\sum_{l=1}^{g_k-1} l^{-\gamma}
	\leq b\sum_{k=1}^{\bar{k}_t}\alpha_k(g_k-1)^{1-\gamma}.
	}
We use the definition of $g_k$ to bound
	\eqn{
	 b\sum_{k=1}^{\bar{k}_t}\alpha_k(g_k-1)^{1-\gamma}
	<b\sum_{k=1}^{\bar{k}_t}\frac{6}{\pi^2k^2(\log t)^2}\leq \frac{b}{(\log t)^2},
	}
as required. The term with $y$ replacing $x$ is identical.

We next consider now the third sum in \eqref{prefatt-appendix-45}, we again use the bound in  \eqref{prefatt-appendix-7} as well as the fact that $k\mapsto g_k$ is non-increasing, while $k\mapsto \alpha_k, \beta_k$ are non-decreasing, to obtain
	\eqan{
	\label{prefatt-appendix-100}
	&\sum_{k=1}^{2\bar{k}_t}\sum_{l=g_{\lfloor k/2\rfloor}}^t
	(\alpha_{\lfloor k/2\rfloor}l^{-\gamma}+\beta_{\lfloor k/2\rfloor}l^{\gamma-1})
	 (\alpha_{\lceil k/2\rceil}l^{-\gamma}+\beta_{\lceil k/2\rceil}l^{\gamma-1})\\
	&\qquad\leq \sum_{k=1}^{2\bar{k}_t}\sum_{l=g_{\lceil k/2\rceil}}^t
	(\alpha_{\lceil k/2\rceil}l^{-\gamma}+\beta_{\lceil k/2\rceil}l^{\gamma-1})^2
	\leq 2\sum_{k=1}^{2\bar{k}_t}\sum_{l=g_{\lceil k/2\rceil}}^t
	\big(\alpha^2_{\lceil k/2\rceil}l^{-2\gamma}+\beta^2_{\lceil k/2\rceil}l^{2\gamma-2}\big)\nn\\
	&\qquad= 2\sum_{k=1}^{2\bar{k}_t}\sum_{l=g_{\lceil k/2\rceil}}^t
	\alpha^2_{\lceil k/2\rceil}l^{-2\gamma}
	+2\sum_{k=1}^{2\bar{k}_t}\sum_{l=g_{\lceil k/2\rceil}}^t\beta^2_{\lceil k/2\rceil}l^{2\gamma-2}.\nn
	}
This leads to two terms that we bound one by one. For the first term in \eqref{prefatt-appendix-100}, we can write
	\eqn{
	2\sum_{k=1}^{2\bar{k}_t}\sum_{l=g_{\lceil k/2\rceil}}^t\alpha^2_{\lceil k/2\rceil}l^{-2\gamma}
	\leq
		2b'\sum_{k=1}^{2\bar{k}_t}\alpha^2_{\lceil k/2\rceil}g_{\lceil k/2\rceil}^{1-2\gamma}
		=2b'\sum_{k=1}^{2\bar{k}_t}\alpha^2_{\lceil k/2\rceil}g_{\lceil k/2\rceil}^{2-2\gamma}
		\eta_{\lceil k/2\rceil}\frac{1}{t}.
	}
By definition of $g_k$, 
	\eqn{
	\alpha_{\lceil k/2\rceil}g_{\lceil k/2\rceil}^{1-\gamma}
	\leq 2^{1-\gamma} \alpha_{\lceil k/2\rceil}(g_{\lceil k/2\rceil}-1)^{1-\gamma}
	\leq \frac{6(1-\gamma)}
	{\pi^2 (k/2)^2(\log t)^2}.
	}
Therefore,
	\eqan{
	\label{first-term-app}
	\frac{2b'}{t}\eta_{\bar{k}_t}\sum_{k=1}^{2\bar{k}_t}
	\alpha^2_{\lceil k/2\rceil}g_{\lceil k/2\rceil}^{2-2\gamma}
	&\leq \frac{2b'2^{2(1-\gamma)}}{t}\eta_{\bar{k}_t}\sum_{k=1}^{2\bar{k}_t} 
	\Big(\frac{24(1-\gamma)}
	{\pi^2k^2(\log t)^2}\Big)^2\\ 
	&\leq \frac{C}{t}\eta_{\bar{k}_t}\frac{1}{(\log t)^4} = o\left((\log t)^{-4}\right),\nn
	}
since $\eta_{\bar{k}_t}=o(t)$ by the definition of $\bar{k}_t$ in \eqref{prefatt-low-kt} and 
Lemma \ref{prefatt-appendix-result4}: in fact, $\gamma = m/(2m+\delta)$
and consequently $\kappa = \gamma / (1-\gamma) = (1+\delta/m)^{-1}$,
that is $\kappa = 1/(\tau - 2)$ (recall that $\tau = 3+\delta/m$).

For the second term in \eqref{prefatt-appendix-100}, we use that $2-2\gamma\in(0,1)$ to compute 
	\eqn{
	2\sum_{k=1}^{2\bar{k}_t}\sum_{l=g_{\lceil k/2\rceil}}^t\beta^2_{\lceil k/2\rceil}l^{2\gamma-2}\leq
		2b''\sum_{k=1}^{2\bar{k}_t}\beta^2_{\lceil k/2\rceil}t^{2\gamma-1}.
	}
By definition of $\alpha_k$, we have $\beta_k\leq \alpha_{k+1}t^{1-2\gamma}$, which means that
	\eqn{
	\sum_{r=1}^{\bar{k}_t}\beta_r^2t^{2\gamma-1} 
	\leq \sum_{r=1}^{\bar{k}_t}\alpha^2_{r+1}t^{2-4\gamma}t^{2\gamma-1} 
	= \sum_{r=1}^{\bar{k}_t}\alpha_{r+1}^2t^{1-2\gamma} 
	\leq \frac{1}{t}\eta^{2-2\gamma}_{\bar{k}_t}\sum_{r=1}^{\bar{k}_t}
	\alpha_{r+1}^2g^{2-2\gamma}_{r+1},
	}
which is $o((\log t)^{-4})$ as in \eqref{first-term-app}. We conclude that there exists a constant $p$ such that
	\eqn{
	\pr(\dist_{\PA{t}}(x,y)\leq 2\bar{k}_t)\leq p(\log t)^{-2},
	}
as required. This completes the proof of Proposition \ref{prefatt-prop-bounddistance}.
\end{proof}

\section{Proof of Lemmas~\ref{lem-degoutcore} and~\ref{prefatt-lemma-boundarybound}}
\label{app-B}

\subsection{Proof of Lemma~\ref{lem-degoutcore}}
We adapt the proof of \cite[Lemma A.4]{DSvdH}. We design a Poly\'a's urn experiment to bound 
the probability that a fixed vertex $i\in[t/2]\setminus \core_t$ accumulates too many edges from 
the vertices $t/2+1,\ldots,t$. 

Let us fix $i\in[t/2]$.
We consider one urn containing blue balls and containing red balls. For every edge that we add to the graph from $t/2+1$ to $t$ (which means $mt/2$ edges), we need to keep track of the number of edges attached to $i$. To give an upper bound, we can assume that $D_{t/2}(i) = (\log t)^\sigma$. Let $R_k$ and $B_k$ denote the number of red and blue balls after $k$ draws, so that $R_0=(\log t)^\sigma$ and $B_0=m
(t/2)-(\log t)^\sigma$. Thus, $R_0$ is the maximal degree of vertex $i$ at time $t/2$, while $B_0$ is the minimal degree of all vertices unequal to $i$ at time $t/2$. We consider two linear weight functions for the number of balls in each urn, 
	\eqn{
	W_k^{\sss(r)} = k+\delta,\quad \mbox{ and }\quad W_k^{\sss(b)} = k+\delta (t/2-1).
	}
At time $k\geq 0$, let $R_k$ and $B_k$ denote the number of red and blue balls after $k$ draws. 
Then we draw a ball colored red or blue according to the weights $W_{R_k}^{\sss(r)}$ and $W_{B_k}^{\sss(b)}$. Here $W_{R_k}^{\sss(r)}$ represents the weight of the vertex $i$ and $W_{B_k}^{\sss(b)}$ is the weight of the rest of the graph. Naturally, $R_k+B_k
=m(t/2)+k$ is deterministic, as it should be in a Poly\'a urn, and also $W_{R_k}^{\sss(r)}+W_{B_k}^{\sss(b)}=(m + \delta)t/2+k$ is deterministic.

We consider $mt/2$ draws, and at every one, we pick a red ball with probability proportional to $W_{R_k}^{\sss(r)}$ and a blue ball with probability proportional to $W_{B_k}^{\sss(b)}$, respectively. We add one ball of the same color as the selected color (next to the drawn ball, which we put back). 
Recalling \eqref{prefatt-attprob}, one
could notice that at every draw we should add $1+\delta/m$ to the weight of the blue balls since the weight of the graph (with $i$ excluded) is always increasing due to the fact that the new edge is attached to a new vertex (that is not equal to $i$). However, $1+\delta/m\geq 0$ and thus ignoring this effect only increases the probability of choosing $i$. This suffices for our purposes, since we are only interested in an upper bound.

We denote by $(X_n)_{n=1}^{mt/2}$ a sequence of random variables, where $X_n = 1$ whenever the $n$-th extraction is a red ball (a new edge is attached to $i$). As a consequence,
	\eqn{
	\pr\Big(D_t(i)\geq (1+B)(\log t)^\sigma \mid D_{t/2}(i)<(\log t)^\sigma\Big) \leq
		\pr\Big(\sum_{n=1}^{mt/2}X_n\geq B(\log t)^\sigma\Big).
	}
As the reader can check, the sequence of random variables $(X_n)_{n=1}^{mt/2}$ is exchangeable, so we can apply De Finetti's Theorem, and obtain that
	\eqn{
	\label{eq::poly1}
	\pr\Big(\sum_{n=1}^{mt/2}X_n\geq B(\log t)^\sigma\Big) = 
		\E\left[\pr\left(\mathrm{Bin}(mt/2,U)\geq 
		B(\log t)^\sigma\mid U\right)\right],
	}
where $U$ is a distribution on $[0,1]$. In the case of the Poly\'a urn with two colors and linear weights, as shown in \cite[Theorem 6.2]{vdH2}, $U$ has a Beta distribution with parameters given by
	\eqn{
	\alpha_t = (\log t)^\sigma + \delta,\quad\mbox{ and }\quad 
	\beta_t = \frac{t}{2}\left(2m+\delta\right)-\left((\log t)^\sigma + \delta \right).
	}
We call
	\eqn{
	\psi(u) = \pr\left(\mathrm{Bin}(mt/2,u)\geq B(\log t)^\sigma\right)\leq 1.
	}
By the classical Chernoff bound $\pr(\mathrm{Bin}(n,u) \ge k) \le
e^{-n I_u(k/n)}$, where $I_u(a) = a(\log (a/u) - 1) + u$ is the large deviation function
of a $\mathrm{Pois}(u)$ random variable, see \cite[Corollary 2.20]{vdH1}.
For $a \ge 8 u$ we can bound $I_u(a) \ge a(\log 8 - 1) \ge a$, hence
	\eqn{
	\label{eq::poly2}
	\psi(u)\leq \e^{-B(\log t)^\sigma}\mbox{ whenever } 4mtu\leq B(\log t)^\sigma.
	}
Define
	\eqn{
	g(t) = \frac{B(\log t)^\sigma}{4mt}.
	}
Using $g$ and $\psi$ in \eqref{eq::poly1}, we have
	\eqn{
	\label{eq::poly3}
	\E\left[\pr\left(\mathrm{Bin}(mt/2,U)\geq 
	B(\log t)^\sigma\mid U\right)\right] \leq
	\e^{-B(\log t)^\sigma} +\pr\left(U>g(t)\right).
	}
Note that $\e^{-B(\log t)^\sigma} = o(1/t)$
since $\sigma>1$ and for $B$ large enough. 
What remains is to show that also the second term in the right-hand side 
of \eqref{eq::poly3} is $o(1/t)$. Since $U$ has a Beta distribution, 
its density $f_U(u) = \frac{\Gamma(\alpha_t+\beta_t)}{\Gamma(\alpha_t)\Gamma(\beta_t)}
u^{\alpha_t-1} (1-u)^{\beta_t-1}$ attains its maximum at the point
$\bar u_t = \frac{\alpha_t-1}{\alpha_t+\beta_t-2}$.
Since
$g(t) \geq \bar u_t$ (which can easily be checked, for $B$ large enough), we have
	\eqn{
	\label{eq::poly8}
	\pr\left(U>g(t)\right) \leq \frac{\Gamma(\alpha_t+\beta_t)}{\Gamma(\alpha_t)\Gamma(\beta_t)}\left(1-g(t)\right)^{\beta_t} g(t)^{\alpha_t-1}.
	}
We next bound each of these terms separately. Firstly, asymptotically as $t\rightarrow\infty$ and since $\delta<0$,
	\eqn{
	\label{eq::poly4}
	\frac{\Gamma(\alpha_t+\beta_t)}{\Gamma(\beta_t)}\leq (\alpha_t+\beta_t)^{\alpha_t}=\left(mt(1+\delta/2)\right)^{\alpha_t}
	\leq (mt)^{\alpha_t}.
	}
Secondly,
	\eqn{
	\label{eq::poly5}
	g(t)^{\alpha_t}=\left(\frac{B(\log t)^\sigma}{4mt}\right)^{\alpha_t}=(B/4)^{\alpha_t} \Big(\frac{(\log t)^\sigma}{mt}\Big)^{\alpha_t}.
	}
Thirdly, since $1-x\leq \e^{-x}$,
	\eqn{
	\label{eq::poly6}
	(1-g(t))^{\beta_t} = \left(1-\frac{B(\log t)^\sigma}{4mt}\right)^{\beta_t}\leq \mathrm{exp}\left(-c B(\log t)^\sigma\right),
	}
for some $c \in (0,\infty)$. Finally, by Stirling's formula $\Gamma(\alpha_t)\geq (\alpha_t/\e)^{\alpha_t}$, so that
	\eqn{
	\label{eq::poly7}
	\Gamma(\alpha_t)^{-1}\leq (\e/\alpha_t)^{\alpha_t}.
	}
Substituting \eqref{eq::poly4}, \eqref{eq::poly5}, \eqref{eq::poly6} and \eqref{eq::poly7} into \eqref{eq::poly8}, we obtain
	\eqn{
	\pr\left(U>g(t)\right)\leq \frac{\mathrm{exp}
	\left(-cB(\log t)^\sigma\right)}{g(t)}
	\Big(\frac{B \e (\log t)^\sigma}{4\alpha_t}\Big)^{\alpha_t}
	\leq t \, \mathrm{exp}\left(-cB(\log t)^\sigma/2\right)= o(1/t),
	}
for $B$ large enough and using that $(\log t)^\sigma/\alpha_t\leq 1-\delta/\alpha_t=1+O(1/\alpha_t)$. This completes the proof of the lemma.
\qed

\subsection{Proof of Lemma~\ref{prefatt-lemma-boundarybound}}
For $i=0,\ldots,k$, we denote by $N_i$ the number of vertices in the $k$-exploration graph 
at distance $i$ from $v$; the set of such vertices will be called ``level $i$''.
Clearly, $N_0=1$ because the only vertex at level $0$ is $v$. Plainly, if there are no collisions
between level $i-1$ and level $i$, then $N_i = m N_{i-1}$. The number of collisions $l_i$
between level $i-1$ and level $i$ is then given by
\begin{equation} \label{eq:elli}
	l_i := m N_{i-1} - N_i \,,
\end{equation}
and the total number of collisions is, by assumtion,
\begin{equation} \label{eq:lsum}
	\bar l_k := l_1 + \ldots + l_k \le l \,.
\end{equation}
\emph{The assumption that no vertex has only self-loops implies that $N_i \ge 1$ for every $i$}, 
because it ensures that the youngest vertex at level $i-1$ has at least one ``descendant'' 
at level $i$. Therefore
\begin{equation} \label{eq:elliconst}
	l_i \le m N_{i-1} - 1 \,.
\end{equation}

For later purposes, it is convenient to start with $N_0 \ge 1$ vertices
(even though we are eventually interested in the case $N_0 = 1$).
Rewriting \eqref{eq:elli} as $N_i = m N_{i-1} - l_i$, a simple iteration yields
\begin{equation} \label{eq:fili}
\begin{split}
	N_k & = m^k N_0 - m^{k-1} l_1 - m^{k-2} l_2 + \ldots - m l_{k-1} - l_k \,.
\end{split}
\end{equation}
This yields $N_k \ge m^k N_0 - m^{k-1} (l_1 + l_2 + \ldots + l_k)$, that is
\begin{equation} \label{eq:lb1}
\begin{split}
	N_k \ge (m N_0 - \bar l_k) m^{k-1} \,.
\end{split}
\end{equation}
This lower bound is only useful if $\bar l_k < m N_0$, otherwise the
right hand side is negative.
To deal with the complementary case, we now show by induction the following
useful bound:
\begin{equation}\label{eq:lb2}
	N_k \ge m^{-\frac{\bar l_k - (m N_0 - 1)}{m-1}} \, m^{k-2} 
	\qquad \text{when} \quad \bar l_k \ge m N_0 -1 \,.
\end{equation}

The case $k=1$ is easy:
since $\bar l_1 = l_1 \le m N_0 - 1$ by \eqref{eq:elliconst},
we only need to consider the extreme case $\bar l_1 = mN_0 - 1$,
when \eqref{eq:lb2} reduces to $N_1 \ge \frac{1}{m}$, which holds since
$N_1 \ge 1$.
Next we fix $k \ge 1$ and
our goal is to prove \eqref{eq:lb2} for $k+1$, assuming that it holds for $k$.

The crucial observation is that, iterating relation $N_i = m N_{i-1} - l_i$
(that is \eqref{eq:elli}) from $i=k+1$ until $i=2$,
we get an analogue of \eqref{eq:fili}, that is
\begin{equation} \label{eq:fili2}
\begin{split}
	N_{k+1} & = m^k N_1 - m^{k-1} l_2 - m^{k-2} l_3 + \ldots - m l_{k} - l_{k+1} \,.
\end{split}
\end{equation}
This means that \emph{$N_{k+1}$ coincides with $N_k$ in \eqref{eq:fili}
where we replace $N_0$ by $N_1$
and $(l_1, \ldots, l_k)$ by $(l_2, \ldots, l_{k+1})$ (therefore
$\bar l_k$ is replaced by $\bar l_{k+1} - l_1$)}.
As a consequence, by the inductive assumption, we can apply \eqref{eq:lb2} which yields
\begin{equation}\label{eq:lb2bis}
	N_{k+1} \ge m^{-\frac{(\bar l_{k+1} - \bar l_1) - (m N_1 - 1)}{m-1}} \, m^{k-2} 
	\qquad \text{when} \quad (\bar l_{k+1} - \bar l_1) \ge m N_1 - 1 \,.
\end{equation}
For later use, we note that, analogously, relation \eqref{eq:lb1} gives
\begin{equation} \label{eq:lb1bis}
\begin{split}
	N_{k+1} \ge (m N_1 - (\bar l_{k+1}-l_1)) m^{k-1} \,.
\end{split}
\end{equation}
By \eqref{eq:elli} and \eqref{eq:elliconst},
which yield $N_1 = m N_0 - l_1$ and $l_1 \le m N_0 - 1$, we can write
\begin{equation*}
	(\bar{l}_{k+1}-l_1)-(mN_1-1) =
	\bar{l}_{k+1} -m^2N_0
	+(m-1)l_1+1 \leq 
	\bar{l}_{k+1}-(mN_0-1) - (m-1),
\end{equation*}
which plugged into \eqref{eq:lb2bis} gives
\begin{equation}\label{eq:lb2bis2}
	N_{k+1} \ge m^{-\frac{\bar{l}_{k+1} -(mN_0-1)}{m-1}} \, m^{k -1} 
	\qquad \text{when} \quad (\bar l_{k+1} - \bar l_1) \ge m N_1 - 1 \,.
\end{equation}
This is precisely the analogue of \eqref{eq:lb2} for $k+1$, which is our goal,
except for the ``wrong'' restriction $(\bar l_{k+1} - \bar l_1) \ge m N_1$,
instead of $\bar l_{k+1} \ge m N_0$. We are thus left with showing that the
inequality in \eqref{eq:lb2bis2} still holds if
$(\bar l_{k+1} - \bar l_1) < m N_1 - 1$ and $\bar l_{k+1} \ge m N_0 - 1$,
but this is easy, because these two
conditions, together with \eqref{eq:lb1bis}, imply
	\begin{equation*}
	N_{k+1} \ge (m N_1 - (\bar l_{k+1}-l_1)) m^{k-1}
	\ge m^{k-1} \ge m^{-\frac{\bar{l}_{k+1} -(mN_0-1)}{m-1}} \, m^{k -1}  \,.
	\end{equation*}

We are ready to conclude. 
The bounds \eqref{eq:lb1} and \eqref{eq:lb2}, for $N_0 = 1$ and $l_k \le l$, yield
\begin{equation}\label{eq:finalboundlb}
	N_k \ge m^{-\frac{l}{m-1}} \, m^{k-1}
	= s(m,l) \, m^k \,, \qquad \text{where} \qquad s(m,l) := m^{-1-\frac{l}{m-1}} \,,
\end{equation}
which is what we wanted to prove.
\qed}

\section*{Acknowledgements}

We thank the referees for their detailed comments, 
which improved the paper considerably.

The work of FC was partially supported by the ERC Advanced Grant 267356 VARIS
and by the PRIN 20155PAWZB ``Large Scale Random Structures''.
The work of AG was partially supported by University of Milano Bicocca through an EXTRA scholarship that sponsored his visit to Eindhoven University of Technology
to complete his master project in March-May 2014. The work of AG and RvdH is supported by the Netherlands Organisation for Scientific Research (NWO) through 
the Gravitation {\sc Networks} grant 024.002.003.  The work of RvdH is also supported by NWO through VICI grant 639.033.806.



\end{document}